\documentclass{article} 
\usepackage{amsmath,amssymb,amsthm} 
\usepackage[OT2,T1]{fontenc}
\usepackage[utf8]{inputenc}
\usepackage{enumitem}
\usepackage{graphicx}
\usepackage{geometry}
\usepackage{hyperref}
\usepackage{tikz}
\usetikzlibrary{cd}
\usepackage{tikz-cd}
\usepackage{xcolor}
\usepackage{ stmaryrd }
\usepackage{mathrsfs}
\usepackage{titling}
\usepackage{fullpage}
\usepackage{mathtools}
\usepackage[english]{babel}
\usepackage{setspace}
\usepackage[normalem]{ulem}
\usepackage{pdflscape}
\usepackage{afterpage}
\usepackage{chngpage}
\usepackage{makecell}

\usepackage[alphabetic,lite]{amsrefs}

\usepackage{arydshln}

\usetikzlibrary{arrows,positioning,decorations.pathmorphing,
	decorations.markings
}
\tikzset{inner sep=0pt,
	root/.style={circle,draw,minimum size=7pt,thick},
	fatroot/.style={circle,draw,minimum size=10pt,thick},
	short root/.style={circle,fill,minimum size=7pt},
	doublearrow/.style={postaction={decorate},
		decoration={markings,mark=at position .7
			with {\arrow{angle 60}}},double distance=3pt,thick}
}

\DeclareUnicodeCharacter{00A0}{ }

\newtheorem{proposition}{Proposition}[section]

\newtheorem{theorem}[proposition]{Theorem}
\newtheorem{lemma}[proposition]{Lemma}
\newtheorem{corollary}[proposition]{Corollary}

\numberwithin{equation}{section}

\theoremstyle{definition}
\newtheorem{definition}[proposition]{Definition}
\newtheorem{remark}[proposition]{Remark}
\newtheorem{example}[proposition]{Example}

\newcommand{\G}{\mathbb{G}}
\DeclareMathOperator{\GL}{GL}
\DeclareMathOperator{\PGL}{PGL}
\DeclareMathOperator{\SL}{SL}

\DeclareMathOperator{\Sp}{Sp}
\DeclareMathOperator{\PSp}{PSp}

\DeclareMathOperator{\SO}{SO}
\DeclareMathOperator{\PSO}{PSO}
\DeclareMathOperator{\Spin}{Spin}

\DeclareMathOperator{\liesl}{\mathfrak{sl}}

\DeclareMathOperator{\lieg}{\mathfrak{g}}
\DeclareMathOperator{\lieh}{\mathfrak{h}}
\DeclareMathOperator{\liet}{\mathfrak{t}}
\DeclareMathOperator{\Ad}{Ad}
\DeclareMathOperator{\Lie}{Lie}

\DeclareMathOperator{\Id}{Id}

\DeclareMathOperator{\Hom}{Hom}
\DeclareMathOperator{\Aut}{Aut}
\DeclareMathOperator{\Out}{Out}

\DeclareMathOperator{\Gal}{Gal}

\DeclareMathOperator{\Sym}{Sym}

\DeclareMathOperator{\rank}{rank}

\DeclareMathOperator{\Stab}{Stab}

\DeclareMathOperator{\Spec}{Spec}

\DeclareMathOperator{\Pic}{Pic}

\newcommand{\A}{\mathbb{A}}
\renewcommand{\P}{\mathbb{P}}

\renewcommand{\O}{\mathcal{O}}
\newcommand{\GIT}{\mathbin{/\mkern-6mu/}}
\newcommand{\HH}{\mathrm{H}}

\newcommand{\Q}{\mathbb{Q}}

\newcommand{\Z}{\mathbb{Z}}

\DeclareMathOperator{\Sel}{Sel}

\DeclareSymbolFont{cyrletters}{OT2}{wncyr}{m}{n}
\DeclareMathSymbol{\Sha}{\mathalpha}{cyrletters}{"58}

\newcommand{\extp}{\@ifnextchar^\@extp{\@extp^{\,}}}
\def\@extp^#1{\mathop{\bigwedge\nolimits^{\!#1}}}
\newcommand{\overbar}[1]{\mkern 1.5mu\overline{\mkern-1.5mu#1\mkern-1.5mu}\mkern 1.5mu}
\DeclareUnicodeCharacter{00A0}{ }
\newcommand{\calO}{\mathcal{O}}
\newcommand{\height}{\mathrm{ht}}

\tikzset{
  symbol/.style={
    draw=none,
    every to/.append style={
      edge node={node [sloped, allow upside down, auto=false]{$#1$}}}
  }
}

\newcommand{\frakg}{\mathfrak g}
\newcommand{\frakh}{\mathfrak h}
\newcommand{\frakz}{\mathfrak z}

\newcommand{\frakt}{\mathfrak t}

\newcommand{\fraksl}{\mathfrak{sl}}
\newcommand{\al}{\alpha}
\newcommand{\be}{\beta}
\newcommand{\bZ}{\mathbb Z}

\newcommand{\bG}{\mathbb G}

\newcommand{\ra}{\rangle}
\DeclareMathOperator{\ad}{ad}
\DeclareMathOperator{\spn}{span}

\newcommand{\rl}{r} 

\usepackage{listings}

\definecolor{codegreen}{rgb}{0,0.6,0}
\definecolor{codegray}{rgb}{0.5,0.5,0.5}
\definecolor{codepurple}{rgb}{0.58,0,0.82}
\definecolor{backcolour}{rgb}{0.97,0.97,0.95}

\lstdefinestyle{mystyle}{
    backgroundcolor=\color{backcolour},   
    commentstyle=\color{codegreen},
    keywordstyle=\color{magenta},
    numberstyle=\tiny\color{codegray},
    stringstyle=\color{codepurple},
    basicstyle=\ttfamily\footnotesize,
    breakatwhitespace=false,         
    breaklines=true,                 
    captionpos=b,                    
    keepspaces=true,                 
    numbers=left,                    
    numbersep=5pt,                  
    showspaces=false,                
    showstringspaces=false,
    showtabs=false,                  
    tabsize=2
}

\usepackage{microtype}

\title{Families of curves in Vinberg representations}
\author{Jef Laga and Beth Romano}

\begin{document}

\maketitle

\begin{abstract}
Inspired by orbit parametrizations in arithmetic statistics, we explain how to construct families of curves associated to certain nilpotent elements in $\Z/m\Z$-graded Lie algebras, generalizing work of Thorne to the $m\geq 3$ case and the non-simply laced case.
We classify such families arising from subregular nilpotents in stable gradings and interpret almost all orbit parametrizations associated with algebraic curves appearing in the literature in this framework.
As an extended example, we give a ``Lie-theoretic'' proof of the integral orbit parametrization of $5$-Selmer elements of elliptic curves over $\Q$, using a $\Z/5\Z$-grading on a Lie algebra of type $E_8$.
\end{abstract}

\tableofcontents

\section{Introduction}

\subsection{Context}

Let $\mathcal{F}$ be a family of (smooth, projective, geometrically connected) curves over $\Q$.
Arithmetic statistics studies the behaviour of arithmetic quantities associated with $C$ as $C$ varies in $\mathcal{F}$, such as: the set of rational points $C(\Q)$ of $C$, the rational points $J(\Q)$ of the Jacobian variety $J$ of $C$, and the $n$-Selmer group $\Sel_n J$ of $J$, a cohomological avatar of the finitely generated abelian group $J(\Q)$.
The last twenty years have witnessed spectacular advances in the arithmetic statistics of curves, see for example \cite{BS-2selmerellcurves,BS-3Selmer, BS-4Selmer, BS-5Selmer, Bhargava-Gross-hyperellcurves, BhargavaGrossWang-positiveproportionnopoints, PoonenStoll-Mosthyperellipticnorational}.
One of the key ideas of Bhargava and his collaborators is that many arithmetic objects admit \emph{orbit parametrizations}.
In this context, this typically means that there exists a representation $V$ of an algebraic group $G/\Q$, an integer $n\geq 2$ and for each $C\in \mathcal{F}$ an injection 
\begin{align}\label{equation intro: embedding n-selmer}
    \Sel_n J \hookrightarrow G(\Q) \backslash V(\Q)
\end{align}
into the rational orbits of $(G,V)$ with specified invariants depending on $C$.
Combining such an orbit parametrization with geometry-of-numbers techniques of counting (integral) orbits often leads to a determination of (or a bound on) the average size of $\#\Sel_nJ$ as $C$ varies in $\mathcal{F}$ and to the results cited above.

It is therefore natural to ask how one finds such parametrizations.
Typically, they arise from classical algebro-geometric constructions.
For example, if $E/\Q$ is an elliptic curve, there exists an injection of $\Sel_2 E$ into the set of $\PGL_2(\Q)$-orbits of binary quartic forms; this is essentially due to Mordell and Birch--Swinnerton-Dyer \cite{BirchSwinnertonDyer-notesonellcurves1}.
Similar constructions exist for $\Sel_n E$ if $n\in \{3,4,5\}$ and for other families of curves \cite{BhargavaHo-coregularspacesgenusone}, but finding such parametrizations is still a relatively ad hoc process.

Graded Lie algebras provide a systematic way to construct orbit parametrizations which have proven useful in arithmetic statistics.
Given a reductive algebraic group $H$ over a field $k$ of characteristic zero with Lie algebra $\frakh$ and a $\mu_m$-action $\theta$ on $H$, the subspaces $\frak{h}_i = \{x\in \frakh \colon \theta(\zeta)(x) = \zeta^i x\text{ for all }\zeta\in \mu_m\}$ define an $m$-grading of $\frakh$, in the sense that
\begin{align*}
    \lieh = \bigoplus_{i \in \Z/m\Z} \lieh_i
\end{align*}
and that $[\lieh_i,\lieh_j] \subset\lieh_{i+j}$ for all $i,j\in \Z/m\Z$.
If $G$ denotes the identity component of the fixed-point subgroup $H^{\theta}$, then $G$ acts on $V = \frakh_1$ via restriction of the adjoint representation, and the pair $(G,V)$ is called a \emph{Vinberg representation} \cite{Vinberg-theweylgroupofgraded}.
Vinberg representations have a well-understood invariant theory and their orbits can be studied using the structure theory of the ambient Lie algebra $\frakh$. 
For example, Vinberg representations are always coregular, meaning that the subring $k[V]^G$ of $G$-invariant polynomials on $V$ is itself a polynomial ring in homogeneous generators $p_{d_1}, \dots, p_{d_{\ell}}$.

Gross \cite{Gross-BhargavasrepresentationsandVinberg} observed that almost all\footnote{The only exception known to us is the action of $\SL_n$ on pairs of symmetric $n\times n$ matrices studied in \cite{BhargavaGrossWang-positiveproportionnopoints}.} representations studied by Bhargava and his collaborators turn out to be Vinberg representations.
This suggests the possibility of taking Vinberg theory as a starting point and constructing the families of curves and orbit parametrizations using Lie theory.
This is exactly the perspective taken in Thorne's PhD thesis \cite{Thorne-thesis}, which analyzed the case where $H$ is simple of type $A,D,E$ and $\theta$ is a so-called stable $\mu_2$-action. 
For example, the $A_2$ case recovers the orbit parametrization of $\Sel_2E$ by binary quartic forms alluded to above.
Thorne's framework has unified and reproved many existing results concerning $2$-Selmer groups in families of curves and has led to new results for families of nonhyperelliptic curves; see \cite{Thorne-E6paper, Romano-Thorne-ArithmeticofsingularitiestypeE, Laga-E6paper, Laga-ADEpaper}.

\subsection{Results}
The purpose of this paper is to extend Thorne's framework to all simple groups $H$ and to all $m \geq 2$.
Since the landscape of stable gradings for $m\geq 3$ is more erratic than the $m=2$ case, our results are less complete and uniform than those of \cite{Thorne-thesis}. 
On the other hand, we interpret most of the known orbit parametrizations in our framework and pave the way to study them using Vinberg theory.
As an example, we analyze a $\mu_5$-grading on $E_8$ to give a ``Lie-theoretic'' proof of the orbit parametrization for the $5$-Selmer group of an elliptic curve used by Bhargava and Shankar \cite{BS-5Selmer} in their determination of the average size of the $5$-Selmer group.
This parametrization has been studied in a series of papers by Fisher \cite{Fisher-invariantsgenusone, Fisher-invarianttheorynormalquinticI, Fisher-invarianttheoryquinticII, Fisher-minimisationreduction5coverings} using completely different methods.

To explain our results in more detail, we need to introduce some terminology.
We say a Vinberg representation $(G,V)$ associated to a pair $(H, \theta)$ is \emph{stable} if $V$ contains stable vectors in the sense of geometric invariant theory.
This is a natural class of Vinberg representations for two reasons.
Firstly, their stable orbits over algebraically closed fields are separated by their invariants, which are well understood.
In addition, if $v\in V(k)$ is stable, then the $G(k)$-orbits in $G(\bar{k})\cdot v \cap V(k)$ (in other words, the passage from geometric to arithmetic orbits) can be understood using the Galois cohomology of the stabilizer of $v$, see \cite[Section 2, Proposition 1]{BhargavaGross-AIT}.
Since this stabilizer is a finite group scheme by assumption, it can often be related to torsion subgroup schemes of abelian varieties and hence the orbits can be related to Selmer groups.

Slodowy \cite{Slodowy-simplesingularitiesalggroups} has given a recipe to construct interesting subvarieties of $\frakh$ using nilpotent elements. 
We adapt this construction to the graded setting in the following way.
Let $e\in \frakh_1$ be a nilpotent element, and let $(e,h,f)$ be an $\fraksl_2$-triple with $h\in \frakh_0$ and $f\in \frakh_{-1}$.
Define the affine linear subspace 
\begin{align}\label{equation: graded slodowy slice intro}
X_e = (e+\frakz_{\frakh}(f)) \cap \frakh_1\subset V.
\end{align}
The restriction of the GIT quotient $\pi\colon V\rightarrow B:= V \GIT G = \Spec(k[V]^G)$ to $X_e$ defines a flat morphism $\varphi\colon X_e\rightarrow B$.
The general idea is now that, for well chosen $e$ and $(G,V)$, $\varphi\colon X_e\rightarrow B$ is a family of curves, smooth above an open subset $B^s\subset B$, and if $b\in B^s(k)$ then the $G(k)$-orbits of $V(k)$ with invariants $b$ should be closely related to the arithmetic of the curve $\varphi^{-1}(b) = X_{e,b}$, in the  
sense that an embedding like \eqref{equation intro: embedding n-selmer} exists, where $C$ is a compactification of $X_{e,b}$.

In this paper, we pursue this general idea for arbitrary $m\geq 2$, simple $H$, and stable $(G,V)$.
This extends Thorne's work \cite{Thorne-thesis}, who additionally assumed $m=2$ and $H$ is simply laced.
We say a stable $\mu_m$-action $(H, \theta)$ is \emph{subregular adapted} if there exists a nonzero subregular nilpotent element $e\in \frakh_1$ such that the morphism $X_e\rightarrow B$ of \eqref{equation: graded slodowy slice intro} has relative dimension $1$.
In that case, we prove that $X_e$ is a closed subscheme of $\A^2_B$ defined by a single polynomial $F \in k[B][x,y]$.
Our first main result is a classification of all subregular-adapted $(H,\theta)$ and an identification of the subregular curves $X_e\rightarrow B$.

\begin{theorem}\label{thm-intro}
    Let $k$ be an algebraically closed field of characteristic zero.
    Let $H/k$ be a simple adjoint algebraic group, $\theta\colon \mu_m\rightarrow \Aut_H$ a subregular-adapted $\mu_m$-action, and $e\in \frakh_1$ a subregular nilpotent such that the morphism $X_e  \rightarrow B$ has relative dimension $1$.
    Then $(m,H,X_e\rightarrow B)$ appears in Table \ref{table:examples intro}.
    Moreover, for each row in this table, there exists a triple $(H,\theta,e)$ over $k$ such that $X_e\rightarrow B$ is isomorphic to the family of curves corresponding to that row.
\end{theorem}

\begin{remark}[Notation of Table \ref{table:examples intro}]    
    The second column denotes the Dynkin type of $H$.
    When a superscript is present, this indicates that $\theta$ is not inner; the superscript gives the order of the outer type of $\theta$ (see \S\ref{subsec: stable gradings over algebraically closed field}). For example $^2E_6$ indicates that the image of $\theta\colon \mu_m\rightarrow \Aut_H \rightarrow \Out_H$ has order $2$.
    The third column denotes an algebraic group $\tilde{G}$ such that there exists a central isogeny $\tilde{G}\rightarrow G$ with finite kernel. 
    The fourth column describes the $\tilde{G}$-representation $V$.
    If $\tilde{G}$ is naturally defined as a subgroup of $\GL_n$, we write the defining representation of $\tilde{G}$ as $(n)$. 
    The notations $\Sym^2_0(n+1)\subset \Sym^2(n+1)$, $\wedge^4_0(8)\subset \wedge^4(8)$ and $\wedge^3_0(6)\subset \wedge^3(6)$ indicate subrepresentations of maximal dimension.
    Since $G$ acts faithfully on $V$ (Corollary \ref{corollary:Gaction_faithful}), $G$ is isomorphic to the image of $\tilde{G}\rightarrow \GL(V)$ and thus can be recovered from the table. The fifth column gives the centralizer in $G$ of a stable vector $v \in V$.
    The sixth columm displays an equation that defines $X$ as a closed subscheme of $\A^2_B$, for some choice of coordinates $p_{d_1}, \dots, p_{d_r}$ of the affine space $B$. 
    When multiple equations are given, there are different $e$ that give rise to different families of curves. 
    In the case when $m = 2$, $H$ is type $D_n$, and $n$ is even, there are two invariants of degree $n$, denoted by $p_n$ and $p_{n}'$.
    The final column indicates a reference that discusses this orbit parametrization (not necessarily from the perspective of Vinberg theory).
\end{remark}


\newgeometry{margin=.5cm}
\afterpage{%
\begin{landscape}

\thispagestyle{empty}

\renewcommand{\arraystretch}{1.3} 

\begin{table} 

\begin{center}
	\begin{tabular}{|c|c|c|c|c|l|c|c|}
	\hline
	Order $m$ & $H$ & $\tilde{G}$   & $V$ & $Z_G(v)$ & Curves $X_e\rightarrow B$  & References \\
	\hline
	$2$ & $^2A_\rl\: (\rl\geq 2)$ & $\PSO_{\rl+1}$ & $\Sym^2_0(\rl+1)$ & $\mu_2^{2\lfloor\rl/2\rfloor}$ & $y^2 = x^{\rl+1} + p_2x^{\rl-1} + \cdots + p_{\rl+1}$ & \cite{Bhargava-Gross-hyperellcurves, ShankarWang-hypermarkednonweierstrass} \\	
	$2$ & $B_\rl\: (\rl\geq 2)$ & $\SO_{\rl + 1} \times \SO_\rl$ & $\Hom((\rl + 1), (\rl))$ & $\mu_2^{r-1}$ & $y^2 = x^{2\rl} + p_2x^{2\rl-2} + \dots +p_{2\rl-2}x^2 + p_{2\rl}$  &  \\	
    $2$ & $C_\rl\: (\rl\geq 3)$ & $\GL_\rl$ & $\Sym^2(\rl) \oplus \Sym^2((\rl)^{\vee})$ & $\mu_2^{r-1}$ & $xy^2 =  x^{\rl}+p_2x^{\rl-1}+p_4x^{\rl-2} + \dots + p_{2\rl}$  &  \\
     &  &  &  & & $y^2=  x^{\rl}+p_2x^{\rl-1}+p_4x^{\rl-2} + \dots + p_{2\rl}$  &  \\
    $2$ & $D_\rl$, $\rl \geq 4$ even & $\SO_\rl\times \SO_\rl$ & $(\rl)\boxtimes (\rl)$ & $\mu_2^{r-2}$ & $y(xy+p_{\rl}') = x^{\rl-1}+p_2x^{\rl-2}+p_4x^{\rl-3} + \dots + p_{2\rl-2}$  & \cite{Laga-ADEpaper, Thorne-averagesizeelliptictwomarkedfunctionfields}\\
     $2$ & $^{2}D_\rl$, $\rl \geq 5$ odd & $\SO_\rl\times \SO_\rl$ & $(\rl)\boxtimes (\rl)$ & $\mu_2^{r-1}$ & $y(xy+p_{\rl}) = x^{\rl-1}+p_2x^{\rl-2}+p_4x^{\rl-3} + \dots + p_{2\rl-2}$  & \cite{Shankar-2selmerhypermarkedpoints}\\
    $2$ & $^2E_6$ & $\mathrm{PSp}_8$ & $\wedge^4_0(8)$ & $\mu_2^6$ & $y^3 = x^4+(p_2x^2+p_5x+p_8)y + (p_6x^2+p_9x+p_{12})$  & \cite{Thorne-E6paper, Laga-E6paper} \\
    $2$ & $E_7$ & $\SL_8$ & $\wedge^4(8)$ & $\mu_2^6$ & $y^3 = x^3y +p_{10}x^2 +x(p_2y^2 + p_8y + p_{14}) + p_6y^2 + p_{12}y + p_{18}$  &\cite{Romano-Thorne-ArithmeticofsingularitiestypeE} \\
    $2$ & $E_8$ & $\Spin_{16}$ & half spin & $\mu_2^8$ & $y^3 = x^5+ (p_2x^3+p_8x^2+p_{14}x+p_{20})y  +(p_{12}x^3+p_{18}x^2+p_{24}x+p_{30}) $   &\cite{Romano-Thorne-ArithmeticofsingularitiestypeE} \\
    $2$ & $F_4$ & $\Sp_6\times \SL_2$ & $\wedge^3_0(6)\boxtimes (2)$ & $\mu_2^4$ & $y^3 = x^4+(p_2x^2+p_8)y + (p_6x^2+p_{12})$  & \cite{Laga-F4paper} \\ 
    &  & &  & & $y^2 = x^3 + p_8x + p_{12}$  & \cite{Laga-F4paper} \\
    $2$ & $G_2$ & $\SL_2\times \SL_2$ & $(2)\boxtimes \Sym^3(2)$ &  $\mu_2^2$ & $y^2x = x^3 + p_2x^2+p_6$ & \cite{BhargavaHo-2Selmergroupsofsomefamilies} \\ 
     &  &  &  & &  $y^2 = x^3 + p_2x^2+p_6$  &  \\
    \hline
	$3$ &  $^3D_4$ & $\PGL_3$ & $\Sym^3(3)$ &  $\mu_3^2$ & $y^2 = x^3 + p_4 x + p_6$  & \cite{BS-3Selmer}\\
    $3$ & $E_6$ & $\SL_3^3 $ & $(3)\boxtimes(3)\boxtimes (3)$  & $\mu_3^2$ & $y^2 = x^4 + p_{6} x^2 +p_9 x +p_{12}$   & \cite{BhargavaHo-coregularspacesgenusone}\\
    $3$ & $E_8$ & $\SL_9/\mu_3 $ & $\bigwedge^3(9)$  & $\mu_3^4$ & $y^2 =x^5+p_{12}x^3+p_{18}x^2+p_{24}x + p_{30}$  & \cite{Thorne-Romano-E8} \\
    $3$ & $F_4$ & $\SL_3 \times \SL_3$ & $(3)\boxtimes \Sym^2(3)$ & $\mu_3^2$ & $y^2 = x^4 + p_{6} x^2 + p_{12}$   & \cite{BhargavaHo-coregularspacesgenusone}\\
    $3$ & $G_2$ & $\GL_2 $ & $(\Sym^3(2)\otimes \text{det}^{-2}) \oplus \text{det}$ & $\mu_3$ & $y^2 = x^3 + p_{6}$  & \cite{BhargavaElkiesShnidman}\\
	\hline
    $4$ & $^2E_6$ & $\SL_2\times \SL_4$ & $(2)\boxtimes \Sym^2(4)$ & $\mu_4^2$ & $y^2 = x^3 + p_{8}x+p_{12}$ &  \cite{BS-4Selmer} \\
    $4$ & $F_4$ & $\SL_2\times \SL_3$ & $((2) \boxtimes \Sym^2(3)) \oplus ((2) \boxtimes 1)$ & $\mu_2^2$ & $y^2 = x^3+p_8x+p_{12}$ &  \\ 
    \hline
    $5$ & $E_8$ & $\SL_5\times \SL_5$ & $(5)\boxtimes \wedge^2(5)$ & $\mu_5^2$ & $y^2 = x^3 + p_{20}x + p_{30}$  & \cite{BS-5Selmer, Fisher-invarianttheorynormalquinticI} \\
    \hline
    $\rl +1$ & $A_\rl\: (\rl\geq 2)$ & $\G_m^\rl$ & sum of characters & $1$ & $xy = p_{\rl+1}$  &   \\
    $\{\frac{2\rl}{n}\colon n\mid \rl\}$ & $B_\rl\: (\rl\geq 2)$ & various & various & $\mu_2^{n-1}$ & $xy = p_{2\rl}$  &   \\
    \hline
	\end{tabular}
\end{center}
\caption{Classification of subregular-adapted gradings and their subregular curves}\label{table:examples intro}
\end{table}

\end{landscape}
}
\restoregeometry
\vspace{.1in}

\begin{remark}
    For many entries in Table \ref{table:examples intro}, there even exists a triple $(H, \theta, e)$ defined over $\Q$ such that $X_e\rightarrow B$ is isomorphic to the family of curves of that entry.
    When $m=2$ and $H$ is simply laced, this follows from \cite[Theorem 3.8]{Thorne-thesis}.
    When $H$ is of type $E_8$ and $m=3,5$, this follows from \cite[Proposition 4.12]{Thorne-Romano-E8} and Proposition \ref{proposition: determination surface slice 5-grading} respectively.
    In those examples, $H$ can be taken to be split.
    In the case $(m,H) = (3,{}^3D_4)$, no such split $H$ exists, see Section \ref{subsec: stable gradings over general fields}.
    However, there exists a triple $(H, \theta, e)$ over $\Q$ with $H$ quasi-split of type $D_4$ such that $X_e\rightarrow B$ is the family of elliptic curves and $(G,V)$ is the representation of $\PGL_3$ on ternary cubic forms studied by Bhargava--Shankar \cite{BS-3Selmer}.
\end{remark}

\begin{example}[Subregular-adapted gradings on $E_8$]
    We illustrate Theorem \ref{thm-intro} by exhibiting the families of curves $X_e\rightarrow B$ when $H$ is split and of type $E_8$; see also \cite[Section 6]{Romano-centralextensions}.
    Let $e\in \frakh = \Lie(H)$ be a subregular nilpotent element, and let $(e,h,f)$ be an $\fraksl_2$-triple containing $e$.
    The space of $H$-invariant polynomials $\frakh \GIT H = \Spec( k[\frakh]^H)$ is isomorphic to $\A^8$ with coordinates $(p_2, p_8, p_{12}, p_{14}, p_{18}, p_{20}, p_{24}, p_{30})$.
    The restriction of the invariant map $\frakh \rightarrow \frakh \GIT H$ to the slice $S = e + \frakz_{\frakh}(f)$ is a family of affine surfaces $S\rightarrow \A^8$ whose fiber over a point $b = (p_2, \dots,p_{30})$ is isomorphic to 
    \[
    S_b\colon z^2 + y^3 = x^5+ (p_2x^3+p_8x^2+p_{14}x+p_{20})y+(p_{12}x^3+p_{18}x^2+p_{24}x+p_{30}). 
    \]
    For each $m\in \{2,3,5\}$, there exists a stable subregular-adapted $\mu_m$-action $\theta_m$.
    The space of invariant polynomials $B = V\GIT G$ is the closed subscheme of $\frakh \GIT H$ obtained by setting those $p_i$ equal to zero unless $i$ is divisible by $m$.
    For each $m=2,3,5$, the triple $(e,h,f)$ can be chosen so that $X = S\cap V \rightarrow B$ is a family of curves, corresponding to setting $z,y$ or $x$ equal to zero, respectively. 
    When $m=2$, each smooth member $X_b$ is an open subset of a genus $4$ curve $C_b$, and the representation $(G,V)$ has been used to study the $2$-Selmer group of $\mathrm{Jac}(C_b)$ when $b$ varies in $V \GIT G$; see \cite{Thorne-thesis, Romano-Thorne-ArithmeticofsingularitiestypeE, Laga-ADEpaper}.
    When $m=3$, each smooth member $X_b$ is an open subset of a genus $2$ curve $C_b$ with a marked Weierstrass point, and the representation $(G,V) = (\SL_9/\mu_3, \wedge^3(9))$ has been used to determine the average size of the $3$-Selmer group of $\mathrm{Jac}(C_b)$ when $b$ varies over integral $b = (p_{12},p_{18},p_{24},p_{30})$ ordered by height; see \cite{Thorne-Romano-E8}.
    When $m = 5$, each smooth member $X_b$ is an open subset of an elliptic curve $E_b$ with a marked Weierstrass point, and the representation $(G,V)$ has been used to determine the average size of the $5$-Selmer group of elliptic curves when ordered by naive height \cite{BS-5Selmer}.
\end{example}

Note that there are entries on the table that have not yet been studied (to the best of our knowledge), particularly the stable $\mu_2$-gradings in types $B_\rl$ and $C_\rl$.  It would be interesting to investigate these entries 
and to see whether one can obtain new results in arithmetic statistics for the corresponding families of curves. In this direction, Martí Oller analyzes the $(m,H) = (2,B_{2n})$ case in upcoming work.

Beyond the $m=2$ case and the final two rows of Table \ref{table:examples intro} (which give rise to genus-zero curves), all remaining subregular-adapted $(G,V)$ arise from gradings on exceptional groups $H$.
These exceptional cases have all been studied from a different viewpoint in \cite{BhargavaHo-coregularspacesgenusone} and \cite{BhargavaHo-2Selmergroupsofsomefamilies}.
Interpreting them in the framework of Vinberg theory can be used to reprove these orbit parametrizations.
As an extended example, we treat a stable $\mu_5$-action on $E_8$, and use it to reprove the integral orbit parametrization of $5$-Selmer groups of elliptic curves.
To state this parametrization, let $G = (\SL_5\times \SL_5)/\mu$, where $\mu=\{ (\zeta, \zeta^2)\colon \zeta\in \mu_5\}$ and let $V = (5)\boxtimes \wedge^2(5)$, where $(5)$ denotes the standard representation of $\SL_5$.

\begin{theorem}\label{theorem:5Selmerintro}
There exists invariant polynomials $I,J \in \Q[V]^G$ (unique up to $\Q^{\times}$-multiples) and an integer $N\geq 1$ with the following properties:
\begin{enumerate}
    \item $I,J$ are algebraically independent, have degrees $20$ and $30$ respectively, and satisfy $\Q[V]^G = \Q[I,J]$.
    \item If $k/\Q$ is a field and $A,B\in k$ satisfy $4A^3 +27B^2\neq0$, consider the elliptic curve $E_{A,B}\colon y^2=  x^3+ A x + B$.
    Then there exists an embedding \[\eta_{A,B}\colon E_{A,B}(k)/5E_{A,B}(k) \hookrightarrow G(k)\backslash V_{A,B}(k)\] compatible with base change along field extensions, where $V_{A,B}(k)$ denotes the subset of $v \in V(k)$ with $I(v) = A$ and $J(v) = B$.
    \item If $k=\Q$, $\eta_{A,B}$ extends to an embedding $\tilde{\eta}_{A,B}\colon \Sel_5 E_b\hookrightarrow G(\Q)\backslash V_b(\Q)$.
    If $A,B \in \Z$, then the image of $\tilde{\eta}_{A,B}$ lies in $\frac{1}{N} V(\Z)$, i.e., has integral representatives up to bounded denominators.
\end{enumerate}
\end{theorem}
As mentioned before, this (partially) recovers results of Fisher \cite{Fisher-invariantsgenusone, Fisher-invarianttheorynormalquinticI, Fisher-invarianttheoryquinticII, Fisher-minimisationreduction5coverings}, whose work forms the basis of the determination of the average size of $5$-Selmer groups of elliptic curves in \cite{BS-5Selmer}.

\begin{remark}
   Although we focus on subregular-adapted $\mu_m$-actions,
    much of our framework applies to more general stable gradings. There are examples of non-subregular-adapted gradings for which we can construct families of curves $X_e\rightarrow B$ (we construct some examples in the appendix). In fact, computations suggest that for every stable $\mu_m$-action there exists a nilpotent $e \in \frakh_1$ such that the fibers of $X_e \rightarrow B$ form a family of curves, but we have not found a general proof of this.
    However, the same computations suggest that non-subregular-adapted gradings do not give rise to interesting new families of curves, and that they are often of genus zero or already arise as transverse slices in subregular-adapted Vinberg representations. 
\end{remark}

\begin{remark}
Finally, we comment on which stable gradings are subregular adapted, assuming the base field $k$ to be algebraically closed of characteristic zero. Every simple adjoint group $H$ has a unique stable $\mu_2$-grading, and this grading is always subregular adapted, using the results of \cite{Thorne-thesis} and \cite{KostantRallis-Orbitsrepresentationssymmetrisspaces}. 
By Theorem \ref{thm-intro}, every stable $\mu_3$-grading is also subregular adapted, but there are stable $\mu_4$-gradings that are not subregular adapted. As mentioned already, subregular-adapated gradings of classical groups are rare. For exceptional groups $H$, every stable grading with nontrivial generic stabilizer (i.e. such that the centralizer in $G$ of a stable vector is nontrivial) is subregular adapted with only three exceptions: the stable $\mu_4$- and $\mu_8$-gradings on $E_8$ and the stable $\mu_8$-grading on $F_4$. We give examples of curves coming from the $\mu_8$-grading of $F_4$ in the appendix. 
An interesting geometric construction related to the $\mu_4$-grading on $E_8$ appears in \cite[Section 9]{Gruson-Sam-Weyman}, but it is not yet clear to what extent these three gradings might be useful in orbit-parametrization problems. 
\end{remark}

\subsection{Organization}
We start in Section \ref{section: background} by reviewing some background on graded Lie algebras, the classification of stable gradings, and nilpotent elements. 
We complement the literature by discussing stable gradings for general reductive groups and in the case when the field of definition is not algebraically closed. 
In Section \ref{section: constructing families of curves}, we adapt Slodowy's recipe for constructing transverse slices to the setting of Vinberg representations, and give some first examples.
In Section \ref{section:subregular curves}, we define and classify all subregular-adapted gradings and subregular curves and prove Theorem \ref{thm-intro}.
In Section \ref{section: 5 grading E8}, we treat the stable $5$-grading on $E_8$ in detail, and reprove the integral orbit parametrization for $5$-Selmer groups of elliptic curves, hence proving Theorem \ref{theorem:5Selmerintro}.

At various points in the paper, we need to perform calculations with nilpotent elements in exceptional Lie algebras. 
We explain in the appendix how we use de Graaf's \texttt{SLA} package \cite{deGraaf-nilporbitsthetagroups} (integrated in \texttt{GAP} \cite{GAP4}) for this purpose. In the appendix we also give examples of families of curves that come from non-subregular-adapted gradings.

\subsection{Notation}\label{subsec: notation}

\begin{itemize}
    \item Throughout this paper, $k$ will denote a field of characteristic zero with algebraic closure $\bar{k}$ and Galois group $\Gamma_k = \Gal(\bar{k}|k)$.
    \item A variety over $k$ is by definition a separated scheme of finite type over $k$.
    If $K/k$ is a field extension, we denote the $K$-points of $X$ by $X(K) = \Hom(\Spec(K), X\times_k K)$.
    \item If $V$ is a $k$-vector space, let $k[V] = \Sym^*(V^{\vee})$ be the $k$-algebra of polynomials in $V$.
    There is a natural bijection between $V$ and the $k$-points of $\Spec(k[V])$, and we use the latter to view the former as an affine space over $k$. 
    \item If an algebraic group $G$ acts on a variety $Y$ and $y\in Y(k)$, write $Z_G(y) = \{g\in G\colon g\cdot y = y\}$ for the stabilizer of $y$, an algebraic subgroup of $G$.
    \item If $V$ is a representation of a Lie algebra $\lieg$ and $v\in V$, write $\frakz_{\lieg}(v) = \{x\in \lieg\colon x\cdot v= 0\}$ for the stabilizer of $v$, a subalgebra of $\lieg$.
    \item If $G$ is a reductive group over $k$, we denote by $\Aut_G$ automorphism group scheme of $G$, which has the property that $\Aut_G(K) = \Aut_K(G_K)$ for every field extension $K/k$.
    We denote the center of $G$ by $Z_G$ and let $G^{\mathrm{ad}} = G / Z_G$ be the adjoint group of $G$.
    There is an exact sequence
    \begin{align}
        1\rightarrow G^{\mathrm{ad}} \rightarrow \Aut_G \rightarrow \Out_G \rightarrow 1,
    \end{align}
    where $\Out_G$ is the outer automorphism group scheme; see \cite[Theorem 7.1.9]{Conrad-SGA} for these results.
    \item If $G$ is an algebraic group over $k$ with Lie algebra $\frakg$ and $x\in \frakg$, then the inclusion $\Lie(Z_G(x))  \subset \frakz_{\frakg}(x)$ is an equality since we assume $k$ has characteristic zero; see \cite[Proposition 1.10]{Humphreys-conjugacyclassesalgebraic}.
    We use this fact without further mention.
\end{itemize}

\section{Graded Lie algebras}\label{section: background}

\subsection{Basic properties}

Let $m\geq 1$ be an integer and let $H/k$ be an algebraic group.
A $\mu_m$-action on $H$ is, by definition, a morphism of schemes $\alpha\colon \mu_m \times H\rightarrow H$ satisfying the axioms of a group action and such that $\alpha(\zeta,-)\colon H_{\bar{k}}\rightarrow H_{\bar{k}}$ is a homomorphism of algebraic groups for every $\zeta \in \mu_m(\bar{k})$.
Giving a $\mu_m$-action on $H$ is the same as giving a $\Gamma_k$-equivariant homomorphism $\theta\colon \mu_m(\bar{k})\rightarrow \Aut(H_{\bar{k}})$, which is also the same as giving a morphism of group schemes $\mu_m\rightarrow \Aut_H$.
We will use these three points of view interchangeably.
If we fix a root of unity $\zeta$ of order $m$ in $k$, giving a $\mu_m$-action $\theta$ on $H$ is the same as giving an order-$m$ automorphism $\theta(\zeta)$ of $H$.
If no such $\zeta$ exists or is given, we prefer to use the concept of a $\mu_m$-action.

If $h\in H(k)$, let $\Ad(h)\in \Aut(H)$ be the homomorphism $x\mapsto hxh^{-1}$.
Two $\mu_m$-actions on $H$, seen as homomorphisms $\theta, \theta'\colon \mu_m(\bar{k})\rightarrow \Aut(H_{\bar{k}})$, are said to be $H$-conjugate if $\theta' = \Ad(h) \circ \theta\circ \Ad(h)^{-1}$ for some $h\in H(k)$.
Two such homomorphisms are said to be $H^{\mathrm{ad}}$-conjugate if we only require $h$ to define a $k$-rational element of $H^{\mathrm{ad}}(k)$, in other words if there exists an element $\varphi\in H^{\mathrm{ad}}(k)$ such that $\theta' = \varphi \circ \theta\circ \varphi^{-1}$.
Note that if $\theta, \theta'$ are $H$-conjugate, then they are also $H^{\mathrm{ad}}$-conjugate.
The converse holds if $k$ is algebraically closed.

Let $\theta \colon \mu_m(\bar{k})\rightarrow \Aut(H_{\bar{k}})$ be a $\mu_m$-action on $H$.
Let $\frak{h}$ be the Lie algebra of $H$.
Taking derivatives, we obtain a $\mu_m$-action on $\frakh$, still denoted by $\theta\colon \mu_m(\bar{k})\rightarrow \Aut(\frakh_{\bar{k}})$.
For $i\in \Z/m\Z$, let 
\begin{equation*}
\frakh_i = \{ x\in \frakh \colon \theta(\zeta)(v) = \zeta^i v\text{ for all }\zeta\in \mu_m(\bar{k})\}.
\end{equation*}
These subspaces define a $\Z/m\Z$-grading on $\frakh$, in the sense that we have a direct-sum decomposition
\begin{align*}
\lieh = \bigoplus_{i \in \Z/m\Z} \lieh_i
\end{align*}
into linear subspaces satisfying $[\lieh_i,\lieh_j ] \subset \lieh_{i+j}$ for all $i,j\in \Z/m\Z$.
Let $G = (H^{\theta})^{\circ}$ be the identity component of the fixed-point subgroup under the action. 
Then $G$ has Lie algebra $\frakg:= \frakh_0$.
The adjoint representation of $H$ on $\frakh$ restricts to a representation of $G$ on $V:= \frakh_1$.

We have just associated, to any $\mu_m$-action on $H$, a representation $V$ of an algebraic group $G$.
We refer to the pair $(G,V)$ as a \emph{Vinberg representation}, and the study of the invariant theory of this pair is known as \emph{Vinberg theory} \cite{Vinberg-theweylgroupofgraded}.

\begin{remark}
    Suppose that $H$ is semisimple.
    If $H$ is also adjoint (meaning it has trivial center) or simply connected, then the natural map $\Aut(H) \rightarrow \Aut(\frak{h})$ is an isomorphism (cf. \cite[Section 1.5]{Conrad-SGA}, \cite[Section 3.1]{Reeder-torsion}).
    So giving a $\mu_m$-action on $H$ is equivalent to giving one on $\frakh$, which is equivalent to giving a $\Z/m\Z$-grading on $\frak{h}$.
    In general, we use the setting of $\mu_m$-actions on $H$ rather than $\Z/m\Z$-gradings on $\frakh$, since the definition of $G$ involves $H$, not just $\lieh$, and sometimes there are multiple natural choices for $H$. 
\end{remark}

We now summarize some of the highlights of Vinberg theory, referring to \cite{Vinberg-theweylgroupofgraded} (or the more modern \cite{Panyushev-Invarianttheorythetagroups,  Levy-Vinbergtheoryposchar}) for proofs.
We assume for these results and the remainder of the paper that $H$ is reductive.
Note that $H$ and $\theta$ can be defined over a subfield of $k$ that can be embedded in $\mathbb{C}$. 
So for any property of $(H, \theta)$ whose validity is invariant under base extension, we may assume $k=\mathbb{C}$ and freely cite references that only treat this case.

We say an element $x\in V$ is semisimple, nilpotent or regular if it is so when considered as an element of $\lieh$. 
A subspace $\mathfrak{c} \subset V$ is called a \emph{Cartan subspace} if it consists of semisimple elements, satisfies $[\mathfrak{c},\mathfrak{c}]=0$, and is maximal with these properties (among subspaces of $V$).

\begin{proposition}
    \begin{enumerate}
        \item $G = (H^{\theta})^{\circ}$ is reductive.
        \item If $x\in V$ has Jordan decomposition $x = x_s + x_n$ where $x_s, x_n$ are commuting elements that are semisimple and nilpotent respectively, then $x_s,x_n \in V$.
        \item Every semisimple $x\in V$ is contained in a Cartan subspace of $V$. 
Every two Cartan subspaces are $G(\bar k)$-conjugate. 
    \end{enumerate}
\end{proposition}
\begin{proof}
    The first two parts are justified in \cite[Section 1]{Vinberg-theweylgroupofgraded}.
    The last part is \cite[Section 3, Theorem 1]{Vinberg-theweylgroupofgraded}.
\end{proof}

We call a triple $(e,h,f)$ an \emph{$\liesl_2$-triple} of $\lieh$ if $e,h,f$ are nonzero elements of $\lieh$ satisfying the following relations:
\begin{equation*}
[h,e] = 2e , \quad [h,f] = -2f ,\quad [e,f] = h .
\end{equation*}
The classical Jacobson--Morozov Theorem states that every nonzero nilpotent element in $\lieh$ can be completed to an $\liesl_2$-triple. 
If $\lieh$ is $\Z/m\Z$-graded, we say an $\liesl_2$-triple is \emph{normal} if $e \in \lieh_1$, $h \in \lieh_0$ and $f\in \lieh_{-1}$.

\begin{lemma}[Graded Jacobson--Morozov]\label{lemma: graded Jacobson-Morozov}
Every nonzero nilpotent $e\in \lieh_{1}$ is contained in a normal $\liesl_2$-triple $(e,h,f)$.
Any two normal $\liesl_2$-triples $(e,h,f)$ and $(e,h',f')$ are $Z_G(e)(k)$-conjugate.
\end{lemma}
\begin{proof}
    In the case when $k$ is algebraically closed, this is \cite[Theorem 1]{Vinberg-nilpotent}. The case when $k\neq \bar{k}$ follows from the same reasoning as the proof of \cite[Lemma 2.17]{Thorne-thesis}.
\end{proof}

The next proposition describes the basic geometric invariant theory of the representation $(G,V)$.
Let 
\begin{align}\label{equation: GIT quotient map}
\pi\colon V\rightarrow V \GIT G = \Spec k[V]^G    
\end{align} 
be the GIT quotient map.
\begin{proposition}\label{proposition: basic invariant theory of vinberg representation full generality}
\begin{enumerate}
\item Let $\mathfrak{c} \subset V$ be a Cartan subspace and $W(\mathfrak{c})  \coloneqq N_G(\mathfrak{c})/Z_G(\mathfrak{c})$. Then the inclusion $\mathfrak{c} \subset V$ induces an isomorphism 
\begin{align*}
\mathfrak{c} \GIT W(\mathfrak{c}) \simeq V\GIT G. 
\end{align*}
The group $W(\mathfrak{c})$ is a finite pseudo-reflection group, so the quotient is isomorphic to affine space.

\item If $k$ is algebraically closed and $b\in (V \GIT G)(k)$, then the fiber $\pi^{-1}(b)$ consists of finitely many $G(k)$-orbits.
It contains a unique closed $G(k)$-orbit, the set of semisimple elements with invariants $b$.
\item If $k$ is algebraically closed and $v_1,v_2\in V(k)$, then $\pi(v_1) = \pi(v_2)$ if and only if the semisimple parts of $v_1$ and $v_2$ in their Jordan decompositions are $G(\bar{k})$-conjugate.

\item The quotient map $\pi\colon V\rightarrow V\GIT G$ is flat, and every fiber has dimension $\dim V - \dim \mathfrak{c}$.
\end{enumerate}
\end{proposition}

\begin{proof}
We may assume $k=\mathbb{C}$.
These claims then follow from \cite[Sections 2,4]{Vinberg-theweylgroupofgraded} (see also \cite[Theorem 1.1]{Panyushev-Invarianttheorythetagroups}), except the flatness of $\pi$.
This flatness follows from the fact that the source and domain of $\pi$ are smooth, its fibers are equidimensionsal, and ``Miracle Flatnes'' \cite[Theorem 23.1]{Matsumura-CommutativeRingTheory}.
\end{proof}

\subsection{Stable and principal gradings}\label{subsec: stable gradings over algebraically closed field}

Maintain notation as in the previous section, so $H$ is a reductive group with $\mu_m$-action $\theta$, and $(G = (H^\theta)^\circ, V = \frakh_1)$ is the corresponding Vinberg representation. Of particular interest to us are stable gradings.

\begin{definition}\label{definition: stable gradings}
A vector $v \in V$ is \emph{stable} (in the sense of geometric invariant theory) if the $G$-orbit of $v$ is closed and its stabilizer $Z_G(v)$ is finite (as an algebraic group).
We say the Vinberg representation $(G,V)$ is \emph{stable} if $V_{\bar{k}}$ contains stable vectors for the action of $G_{\bar{k}}$. 
\end{definition}

We say a $\mu_m$-action $\theta$ on $H$ is \emph{stable} if the associated Vinberg representation $(G,V) = ((H^{\theta})^{\circ},\frakh_1)$ is stable.

\begin{remark}
    All the known examples of Vinberg representations that are employed in arithmetic statistics (for example those appearing in \cite{BhargavaHo-2Selmergroupsofsomefamilies, BS-4Selmer, BS-5Selmer, Bhargava-Gross-hyperellcurves}) are stable.
    If $(G,V)$ is stable, then the stable vectors form a Zariski open subset $V^s\subset V$ with the following two desirable properties.
    Firstly, the geometric orbits (i.e. those over an algebraically closed field) in $V^s$ are separated by their invariants, by Proposition \ref{proposition: basic invariant theory of vinberg representation full generality}.
    Secondly, if $v\in V^s(k)$ with $\pi(v) = b$, and $V_b:= \pi^{-1}(b)$
 is the subscheme of elements with invariants $b$, then $G(k)\backslash V_b(k)$ is in natural bijection with the pointed kernel of the maps on Galois cohomology $\ker(\HH^1(k, Z_G(v)) \rightarrow \HH^1(k,G))$, see \cite[Section 2, Proposition 1]{BhargavaGross-AIT}. 
    Since $\Stab_G(v)$ is a finite group scheme, it can often be related to a torsion subgroup $A[n]$ of an abelian variety $A/k$, so the set of orbits $G(k)\backslash V_b(k)$ is a natural target for the descent map $A(k)/nA(k)\rightarrow \HH^1(k,A[n])$ and may contain Selmer groups when $k$ is a global field; this is the case for all the references cited above.
\end{remark}

 \begin{example}\label{ex-torus}
 If $H$ is a torus, then the stable $\mu_m$-actions on $H$ are in bijection with the elliptic order-$m$ automorphisms of $H$. In this case, if $\theta$ is stable, then $H^\theta$ is finite, $G$ is trivial, and every element of $V$ is stable.
 \end{example}

Stable $\mu_m$-actions on simple algebraic groups $H$ over an algebraically closed field of characteristic zero have been classified \cite[\S7.1, \S7.2]{GrossLevyReederYu-GradingsPosRank} in terms of $\Z$-regular elliptic conjugacy classes of (twisted) Weyl groups.
We recall the classification in the next subsection and some properties of stable gradings. We assume some basic knowledge about semisimple Lie algebras and split reductive groups, which can be found, for example, in \cite{Humphreys}, \cite{Steinberg-Chevalley}, \cite{Conrad-SGA}. For relevant background about finite-order automorphisms of Lie algebras, see \cite{Reeder-torsion}.

Let $H$ be a reductive group over a field $k$ of characteristic zero. 
Assume that there exists a pinning $(T,B,\{E_i\})$ of $H$ \cite[Definition 1.5.4]{Conrad-SGA}.
This is the data of a maximal torus $T\subset H$ (which determines a set of roots $\Phi\subset X^*(T_{\bar{k}}) = \Hom(T_{\bar{k}}, \G_{m,\bar{k}})$), a Borel subgroup $B\subset H$ containing $T$ (which determines a set of simple roots $\{\alpha_1, \dots, \alpha_{\ell}\}\subset \Phi$) and a generator $E_i \in \frakh_{\alpha_i}$ for each root space $\alpha_i$.
We require $T,B$ and the set $\{E_i\}$ to be defined over $k$, but we don't require the individual elements $E_i$ to be defined over $k$.
Such a pinning exists if and only if $H$ is quasi-split (for example, when $k = \bar{k}$), and in that case every two pinnings are conjugate by a unique element of $H^{\mathrm{ad}}(k)$.
An automorphism of $H$ is said to be \textit{pinned} if it preserves $T$, $B$ and the set $\{E_i\}$.

Recall that $H^{\mathrm{ad}}$ is the quotient of $H$ by its center $Z_H$ and that we have an exact sequence
\[
1\rightarrow H^{\mathrm{ad}}\xrightarrow{\Ad} \Aut_H \rightarrow \Out_H\rightarrow 1.
\]
Let $\Theta_H \subset \Aut_H$ denote the subgroup (scheme) of pinned automorphisms of $H$.
Then $\Theta_H$ defines a splitting of the above sequence, giving rise to the identification $\Aut(H) = H^{\mathrm{ad}} \rtimes \Theta_H$.
If $H$ is split and semisimple, then $\Theta_H$ is constant and can be identified with a subgroup of the automorphism group $\Aut(D)$ of the Dynkin diagram $D$ of $\frakh$.
If $H$ is additionally adjoint (in other words, $H = H^{\mathrm{ad}}$) or simply connected, then $\Theta_H$ equals $\Aut(D)$.
See \cite[Section 1.5]{Conrad-SGA} for these claims.

Given a $\mu_m$-action 
$\theta\colon \mu_m\rightarrow \Aut_H$,
we can postcompose $\theta$ with $\Aut_H\rightarrow \Out_H$ to obtain a homomorphism $\mu_m\rightarrow \Out_H$. 
There exists a unique divisor $e$ of $m$ such that the latter homomorphism factors as $\mu_m \rightarrow \mu_e \rightarrow \Out_H$, where $\mu_m\rightarrow \mu_e$ is raising to the $m/e$-th power, and $\tau\colon \mu_e\hookrightarrow \Out_H$ is injective.
We call $\tau$ the \emph{outer type} of $\theta$.
If $\theta$ is $H_{\bar{k}}$-conjugate to $\theta'$, they have the same outer type.

To recall the classification of stable gradings, we first define principal gradings. 
Write $\check\rho$ for the sum of the fundamental co-weights corresponding to our choice of simple roots; it is the unique element of $X_*(T^{\mathrm{ad}}) = \Hom(\G_m,T^{\mathrm{ad}})$ satisfying $\langle \alpha_i, \check{\rho} \rangle = 1$ for all $i$. 
Let $\tau\colon \mu_e\hookrightarrow \Out_H$ be an injection of group schemes.
If a primitive $e$th root of unity $\zeta_e\in k$ exists and is fixed, giving $\tau$ is the same as giving an element $\vartheta\in \Out_H(k)$ of order $e$. 
Let $m\geq 1$ be an integer divisible by $e$.
Let $\theta\colon \mu_m\rightarrow \Aut_H$ be defined by
\begin{align}\label{equation: definition pinned principal mu_m action}
    \theta(\zeta) = \Ad(\check{\rho}(\zeta)) \circ \tau(\zeta^{m/e})  \text{ for }\zeta\in \mu_m.
\end{align}
Since $\check{\rho}$ and $\vartheta$ commute, this defines a $\mu_m$-action on $H$.

\begin{definition}
 We say a $\mu_m$-action $\theta: \mu_m \to \Aut_H$ is \emph{principal} (respectively \emph{split principal}) if it is $H({\bar{k}})$-conjugate (respectively $H^{\mathrm{ad}}(k)$-conjugate) to one of the form \eqref{equation: definition pinned principal mu_m action} for some outer type $\tau\colon \mu_e\hookrightarrow \Out_H$ and integer $e$ dividing $m$.
\end{definition}

\begin{example}\label{ex-princ-inner}
If $\tau$ is trivial (in other words, $\theta$ is an inner automorphism), then \eqref{equation: definition pinned principal mu_m action} simply becomes $\Ad(\check{\rho}(\zeta))$ and we can describe the associated grading in the following way: given a root $\al \in \Phi$, let $\al = \sum_{i = 1}^\ell a_i\al_i$ be its decomposition as a sum of simple roots. The root height of $\al$ is then $\height(\al) := \sum_{i = 1}^\ell a_i$. The grading of $\frakh$ corresponding to $\theta$ is then given by
\begin{align*}
\frakh_0 &= \frakt \oplus \sum_{\substack{\height(\al) \equiv 0\\ \mod m}} \frakh_\al,\\
\frakh_j &= \sum_{\substack{\height(\al) \equiv j \\\mod m}} \frakh_\al \quad (\text{if }j \neq 0).
\end{align*}
(Here $\frakh_\al$ denotes the root space corresponding to $\al \in \Phi$.) If $\tau\neq 1$, a similar but more complicated explicit description can be given.
\end{example}

\begin{lemma}\label{lemma: split principal if and only if regular nilpotent exists}
    Suppose $H$ is quasi-split.
    A $\mu_m$-action $\theta$ on $H$ is split principal if and only if there exists a regular nilpotent $E\in V(k)$.
\end{lemma}
\begin{proof}
    Fix a pinning $(T, B, \{E_i\})$ of $H$.
    If $\theta$ is split principal, then after $H^{\mathrm{ad}}(k)$-conjugation we may assume it equals the $\mu_m$-action \eqref{equation: definition pinned principal mu_m action} associated to this pinning.
    Then $E = \sum E_i$ is a regular nilpotent in $V(k)$, see \cite[Section 5, Corollary 5.3]{Kostant-principalthreedimensional}.
    Conversely, suppose $V(k)$ contains a regular nilpotent $E$.
    Complete $E$ to a normal $\liesl_2$-triple $(E,X,F)$ using Lemma \ref{lemma: graded Jacobson-Morozov}.
    Let $\tau\colon \mu_e\hookrightarrow \Out_H$ be the outer type of $\theta$, and let $\theta_0$ be the split principal $\mu_m$-action associated to $\tau$ and the pinning $(T, P,\{E_i\})$, inducing a grading $\frakh= \oplus \frakh_i'$.
    Let $E_0 = E_1+ \cdots + E_{\ell}$.
    Then $E_0\in \frakh_1'(k)$ is regular nilpotent; complete it to a normal $\liesl_2$-triple $(E_0,X_0,F_0)$ with respect to $\theta_0$.

    Since all regular nilpotent elements are conjugate and the stabilizer of $(E,X,F)$ in $H^{\mathrm{ad}}$ is trivial \cite[Section 5, Corollaries 5.2 and 5.5]{Kostant-principalthreedimensional}, there exists a unique element $\varphi\in H^{\mathrm{ad}}(\bar{k})$ conjugating $(E,X,F)$ to $(E_0,X_0,F_0)$, which by uniquess must necessarily be defined over $k$.
    So after conjugating $\theta$ by $\varphi$ we may assume $(E,X,F) = (E_0,X_0,F_0)$.
    If $\zeta\in \bar{k}$ is a root of unity of order $m$, then
    Then $\theta_0(\zeta)^{-1} \circ \theta(\zeta)$ is an inner automorphism of $H$ that preserves $E,X$ and $F$.
    This automorphism must be trivial by the first sentence of the paragraph, hence $\theta = \theta_0$, as desired.
\end{proof}

 The next two examples give explicit descriptions of the inner principal 2-gradings when $H$ has type $B$ or $C$ following \cite{Ohta-classificationadmissiblenilpotent}. As we will see in the next section, for each such $H$, this is the unique stable $2$-grading of $H$. We will use these explicit descriptions to descibe subregular nilpotent elements in $\frakh_1$ in Section \ref{subreg-exhibition}. 

\begin{example}\label{ex-Bn-2grading}(Type $B_{\rl}$)
Let $\rl\geq 2$ be an integer.
Let $W_1$ be the vector space over $k$ with basis $v_1, \dots, v_{\rl+1}$ and $W_2$ the vector space with basis $w_1, \dots, w_\rl$ and $W = W_1\oplus W_2$.
Let $(-,-)_1$ be the symmetric bilinear form on $W_1$ with the property that $(v_i, v_{\rl+1-i})_1 = 1$ for all $1\leq i \leq \rl$, $(v_{\rl+1},v_{\rl+1})_1 = 1$ and all other pairings between basis vectors are zero.
Let $(-,-)_2$ be the symmetric bilinear form on $W_2$ with $(w_i, w_{\rl-i})_2=-1$, $(w_\rl, w_\rl)_2 = -1$ and all other basis pairings are zero.
Let $(-,-) = (-,-)_1 \oplus (-,-)_2$ be the direct sum pairing on $W$.
Let $H = \SO(W)$ with Lie algebra $\frakh = \mathfrak{so}(W)$; a semisimple Lie algebra of type $\rl$.
Let $s\in \GL(W)$ be the element that acts as $\Id$ on $W_1$ and $-\Id$ on $W_2$.
Then $s$ lies in $\SO(W)$ and $\theta=  \Ad(s)$ defines a $\Z/2\Z$-grading on $\frakh$.
In the basis $v_1, \dots, v_\rl, w_1, \dots, w_\rl$, we have
\[
\frakh_0 = 
\left\{
\begin{pmatrix}
    X & 0 \\ 0 & W 
\end{pmatrix}
\colon X \in \mathfrak{so}(W_1), W\in \mathfrak{so}(W_2)
\right\}
, \quad 
\frakh_1 = \left\{
\begin{pmatrix}
    0 & Y \\ Z & 0 
\end{pmatrix}
\colon Y \in M_{(\rl + 1) \times \rl}, Z \in M_{\rl \times (\rl + 1)}
\right\} \cap \frakh.
\]
Here $M_{n_1 \times n_2}$ denotes the space of $n_1 \times n_2$ matrices. Note that given $Y \in M_{(\rl + 1) \times \rl}$ as above, $Z$ is completely determined by the fact that $\langle Av, w\rangle_2 + \langle v, Aw\rangle_1 = 0$ for all $v \in W_1, w \in W_2, A \in \frakh_1$. We have that $\frakh_1 \simeq \Hom(W_1, W_2)$ as a representation of $G = \SO(W_1) \times \SO(W_2)$, where $\SO(W_1)$ acts as precomposition and $\SO(W_2)$ acts by post-composition.
Using the classification of nilpotent $G$-orbits in $V$ of \cite[p. 305, Proposition 2, Case (BDI)]{Ohta-classificationadmissiblenilpotent} in terms of ab-diagrams and the fact that the diagram $[abab\dots aba]$ gives rise to the regular nilpotent element by \cite[Sections 5.1 and 5.4]{CollingwoodMcGovern-nilpotentorbits}, we conclude that $V(\bar{k})$ contains a regular nilpotent element.
Therefore this 2-grading is principal by Lemma \ref{lemma: split principal if and only if regular nilpotent exists}.
\end{example}

\begin{example}\label{ex-Cn-2grading} (Type $C_{\rl}$)
Let $\rl\geq2$ be an integer.
Let $W_1$ be the vector space with basis $v_1, \dots, v_\rl$, $W_2$ the vector space with basis $w_1, \dots, w_\rl$ and let $W = W_1 \oplus W_2$.
Let $W$ be a vector space with basis $v_1, \dots, v_\rl, w_1, \dots, w_\rl$, .
Let $(-,-)$ be the nondegenerate alternating form on $W$ with $(v_i,w_j) = \delta_{ij}$, $(w_i, v_j) = -\delta_{ij}$ and $(v_i, v_j) = (w_i, w_j)$ for all $1\leq i,j\leq \rl$.
This defines the group $H = \PSp(W)$ and Lie algebra $\frak{h} = \mathfrak{sp}(W)$
Let $s\in \GL(W)$ be given by $v_i\mapsto v_i$, $w_i \mapsto -w_i$ and let $\theta = \Ad(s)$. 
Then $\theta$ induces a $\Z/2\Z$-grading on $\frakh$.
In the basis $v_1, \dots, v_\rl, w_1, \dots, w_\rl$, we have
\[
\frakh_0 = 
\left\{
\begin{pmatrix}
    X & 0 \\ 0 & -X^t 
\end{pmatrix}
\colon X \in \frak{gl}_\rl
\right\}
, \quad 
\frakh_1 = \left\{
\begin{pmatrix}
    0 & Y \\ Z & 0 
\end{pmatrix}
\colon Y,Z \in \frak{gl}_\rl, Y^t = Y \text{ and }Z^t = Z
\right\}.
\]
The fact that this is the principal $2$-grading on $H$ follows by the same logic as in the previous example, using \cite[p. 305, Proposition 2, Case (CI)]{Ohta-classificationadmissiblenilpotent}.
\end{example}

\subsection{Classification of stable gradings}

Keep the notation of the previous subsection.
Let $W = N_H(T_{\bar{k}})/T_{\bar{k}}$ be the Weyl group of $T_{\bar{k}}$ and let $\Lambda = X^*(T_{\bar{k}})$.
Let $\Aut(\Lambda, \Phi)$ denote the set of group automorphisms of $\Lambda$ preserving $\Phi$. 
An automorphism $\sigma  \in \Aut(\Lambda, \Phi)$ is called \textit{elliptic} if the fixed point set satisfies $\Lambda^\sigma = 0$,
and $\sigma$ is called $\bZ$-\textit{regular} if $\langle \sigma \ra$ acts freely on $\Phi$.

Since an element $\vartheta\in \Theta_H(k)$ preserves $T$, it induces an element of $\Aut(\Lambda, \Phi)$, still denoted by $\vartheta$. 
Hence we may consider the coset $W\vartheta$ in $\Aut((\Lambda, \Phi))$.

\begin{theorem}\label{theorem: characterization stable gradings}
Suppose that $k$ is algebraically closed, that $H$ is simple and let $\zeta \in k$ be a primitive $m$th root of unity.
Let $\theta\colon \mu_m\rightarrow \Aut_H$ be a $\mu_m$-action on $H$, with associated Vinberg representation $(G,V)$.
Let $\vartheta$ be the image of $\theta(\zeta)$ in $\Out_H(k)\simeq \Theta_H(k)$, in other words the unique element of $\Theta_H$ such that $\theta(\zeta)\in H^{\mathrm{ad}} \vartheta$. 
Then $(G,V)$ is stable if and only if $\theta$ is principal of order $m$ and there exists an elliptic $\Z$-regular element of $W\vartheta\subset \Aut(\Lambda, \Phi)$ of order $m$.
\end{theorem}

\begin{proof}
    This is \cite[Corollary 14]{GrossLevyReederYu-GradingsPosRank}.
\end{proof}

\begin{remark}
    Suppose that $k$ is algebraically closed, $H$ is simple, $\tau\colon \mu_e\hookrightarrow \Out_H$ an injective homomorphism, and $m$ an integer divisible by $e$.
    Then there exists at most one $H$-conjugacy class of stable $\mu_m$-action on $H$ of outer type $\tau$, since such a $\mu_m$-action must be principal by Theorem \ref{theorem: characterization stable gradings}.
    It is therefore justified to speak of ``the'' stable grading on $H$ of order $m$ and outer type $\tau$, provided it exists.
\end{remark}

Since elliptic $\Z$-regular elements of automorphism groups of irreducible root systems can be classified, stable $\mu_m$-actions on simple adjoint groups can be classified and the result is given in \cite[Section 8]{GrossLevyReederYu-GradingsPosRank}.
We highlight a few examples here.

\begin{example}\label{ex-m=2-stable}
    If $m=2$, then for any root system $\Phi$, $-1$ is the unique $\Z$-regular elliptic automorphism of $\Phi$ of order $m$.
    So there is a unique conjugacy class of stable $\Z/2\Z$-gradings on every simple Lie algebra $\lieh$. The principal 2-gradings of $\SO_{2n + 1}$ and $\Sp_{2n}$ described in Examples \ref{ex-Bn-2grading} and \ref{ex-Bn-2grading} are the stable 2-gradings for these groups, since $\Out_H$ is trivial in these cases.
\end{example}

\begin{example}\label{ex-cox}
Assume $H$ is simple. The Weyl group $W$ has a conjugacy class of \textit{Coxeter elements}, which are elliptic and $\bZ$-regular. The order of a Coxeter element is the Coxeter number $h$, which is given by $\max\{\height(\al) \mid \al \in \Phi\} + 1$ (here $\height(\al)$ is defined as in Example \ref{ex-princ-inner}, and this maximum does not depend on choice of simple roots). 
The inner stable $\mu_h$-action on $H$ is called the \textit{Coxeter grading}. For this grading we have 
$G = T$ and
\begin{equation*}
V = \sum_{i = 0}^\ell \frakh_{\al_i},
\end{equation*}
where $\al_0 \in \Phi$ is the unique root of lowest height. 
If $\{v_0, \dots, v_\ell \}$ is a basis of root vectors of $V$, then a vector $v = \sum c_iv_i$ is stable if and only if $\prod c_i \neq 0$. 
Note that if $\frakh$ has an inner stable $m$-grading, then $m$ is the order of a regular element in $W$, so $m$ divides at least one of the invariant degrees of $\frakh$ by \cite[Theorem 4.2(i)-(ii)]{Springer}, hence $m\leq h$.
So $h$ is maximal in this sense.
\end{example}

The tables of \cite{GrossLevyReederYu-GradingsPosRank} only describe stable gradings under the assumption that $H$ is simple, and one might wonder whether other interesting examples of stable gradings arise from replacing a simple group with a semisimple group. The answer, essentially, is no. To see this, assume $H$ is semisimple and $k$ is algebraically closed. For simplicity, assume $H$ is simply connected, so $H = H^1 \times \dots \times H^n$ is a product of simple groups. (It is enough to consider this case since every stable grading of a semisimple group lifts to a stable grading of its simply connected cover.) The automorphism $\theta$ induces a permutation $\sigma$ of $\{1, \dots, n\}$ via $\theta(H^i) = H^{\sigma(i)}$ for all $i$. 

As a special case, assume $\sigma$ is the $n$-cycle $(1~2\dots n)$. 
In this case all of the simple factors $H^i$ are isomorphic, $n \vert m$, and $\theta^n$ preserves $H^i$ for each $i$. We see that $\theta^n: \mu_{m/n} \to \Aut(H^1)$ is an $m/n$-grading on $H^1$. It's not hard to see that projection $H \to H^1$ gives an isomorphism $H^\theta \to H^{\theta^n}$, and projection $\frakh \to \frakh^1$ induces an isomorphism from the $\zeta_m$-eigenspace for $\theta$ to the $\zeta_m^n$-eigenspace for $\theta^n$. Thus $\theta$ is a stable $\mu_m$-action on $H$ if and only if $\theta^n$ is a stable $\mu_{m/n}$-action on $H^1$.

Generalizing from this case, for each orbit $\mathcal{O}$ for the action of $\langle \sigma \ra$ on $\{1, 2, \dots, n\}$, we see that $\theta$ restricts to a $\mu_m$-action on $H^{\mathcal{O}} := \prod_{i \in \mathcal{O}} H^i$. The $\mu_m$-action on $H$ is stable if and only if the restricted $\mu_m$-action on $H^{\mathcal{O}}$ is stable for each orbit $\mathcal{O}$. By the previous case, this $\mu_m$-action is stable if and only if the action of $\theta^{\lvert \mathcal{O} \rvert}$ is stable on some (equivalently any) $H^j$ for $j \in \mathcal{O}$.
Thus every Vinberg representation $(G, V)$ arising from a simply connected, semisimple group $H$ is of the form $G \simeq G_1 \times \dots \times G_d$, $V \simeq V_1 \oplus \dots \oplus V_d$, where each $(G_i, V_i)$ is a Vinberg representation arising from a simple group $H_i$, and $(G,V)$ is stable if and only if each $(G_i,V_i)$ is stable.

Similarly, one can use an understanding of stable gradings for semisimple groups, along with Example \ref{ex-torus}, to describe stable gradings for reductive groups. Indeed, if $H$ is reductive, then 
$\theta: \mu_m \to \Aut_H$ induces a $\mu_m$-action on both the derived subgroup $H^{\mathrm{der}}$ and the center $Z_H$. 
It's not hard to see that the grading given by $\theta$ is stable if and only if $(Z_H)^\theta$ is finite and the corresponding grading of $H^{\mathrm{der}}$ is stable. Using this idea, one can show that the assumption that $H$ is simple is unnecessary in Theorem \ref{theorem: characterization stable gradings}, but we choose not to give the details here, since it would be too much of a digression.

For later use, we record two properties of stable $\mu_m$-actions.

\begin{lemma}
    Suppose that $\theta$ is stable $\mu_m$-action on $H$ and that $x\in \lieh_1$ is semisimple.
    Let $L = Z_H(x)$.
    Then $L$ is a (connected) reductive group, and $\theta$ restricts to a stable $\mu_m$-action on $L$. 
\end{lemma}
\begin{proof}
    We may assume $k$ is algebraically closed.
    The connectedness of $L$ is \cite[Corollary 3.11]{Steinberg-Torsioninreductivegroups}.
    Choose a Cartan subspace $\mathfrak{c}\subset \lieh_1$ containing $x$. 
    Since every stable element is semisimple (\cite[Lemma 13]{GrossLevyReederYu-GradingsPosRank}), hence contained in a Cartan subspace, and every two Cartan subspaces are conjugate, $\mathfrak{c}$ contains stable vectors for the $G$-action.
    But such vectors lie in $\Lie(L)_1$ and are also stable for the $(L^{\theta})^0$-action, proving the lemma.
\end{proof}

The \emph{discriminant polynomial} $\Delta_0 \in k[\lieh]^H$ is the image of $\prod_{\al \in \Phi} \alpha \in k[\mathfrak{t}]^W$ under the Chevalley restriction isomorphism $k[\mathfrak{t}]^W \xrightarrow{\sim} k[\lieh]^H$; it is independent of the choice of $\mathfrak{t}$.
For $x\in \lieh$ we have $\Delta_0(x) \neq 0$ if and only if $x$ is regular semisimple.
Let $\Delta\in k[V]^G$ be the restriction of $\Delta_0$ to $V$.

\begin{proposition}\label{proposition: regular semisimple iff stable}
    Suppose that $\theta$ is a stable $\mu_m$-action on $H$ and that $(G,V)$ is the associated Vinberg representation.
    Let $x\in V$.
    Then $x$ is stable if and only if $x$ is regular semisimple if and only if $\Delta(x) \neq 0$.
    Moreover, if $x$ is stable, then for every primitive $m$th root of unity $\zeta$, the induced automorphism $\theta(\zeta)$ on the maximal torus $Z_H(x)$ is $\Z$-regular and elliptic.
\end{proposition}
\begin{proof}
    If $x$ is stable, then $x$ is regular semisimple by \cite[Lemma 13]{GrossLevyReederYu-GradingsPosRank}. 
    Conversely, assume $x$ is regular semisimple. 
    By the previous lemma, the action of $\theta$ on the maximal torus $S = Z_H(x) \subset H$ is stable. 
    Since the conjugation action of $S$ on itself is trivial, this is equivalent to the statement that $(S^{\theta})^{\circ}$ is finite.
    This implies that $S^{\theta}$ and hence $Z_G(x) \subset Z_H(x)^{\theta} = S^{\theta}$ is finite. Since $x$ is semisimple, the orbit $G\cdot x$ is closed by a result Vinberg \cite[Prop. 3]{Vinberg-theweylgroupofgraded}.
    Combining the last two sentences shows that $x$ is indeed stable for the $G$-action.
    This argument also justifies the final sentence of the proposition.
\end{proof}

\subsection{Stable gradings over general fields}\label{subsec: stable gradings over general fields}

Let $H/k$ be a reductive group.
We make some remarks about when principal gradings on $H_{\bar{k}}$ descend to gradings on $H$; this section will not be used later in the paper.
Recall that given a $\mu_m$-grading $\theta\colon \mu_m \rightarrow \Aut_H$, there is an associated outer type $\tau\colon \mu_e\hookrightarrow \Out_H$.

\begin{lemma}\label{lemma: mum-action descends if and only if outer type descends}
    Suppose $H$ is a quasi-split reductive group over $k$.
    Let $\theta_0$ be a principal $\mu_{m,\bar{k}}$-action on $H_{\bar{k}}$ with outer type $\tau_0\colon \mu_{e,\bar{k}}\rightarrow \Out_{H_{\bar{k}}}$.
    Then there exists a $\mu_m$-action $\theta$ on $H$ such that $\theta_{\bar{k}}$ is $H_{\bar{k}}$-conjugate to $\theta_0$ if and only if $\tau_0$ is $\Gamma_k$-equivariant, i.e. descends to a (necessarily unique) homomorphism $\tau\colon \mu_e\rightarrow \Out_H$.
\end{lemma}
\begin{proof}
    If such a $\theta$ exists, then its outer type $\tau\colon \mu_e\rightarrow \Out_H$ must have the property that $\tau_{\bar{k}} = \tau_0$, since the outer types of two $H_{\bar{k}}$-conjugate $\mu_m$-actions must agree.
    Conversely, if $\tau_0$ descends to a homomorphism $\tau\colon \mu_e\rightarrow \Out_H$, we may take $\theta$ to be the split principal $\mu_m$-action associated to $\tau$ and a pinning of $H$.
\end{proof}

\begin{example}\label{example: field of definition gradings outer type 1 or 2}
    In the notation of the lemma, suppose that $\tau_0$ is trivial, i.e., $e=1$.
    Then every principal $\mu_m$-action on $H_{\bar{k}}$-can be defined over $k$.
    Similarly, if $e=2$, then every finite group scheme of order $e$ is isomorphic to $\mu_2\simeq \Z/2\Z$, so $\tau_0$ is automatically Galois equivariant.
\end{example}

Suppose that $H$ is quasi-split and $H_{\bar{k}}$ is adjoint and simple. 
    By the classification of Dynkin diagrams and their automorphism groups, we can describe $\Out_H$ very explicitly. 
    If $H$ is not of type $D_4$, then $\Out_H$ is trivial or isomorphic to $\Z/2\Z\simeq \mu_2$, and by Example \ref{example: field of definition gradings outer type 1 or 2} every principal $\mu_m$-action on $H_{\bar{k}}$ (in particular, every stable $\mu_m$-action) can be defined over $k$ by taking split principal gradings.

    It remains to consider $H$ of type $D_4$.
    Then $\Out_H$ is a finite \'etale group scheme with $\Out_{H, \bar{k}} \simeq S_3$.
    Since $H$ is quasi-split, the isomorphism class of $H$ determines a homomorphism $c\colon \Gamma_k \rightarrow S_3$, in other words, an \'etale $k$-algebra $K$ of degree $3$.
    The group scheme $\Out_{H}$ is a twist of the group $S_3$, with cocycle given by the composition $\ad\circ c\colon \Gamma_k \rightarrow \Aut(S_3)$, where $\ad\colon S_3\rightarrow \Aut(S_3)$ denotes conjugation; see \cite[Section 2 and Section 5.2]{GanSavin-TwistedBhargavaCubes} for these claims.
    From this description, it is easy to work out explicitly which injective homomorphisms $\mu_{3, \bar{k}} \rightarrow \Out_{H_{\bar{k}}}$ can be defined over $k$. 
    By postcomposing $c$ with the sign homomorphism $\mathrm{sgn}\colon S_3\rightarrow \{\pm 1\}$, we get a homomorphism $\Gamma_k \rightarrow \{\pm 1\}$, in other words a quadratic etale $k$-algebra $K_0$.
    Let $C\leq \Out_H$ be the unique subgroup scheme of order $3$.
    This is a twist of $\Z/3\Z$ with Galois action cut out by the quadratic character $\mathrm{sgn}\circ c$.
    It follows that $C\simeq \mu_3$ if and only if $\mathrm{sgn}\circ c$ corresponds to the Galois action on $\mu_3$, if and only if $K_0 \simeq \Q(\omega)\otimes_{\Q} k$ as $k$-algebras, where $\omega$ is a primitive third root of unity.
    So there exists an injective homomorphism $\tau\colon \mu_3\rightarrow \Out_H$ if and only if $K_0\simeq \Q(\omega) \otimes_{\Q} k$.

    Now suppose $k =\Q$ and let $H$ be the quasi-split adjoint group of type $D_4$ corresponding to the etale cubic algebra $K = \Q(\omega) \times \Q$.
    Then $K_0 = \Q(\omega)$ so by the previous paragraph there exists a (unique up to inversion) injective homomorphism $\tau\colon\mu_3\rightarrow  \Out_H$.
    Let $(T,B,\{E_i\})$ be a pinning of $H$ and let $\theta$ be the split principal $\mu_3$-action associated to this pinning and $\tau$.
    Let $(G,V)$ be the associated Vinberg representation. 
    Then $G\simeq \PGL_3$ and $V = \Sym^3(3)$ is the space of ternary cubic forms, with action given by $([g]\cdot F)(X,Y,Z) = F((X,Y,Z)\cdot g)/\det(g)$ where $[g]\in \PGL_3$ and $F\in \Q[X,Y,Z]_{\deg=3}$.
    This is the representation used in the study of $3$-descent on elliptic curves \cite{BS-3Selmer}.
    We have just exhibited it as a Vinberg representation arising from a quasi-split group of type $D_4$. 
    By Lemma \ref{lemma: mum-action descends if and only if outer type descends}, this $\mu_3$-action cannot be defined on the split form of type $D_4$, so the added generality of considering quasi-split forms is necessary to incorporate this natural example.

\subsection{Centralizers}

Let $H$ be a reductive group over $k$ with $\mu_m$-action $\theta\colon \mu_m \rightarrow \Aut_H$, giving rise to a $\Z/m\Z$-grading $\frakh = \oplus \frakh_i$ on the Lie algebra of $H$.
If $x\in \frakh$, write $\frak{z}_{\frakh_i}(x) = \frakz_{\frakh}(x)\cap \frakh_i$ for the intersection of the centralizer of $x$ in $\frakh$ with $\frakh_i$.
For later purposes, we record some lemmas concerning the dimensions of these centralizers.

\begin{lemma}\label{lem-pany}
If $x \in \frakh_{-1}$ and $n\in \Z$, then 
\begin{equation*}
    \dim \frakh_n - \dim \frakz_{\frakh_n}(x) = \dim \frakh_{1-n} - \dim \frakz_{\frakh_{1 - n}}(x).
\end{equation*}
If $x$ is additionally semisimple, then $\dim \frakh_n - \dim \frakz_{\frakh_n}(x)$ is independent of $n$.
\end{lemma}
\begin{proof}
The second statement is in \cite[Proposition 2.1]{Panyushev-Invarianttheorythetagroups}. The first statement follows from modifying the proof of that proposition; we provide the details.
Since there is a $\theta$-stable direct sum decomposition $\frakh = [\frakh,\frakh] \oplus \frak{z}(\frakh)$ with $[\frakh,\frakh]$ semisimple, and since the lemma is obvious when $\frakh$ is abelian, we may assume $\frakh$ is semisimple.
Let $\kappa$ be the Killing form on $\frakh$.
If $x\in \frak{h}$ let $\kappa_x$ be the bilinear form defined by $\kappa_x(y, z) = \kappa(x, [yz]) = \kappa([x,y],z)$. 
Since $\kappa$ is nondegenerate, the radical of $\kappa_x$ equals the centralizer $\frakz_\frakh(x)$. 
Moreover, $\kappa(\frakh_i , \frakh_j) = 0$ unless $i+ j =0$, so $\kappa_x(\frakh_i, \frakh_j) = 0$ unless $i + j = 1$. 
Therefore $\kappa_x$ induces a perfect pairing between $\frakh_n/\frakz_{\frakh_n}(x)$ and $\frakh_{1-n}/\frakz_{\frakh_{1-n}}(x)$ for each $n$, so $\dim \frakh_n/\frakz_{\frakh_n}(x) = \dim \frakh_{1 - n}/\frakz_{\frakh_{1 - n}}(x)$.
\end{proof}

\begin{lemma}\label{lem-ef-centralizer} 
Let $e \in \frakh_1$ be a nilpotent element, and let $(e, h, f)$ be a normal $\fraksl_2$-triple containing $e$. Then 
\begin{equation*}
\dim \frakz_{\frakh_0}(e) = \dim \frakz_{\frakh_0}(f)
\end{equation*}
\end{lemma}

\begin{proof}
The triple $(e,h,f)$ determines a homomorphism $\frak{sl}_2\rightarrow \frakh$ of Lie algebras.
The restriction of the adjoint action of $\frakh$ on itself to $\frak{sl}_2$ defines an action of $\frak{sl}_2$ on $\frakh$.
Using the representation theory of $\fraksl_2$, we see that 
\begin{equation*}
\frakh = [\frakh, e] \oplus \frakz_{\frakh}(f).
\end{equation*}
Since the summands on the right hand side are $\theta$-stable, we can intersect both sides of this equality with $\frakh_1$ and get
\begin{equation*}
\frakh_1 = [\frakh_0, e] \oplus \frakz_{\frakh_1}(f).
\end{equation*}
Taking dimensions, we have
\begin{equation*}
\dim \frakh_1 = \dim [\frakh_0,e] + \dim\frakz_{\frakh_1}(f) =  \dim \frakh_0 - \dim \frakz_{\frakh_0}(e) + \dim \frakz_{\frakh_1}(f). 
\end{equation*}
The result then follows directly from Lemma \ref{lem-pany}.
\end{proof}

\begin{lemma}\label{lemma: dim centralizers vs GIT quotient}
    Assume that $\theta$ is stable and let $x\in \frakh_{-1}$.
    Then 
    \begin{equation}\label{eqn-cent-quotient}
    \dim \frakz_{\frakh_1}(x) - \dim (V\GIT G) = \dim \frakz_{\frakh_0}(x).
\end{equation}
    In particular, 
    \begin{equation}\label{eqn-quotient}
    \dim V = \dim\frakh_0 + \dim (V\GIT G).
\end{equation}
\end{lemma}
\begin{proof}
   We first prove \eqref{eqn-quotient}.
    Let $v\in V$ be a stable element.
    Then $v$ is regular semisimple and $\frakz_{\frakh_0}(v)=\{0\}$, by Proposition \ref{proposition: regular semisimple iff stable}.
    Moreover $\frak{c} = \frakz_{\frakh_1}(v)$ is a Cartan subspace of $V = \frakh_1$.
    By Proposition \ref{proposition: basic invariant theory of vinberg representation full generality}, $\dim(V\GIT G) = \dim \frak{c}$.
    By Lemma \ref{lem-pany}, $\dim \frakh_1 - \dim \frakz_{\frakh_1}(v) = \dim \frakh_0 = \dim \frakz_{\frakh_0}(v)$.
    Combining the last three sentences shows \eqref{eqn-quotient}.
    Lemmas \ref{lem-pany} and \eqref{eqn-quotient} then shows \eqref{eqn-cent-quotient}.
\end{proof}

\subsection{The nilpotent cone}\label{subsec: nilpotent cone}

Of particular interest to us are the nilpotent elements of a graded Lie algebra, since they will be used to construct transverse slices and families of curves in \S\ref{section: constructing families of curves}.
We prove some basic facts about these nilpotent elements here.
Let $H$ be a reductive group over $k$ with $\mu_m$-action $\theta$.
Recall the GIT quotient $\pi\colon V\rightarrow V\GIT G$ from \eqref{equation: GIT quotient map}.
Let $\mathcal{N}$ be the (scheme-theoretic) fiber of $\pi$ above $0$.
For every field extension $K/k$, $\mathcal{N}(K)$ equals the subset of nilpotent elements of $V(K)$; this follows from Proposition \ref{proposition: basic invariant theory of vinberg representation full generality}.
We are especially interested in the $G(k)$-action on $\mathcal{N}(k)$.
If $k$ is algebraically closed, then there are finitely many orbits for this action.
These nilpotent orbits have been extensively studied, especialy when $m=1$; see \cite{CollingwoodMcGovern-nilpotentorbits} for the $m=1$ case and \cite{Vinberg-nilpotent, Vinberg-classificationnilpotentelementsthetagroups, deGraaf-nilporbitsthetagroups} for the general case.

If $\calO_1, \O_2\subset \mathcal{N}$ are (geometric) $G$-orbits, then they are locally closed subvarieties of  $\mathcal{N}$; we write $\O_1\preceq \O_2$ if $\O_1$ is contained in the Zariski closure of $\O_2$. 
This defines a partial order on the finite set of $G$-orbits on $\mathcal{N}$.
Note that $\dim(\calO_1)< \dim(\calO_2)$ if $\calO_1\preceq \calO_2$ and $\calO_1\neq \O_2$.

We assume for the remainder of this section that $\theta$ is stable. 

\begin{lemma}\label{lemma: dimension nilpotent cone}
$\dim \mathcal{N} = \dim G$. 
\end{lemma}
\begin{proof}
Since $\pi\colon V \rightarrow V \GIT G$ is flat, $\dim \mathcal{N} = \dim V -\dim(V\GIT G)$.
By Lemma \ref{lemma: dim centralizers vs GIT quotient}, $\dim(V \GIT G) = \dim V - \dim \lieh_0$.
So $\dim \mathcal{N} = \dim \lieh_0  = \dim G$.
\end{proof}

\begin{lemma}\label{lemma: codimension nilpotent orbit is centralizer}
If $e\in \mathcal{N}$, the $G$-orbit of $e$ (a locally closed subset of $\mathcal{N}$) has codimension $\dim \mathfrak{z}_{\lieh_0}(e)$. 
\end{lemma}
\begin{proof}
This follows from the equalities $\dim G\cdot e = \dim G  - \dim Z_{G}(e) = \dim \mathcal{N} -\dim \mathfrak{z}_{\lieh_0}(e)$. 
(Recall from \S\ref{subsec: notation} that $Z_G(e)$ has Lie algebra $\frakz_{\lieh_0}(e)$.)
\end{proof}

Recall that an element $x\in \frakh$ is called regular if $\dim \frakz_{\frakh}(x)$ equals the reductive rank of $H$.

\begin{lemma}\label{lemma: regular nilpotent defines top component of N}
Let $e\in V$ be regular nilpotent.
\begin{enumerate}
    \item If $H$ is adjoint, then $Z_{H^{\theta}}(e) = \{1\}$ and $Z_G(e)=\{1\}$.
    \item If $H$ is adjoint and $e'\in V$ is another regular nilpotent, there exists a unique $h\in H^{\theta}(k)$ such that $h\cdot e = e'$.
    \item The closure of the $G$-orbit of $e$ is an irreducible component of $\mathcal{N}$.
\end{enumerate}
\end{lemma}
\begin{proof} 
\begin{enumerate}
\item The closed subgroup $Z_{H^{\theta}}(e)$ of $H^{\theta}\subset H$ has Lie algebra $\frakz_{\frakh_0}(e)$.
Since $\theta$ is stable, Lemma \ref{lemma: dim centralizers vs GIT quotient} and \cite[Section 3, Theorem 3.3]{Panyushev-Invarianttheorythetagroups} show, in the notation of that paper, that $\dim \frakz_{\frakh_0}(e)= k_0$ and $k_0=0$.
Therefore $\frakz_{\frakh_0}(e) = \{0\}$, so $Z_{H^{\theta}}(e)$ is finite. On the other hand, by combining \cite[Lemma 3.7.3]{CollingwoodMcGovern-nilpotentorbits} and the fact that the centralizer of an $\frak{sl}_2$-triple containing $e$ is trivial, it follows that the centralizer of $e$ in $H$ is a unipotent group.
So $Z_{H^{\theta}}(e)$ is a finite subgroup of the unipotent group $Z_H(e)$, hence trivial. 

\item Complete $e$ and $e'$ to normal $\fraksl_2$-triples $(e,x,f)$ and $(e',x',f')$.
By \cite[Section 5, Corollaries 5.2 and 5.5]{Kostant-principalthreedimensional}, there exists a unique $h\in H(\bar{k})$ such that $\Ad(h)\cdot (e,x,f) = (e',x',f')$. 
By uniqueness, $h$ is necessarily defined over $k$.
Since $\theta(h)$ also conjugates $(e,x,f)$ to $(e',x',f')$, again by uniqueness we have $\theta(h)=h$ hence $h \in H^{\theta}$.
So $h\cdot e = e'$ and by Part 1 such a $h\in H^{\theta}$ is unique.

\item After possibly replacing $H$ with the quotient by its center, we may assume that $H$ is adjoint.
Let $\bar{\mathcal{O}}$ be the closure of $\mathcal{O} = G\cdot e$. By the first part and Lemma \ref{lemma: codimension nilpotent orbit is centralizer}, $\dim \bar{\mathcal{O}} = \dim G$.
So $\mathcal{O}$ is an irreducible closed subset of $\mathcal{N}$ of dimension $\dim G = \dim \mathcal{N}$ (Lemma \ref{lemma: codimension nilpotent orbit is centralizer}), hence must be an irreducible component.
\end{enumerate}
\end{proof}

\begin{corollary}\label{corollary:Gaction_faithful}
    If $H$ is adjoint, then the $G$-action on $V$ is faithful.
\end{corollary}
\begin{proof}
    We may assume $k$ is algebraically closed.
    By Theorem \ref{theorem: characterization stable gradings} and Lemma \ref{lemma: split principal if and only if regular nilpotent exists}, there exists a regular nilpotent $e\in V$.
    The corollary then follows from Part 1 of Lemma \ref{lemma: regular nilpotent defines top component of N}.
\end{proof}

It follows that each regular nilpotent $e\in \lieh_1$ defines an irreducible component $\overbar{G\cdot e}$ of $\mathcal{N}$. 
If $m=2$, then every irreducible component is of this form \cite[Theorem 6]{KostantRallis-Orbitsrepresentationssymmetrisspaces}. For general $m\geq 3$, this need not be true: counterexamples may be found in the class of Coxeter gradings of Example \ref{ex-cox}, see \cite[Example 4.4]{Panyushev-Invarianttheorythetagroups}.

The next proposition shows that every nilpotent orbit in $\mathcal{N}_{\bar{k}}$ has a representative defined over $k$, at least when $H$ is split.

\begin{proposition}\label{proposition: nilpotents on split groups descend}
    Suppose that $H$ is split and let $e\in V(\bar{k})$ be nilpotent. 
    Then there exists a nilpotent $e_0 \in V$ that is $G(\bar{k})$-conjugate to $e$.
\end{proposition}
\begin{proof}
    Complete $e$ to a normal $\liesl_2$-triple $(e,h,f)$ in $\frakh_{\bar{k}}$.
    Let $(T,B,\{E_i\})$ be a pinning of $H$. 
    Let $\frak{t}\subset \frak{h}$ be the Lie algebra of $T$.
    Since $h$ is semisimple, we may assume after $G(\bar{k})$-conjugation that it lies in $\liet_{\bar{k}}$.
    The representation theory of $\frak{sl}_2$ shows that  $\alpha(h) \in \Z$ for every root $\alpha\in \Phi$.
    Since $T$ is split, it follows that $h \in \frak{t}$. 
    Using \cite[Proposition 10]{deGraaf-nilporbitsthetagroups}, there exists an open dense subset $U$ of the affine space $\{x\in \lieh_1\colon [h,x] = 2x \}$ with the property that every $e' \in U(\bar{k})$ is $G(\bar{k})$-conjugate to $e$.
    Since $k$ is of characteristic zero and hence infinite, $U(k)$ is nonempty and any element $e_0\in U(k)$ satisfies the conclusion of the proposition.
\end{proof}

\subsection{Invariant polynomials}\label{subsec: invariant polynomials}

Let $H$ be a reductive group over $k$ with a $\mu_m$-grading $\theta$.
Recall that $B = V\GIT G$ denotes the spectrum of the ring of invariant polynomials $k[V]^G$.
This ring is always a polynomial ring (Proposition \ref{proposition: basic invariant theory of vinberg representation full generality}).
In the case that $\theta$ is stable, Panyushev has determined the degrees of the generators of $k[V]^G$.

First consider the (ungraded) adjoint quotient $\frakh \GIT H = \Spec(k[\frakh]^H)$.
Let $F_1, \dots, F_{r}$ be homogeneous algebraically independent generators of $k[\frakh]^H$.
By diagonalizing the $\theta$-action on elements of $k[\frakh]^H$ of fixed degree, we can choose the $F_i$ so that there exist elements $e_i\in \Z/m\Z$ with the property that $\theta(\zeta)(F_i) = \zeta^{e_i} F_i$ for all $\zeta\in \mu_m(\bar{k})$.
The multiset $\{e_1,\dots, e_{r}\}$ only depends on the outer type of $\theta$, and equals $\{0, \dots, 0\}$ if $\theta$ is inner.

\begin{proposition}\label{proposition: invariant polys panyushev}
    Suppose that $\theta$ is stable.
    Then the restriction homomorphism $k[\frakh]^H \rightarrow k[V]^G$ is surjective and $k[V]^G$ is generated by $\bar{F}_i$ for all $F_i$ satisfying $\deg(F_i) +e_i\equiv 0 \mod m$.
\end{proposition}
\begin{proof}
    Since every stable $\theta$ is $N$-regular in the sense of Panyushev (by combining Theorem \ref{theorem: characterization stable gradings} and Lemma \ref{lemma: split principal if and only if regular nilpotent exists}), this follows from \cite[Theorem 3.5(i)]{Panyushev-Invarianttheorythetagroups}.
\end{proof}

Rephrasing the above proposition, the canonical morphism $V\GIT G\rightarrow \frakh \GIT H$ is a closed immersion of affine spaces and is cut out by setting $F_i$ equal to zero whenever $\deg(F_i) +e_i\not\equiv 0 \mod m$.

\begin{example}
    Assume $\theta$ is inner.
    Then $k[V]^G$ is generated by the restriction of those invariant polynomials whose degree is divisible by $m$.
\end{example}

\begin{example}
    Assume $H$ is of type $E_6$ and $\theta$ has nontrivial outer type $\tau\colon \mu_2\hookrightarrow \Out_H$, so $m$ is even.
    The ring $k[\frakh]^H$ has generators $F_1, \dots, F_6$ of degree $2,5,8,6,9,12$.
    We claim that $\theta(\zeta)(F_i) = (-1)^iF_i$ if $\zeta\in \mu_m(\bar{k})$ has order $m$.
    Indeed, to prove this, we may assume $k = \bar{k}$, that $m=2$ and that $\theta$ equals the unique nontrivial pinned automorphism with respect to some pinning of $H$.
    This then follows from an explicit calculation, see \cite[Proposition 2.5(2)]{Laga-F4paper}.
\end{example}

\section{Constructing transverse slices}\label{section: constructing families of curves}

Following Slodowy and Thorne, we construct for each nilpotent $e\in \frakh_1$ a subvariety $X_e\subset \frakh_1$ such that the morphism $X_e\rightarrow B$ is a transverse slice to the $G$-action on $e$.
If the $G$-orbit of $e$ has codimension one in $\mathcal{N}$, $X_e\rightarrow B$ is a family of curves.
We make some observations about the families of curves arising in this way, before classifying such families arising from subregular nilpotent elements in Section \ref{section:subregular curves}. 
As before, let $k$ be a field of characteristic zero.

\subsection{Transverse slices}

Let $L$ be an algebraic group over $k$ acting on a variety $Y/k$.
Let $y\in Y(k)$.
Following Slodowy \cite[\S5.1]{Slodowy-simplesingularitiesalggroups}, we say a locally closed subvariety $S\subset Y$ is a \textit{transverse slice} to the $L$-action on $Y$ at $y$ if 
\begin{itemize}
    \item $y\in S(k)$,
    \item the action map $L\times S\rightarrow Y$ is smooth,
\end{itemize}
and the dimension of $S$ is minimal among locally closed subvarieties satisfying these two properties.

If $Y$ is smooth, then such a transverse slice always exists, and its dimension equals the codimension of the $L$-orbit of $y$; see \cite[Lemma 1 and Remark 2]{Slodowy-simplesingularitiesalggroups}.
Moreover the tangent spaces satisfy $T_y(L\cdot y) \oplus T_y S = T_yY$.

\subsection{Diagonalizable actions on affine spaces}\label{subsec: diagonalizable actions}

Let $T$ be a closed subgroup of $\G_m^n$ for some $n\geq 1$. 
Let $X^*(T_{\bar{k}}) = X^*(T)$ be its character group.
Let $V$ be a $k$-vector space equipped with a linear $T$-action, i.e., a homomorphism of algebraic groups $\lambda\colon T\rightarrow \GL(V)$.
If $w\in X^*(T)$ define the weight space $V_w = \{v\in V\colon \lambda(t)(v) = w(t)v \text{ for all }t\in T(\bar{k})\}$.
We will often make use of the following standard lemma \cite[Chapter 3, Section 8.4]{Borel-linearalgebraicgroups}.
\begin{lemma}\label{lemma: representations diagonalizable groups}
    $V$ is a direct sum of weight spaces: there exists a basis $v_1,\dots,v_n$ of $V$ and elements $w_1, \dots, w_n \in X^*(T)$ such that $\lambda(t)(v_i) = w_i(t) v_i$ for all $t\in T(\bar{k})$.
\end{lemma}

The elements $w_1, \dots, w_n$ are called the \textit{weights} of the $T$-action.
For example, if $T = \G_m$, then $X^*(T) = \Z$ hence the weights are simply integers.
If $T = \mu_m$, then the weights are elements of $X^*(T) = \Z/m\Z$.

If $V,W$ are two vector spaces with linear $T$-action, we say an algebraic morphism $f\colon V\rightarrow W$ is \emph{quasi-homogeneous of weight} $(d_1,\dots,d_m; w_1 ,\dots ; w_n)$ if $f$ is $T$-equivariant, $V$ has weights $w_1,\dots,w_n$, and $W$ has weights $d_1,\dots,d_m$. (Here we count weights with multiplicity.)
The next lemma is proved in \cite[Section 8.1, Lemma 3]{Slodowy-simplesingularitiesalggroups} and will be useful in Section \ref{section-families-of-curves}.

\begin{lemma}\label{lemma: Gm-morphism same positive weights is isomorphism}
    In the above notation, suppose $T = \G_m$. Suppose we are given a morphism $f\colon V\rightarrow W$ that is quasi-homogeneous of weight $(d_1, \dots, d_n ; d_1, \dots, d_n)$, each $d_i$ is positive, and the central fiber $f^{-1}(0)$ is zero dimensional.
    Then $f$ is an isomorphism.
\end{lemma}

\subsection{Graded Slodowy slices}\label{subsec: graded Slodowy slices}

Let $H/k$ be a reductive group with $\mu_m$-action $\theta\colon \mu_m \rightarrow \Aut_H$, giving rise to the Vinberg representation $(G,V)$.
Let $e\in V(k)$ be a nonzero nilpotent element. 
We explain how to construct an explicit transverse slice to the $G$-action at $e$.

By the graded Jacobson--Morozov Theorem (Lemma \ref{lemma: graded Jacobson-Morozov}), there exist $h\in \lieh_0$ and $f\in \lieh_{-1}$ such that $(e,h,f)$ is an $\liesl_2$-triple.
The centralizer $\frakz_{\lieh}(f)$ is a $\theta$-stable subalgebra of $\lieh$.
Define the affine linear subspaces
\begin{align}\label{equation: transverse slodowy slices}
    S_e &:= (e + \frakz_{\lieh}(f))  \subset \lieh, \\
    X_e &:= (e + \frakz_{\lieh}(f)) \cap V = e +\frakz_{\frakh_1}(f) \subset V.
\end{align}
Write $\varphi\colon X_e\rightarrow B$ for the restriction of $\pi \colon V \rightarrow V \GIT G = B$ to $X_e$.
\begin{remark}
    The construction of $S_e$ and $X_e$ depends on a choice of $h,f$. Since any two normal $\liesl_2$-triples $(e,h,f)$ and $(e,h',f')$ are $Z_G(e)$-conjugate, the respective slices $X_e$ and $X_e'$ are $Z_G(e)$-conjugate, so this dependence is usually harmless and we omit it from the notation.
\end{remark}
Slodowy has shown that $S_e$ is a transverse slice to the $H$-action on $\frakh$ at $e$.
The next theorem generalizes Slodowy's result to graded Lie algebras.
It has been proved by Thorne in the $m=2$ case \cite[Proposition 3.4]{Thorne-thesis}.

\begin{theorem}\label{theorem: slodowy slice is a transverse slice}
In the above notation, $X_e$ is a transverse slice to the $G$-action on $V$ at $e$.
More precisely,
\begin{enumerate}
    \item The multiplication map $G\times X_e\rightarrow V$ is smooth;
    \item $\dim X_e +\dim(G\cdot e) = \dim V$; and
    \item The map $\varphi\colon X_e\rightarrow B$ is faithfully flat.
\end{enumerate}
\end{theorem}

The proof of Theorem \ref{theorem: slodowy slice is a transverse slice} follows from adapting arguments of Thorne \cite[\S3]{Thorne-thesis} in a straightforward manner, which are itself based on arguments of Slodowy \cite[\S7.4]{Slodowy-simplesingularitiesalggroups}.
We start by defining important $\G_m\times \mu_m$-actions on $S_e$ and $\G_m$-actions on $X_e$ that will be useful in the proof and in Section \ref{section:subregular curves}.

Let $\frak{m}$ be the subalgebra of $\frakh$ spanned by $\{e,h,f\}$.
The adjoint action of $\frakh$ on itself restricts to an action of $\frak{m} \simeq \liesl_2$ on $\frakh$.
Let $\lambda\colon \G_m\rightarrow H$ be the unique homomorphism whose derivative at $1$ equals $h$.
Then $\Ad(\lambda)\colon \G_m\rightarrow \GL(\frakh)$ defines a linear $\G_m$-action on $\frakh$.
On the other hand, $\theta$ is a $\mu_m$-action on $\frakh$ that commutes with $\Ad(\lambda)$ since $h \in \frakh_0$.
These combine to a $\G_m\times \mu_m$-action $\Ad(\lambda) \times \theta$ on $\frakh$.
This action preserves $\frakz_{\frakh}(f)$, so by Lemma \ref{lemma: representations diagonalizable groups}, there exists a basis $v_1, \dots, v_n$ of $\frakz_{\frakh}(f)$ and elements $w_1, \dots, w_n$ of $X^*(T) = \Z \times (\Z/m\Z)$ such that $v_i$ has weight $w_i$.
Write $w_i = (n_i, k_i)$, where $n_i \in \Z$ and $k_i\in \Z/m\Z$.
If $x\in S_e$ then $x = e + \sum c_i v_i$ for some $c_i \in k$, and we have 
\begin{align*}
    \Ad(\lambda(t))(x) &= t^2 e + \sum c_i t^{n_i} v_i, \\
    \theta(\zeta)(x)&=\zeta e  + \sum c_i \zeta^{k_i}v_i.
\end{align*}
These expressions tell us how to modify $\Ad(\lambda)\times \theta$ so that $e$ becomes a fixed point for the new action. 
Define a $\G_m$-action $\rho$ on $\frakh$ by the formula $\rho(t)(v) = t^2 \Ad(\lambda(t^{-1}))(v)$.
Define a $\mu_m$-action $\sigma$ on $\frakh$ by the formula $\sigma(\zeta)(v) = \zeta \theta(\zeta^{-1})(v)$.
These combine to a $\G_m\times \mu_m$-action $\rho\times \sigma$ on $\frakh$. 
This action preserves $S_e$ and $e$.
Viewing $S_e$ as an affine space with origin $e$, $\rho\times \sigma$ is a linear $\G_m\times \mu_m$-action and its weights are $(2-n_i,1-k_i)$ for $i=1, \dots,n$.
The next lemma shows that $e$ is the unique fixed point for this action:
\begin{lemma}\label{lemma: weights of rho are positive}
    The weights of the $\rho$-action on $S_e$ are $\geq 2$.
\end{lemma}
\begin{proof}
    In the above notation, this is equivalent to showing that $n_i\leq 0$ for all $i$.
    By the representation theory of $\fraksl_2$, the subspace of $\frakh$ generated by  $\{\ad(e)^k(v_i)\colon k\geq 0\}$ is an irreducible representation of $\frak{m}$.
    If it has dimension $m_i$, then $n_i = 1-m_i$. 
    Since $m_i\geq 1$, $n_i\leq 0$.
\end{proof}

The scalar $\G_m$-action on $\frakh$ induces a $\G_m$-action on $\frakh\GIT H$, which we simply denote by $t\cdot b$.
The $\mu_m$-action on $H$ induces one on $\frakh \GIT H$, which we also denote by $\theta$. (This $\mu_m$-action only depends on the outer type of $\theta$ and has been studied in \S\ref{subsec: invariant polynomials}.)
Define a $\G_m\times \mu_m$-action $\rho\times \sigma$ on $\frakh \GIT H$ by the formulae $\rho(t)(b) = t^2 \cdot b$ and $\sigma(\zeta)(b) = \zeta \cdot \theta(\zeta^{-1})(b)$.
Then the GIT quotient morphism $p\colon \frakh \rightarrow \frakh \GIT H$ is equivariant with respect to the action $\rho \times \sigma$ on both sides.

We may now restrict all these considerations to $V$ and $X_e$.
By definition, the fixed-point locus of $S_e$ under $\sigma$ is $X_e$, so $\rho$ restricts to a $\G_m$-action on $X_e$.
Proposition \ref{proposition: invariant polys panyushev} can be reinterpreted as saying that the inclusion $V\subset \frakh$ induces a closed immersion $V\GIT G\hookrightarrow \frakh \GIT H$ whose image equals the $\sigma$-fixed point locus of $\frakh \GIT H$, so again $\rho$ restricts to a $\G_m$-action on $V\GIT G$.

\begin{proof}[Proof of Theorem \ref{theorem: slodowy slice is a transverse slice}]
    \begin{enumerate}
        \item The differential of the action map $\mu\colon G\times X_e\rightarrow V$ at $(1,e)$ equals, after identifying $T_e(X_e)$ with $\frakz_{\frakh}(f)\cap V$ and $T_eV$ with $V$, the linear map $\lieg\times \frakz_{\frakh}(f)\rightarrow V$ given by $(x,z)\mapsto [x,e] + z$. The representation theory of $\fraksl_2$ implies $\frakh = [\frakh,e] \oplus \frakz_{\frakh}(f)$.
        Since both summands are $\theta$-stable, we have
        \begin{align}\label{equation: decomposition V into centralizer and orbit}
            V=[\lieg,e] \oplus (\frakz_{\frakh}(f)\cap V),
        \end{align}
        so this differential is surjective. 
        Therefore $\mu$ is smooth at $(1,e)$ \cite[Proposition III.10.4]{hartshorne-AGbook}.
        Moreover, $\mu(\gamma g,x) = \Ad(\gamma)\cdot \mu(g,x)$ for all $\gamma,g\in G$ and $x\in V$, so $\mu$ is smooth at every point of $G\times\{e\}$.
        Define a $\G_m$-action on $G\times X_e$ via $t\cdot (g,x) = (g\lambda(t), \rho(t)x)$ and on $V$ via $t\cdot v= t^2 v$.
        Then $\mu$ is equivariant with respect to these actions.
        The open subset of points in $G\times X_e$ where $\mu$ is smooth contains $G\times \{e\}$ and is $\G_m$-equivariant.
        By Lemma \ref{lemma: weights of rho are positive}, the $\G_m$-action $\rho$ on $X_e$ contracts every point to $e$, so this open subset must equal the whole of $G\times X_e$.
        Hence $\mu$ is smooth everywhere.
        \item We have $\dim X_e + \dim(G\cdot e) = \dim (\frakz_{\frakh}(f)\cap V) + \dim [\frakg,e]$, which equals $\dim V$ by the decomposition \eqref{equation: decomposition V into centralizer and orbit}.
        \item We first prove that $\varphi\colon X_e\rightarrow B$ is flat.
        It suffices to prove that $G\times X_e\rightarrow B, (g,x)\mapsto \varphi(x)$ is flat. This map is the composition $G\times X_e\xrightarrow{\mu} V\xrightarrow{\pi} B$. 
        Since both $\mu$ and $\pi$ are flat (the former is even smooth by Part 1 of this theorem, the latter is flat by Proposition \ref{proposition: basic invariant theory of vinberg representation full generality}), the flatness of $\varphi$ follows.
        To prove that $\varphi$ is faithfully flat, observe that its image is an open, $\G_m$-stable subset of $B$ that contains the central point $0\in B$.
        Since the $\G_m$-action on $B$ has positive weights, the only such open subset is $B$ itself, so $\varphi$ is surjective.
    \end{enumerate}
\end{proof}

An important and well understood special case of the above construction is when $e$ is regular. 

\begin{proposition}\label{proposition: transverse slice regular nilpotent is isomorphism}
    Assume that $e$ is regular nilpotent. 
    Then $\varphi\colon X_e\rightarrow B$ is an isomorphism.
\end{proposition}
\begin{proof}
    This is \cite[Theorem 3.5(ii)]{Panyushev-Invarianttheorythetagroups}.
\end{proof}

In other words, the transverse slice $X_e$ constructed from a regular $e$ defines a section of the quotient map $\pi\colon V\rightarrow B$.
Such a section is called a Kostant section (or sometimes a Kostant--Weierstrass slice).
If $H$ is quasi-split and $\theta$ is split principal, then regular nilpotent elements $e\in V$ always exist (Lemma \ref{lemma: split principal if and only if regular nilpotent exists}).

\subsection{Families of curves}\label{section-families-of-curves}

Assume $\theta$ is stable and $e\in V$ is a nilpotent element, giving rise to a transverse slice $\varphi\colon X_e\rightarrow B$ associated to some choice of $\fraksl_2$-triple $(e,h,f)$ containing $e$.
We will be interested in the situation where $X_e\rightarrow B$ is a family of curves, i.e. the fibers have dimension $1$.
This is equivalent, by Lemma \ref{lemma: codimension nilpotent orbit is centralizer}, to the condition $\dim Z_G(e)=\dim\frakz_{\frakg}(e)=1$.
So assume the latter condition for the remainder of this section.
We make some general observations about such families of curves.
We already know by Theorem \ref{theorem: slodowy slice is a transverse slice} that $\varphi\colon X_e\rightarrow B$ is affine, surjective, with one-dimensional fibers, and that the action map $G\times X_e\rightarrow V$ is smooth.
Write $C = \varphi^{-1}(0)$ for the central fiber, a singular affine curve.

The next lemma determines which nilpotent orbits appear in $C$: it is exactly those which lie above $e$ in the closure partial ordering.

\begin{lemma}
    Let $\calO$ be a nilpotent orbit in $V$. Then $\calO\cap C \neq \varnothing$ if and only if the closure $\bar{\calO}$ of $\calO$ contains $e$.
\end{lemma}
\begin{proof}
    Suppose that $\calO \cap C \neq \varnothing$.
    Then $\bar{\calO}\cap C$ is a closed nonempty $\G_m$-stable subset of $C$, so must contain $e$, by the same logic as the proof of Lemma \ref{lemma: fibers are connected}.
    So $e\in \bar{\calO}$.
    Conversely, suppose that $e\in \bar{\calO}$.
    Then every $G$-stable open subset of $V$ that contains $e$ also contains $\calO$.
    Since the morphism $m\colon G\times X_e \rightarrow V$ is smooth, its image is open, so applying the previous sentence to the image of this morphism shows that $X_e\cap \calO \neq \varnothing$, so $C\cap \calO \neq \varnothing$.
\end{proof}

\begin{lemma}
    $(G\cdot e) \cap  X_e = \{e\}$.
\end{lemma}
\begin{proof}
    This follows from the fact that $[\frakg,e] \oplus T_e(X_e) = V$.
    Indeed, $(G\cdot e)\cap X_e$ is a closed $\G_m$-invariant subset of $C$, so either $\{e\}$ or a union of irreducible components of $C$.
    If $C_i\subset G\cdot e$, then $T_e C_i \subset [\frakg,e]\cap T_e(X_e)= \{0\}$, so $T_eC_i = 0$, contradiction.
\end{proof}

Let $\mathrm{Orb}^e$ be the set of nilpotent $G$-orbits $\calO\subset \mathcal{N}$ different from $G\cdot e$ with the property that $e\in \bar{\O}$.
In the notation of Section \ref{subsec: nilpotent cone}, $G\cdot e\preceq \O$ and $G\cdot e\neq \O$, so $\dim(G\cdot e) < \dim(\calO)$.
Since the codimension of $G\cdot e\subset \mathcal{N}$ is $1$, $\dim(\calO)=\dim(\mathcal{N})$, in other words $\bar{\O}$ is an irreducible component of $\mathcal{N}$.

The above two lemmas show that the association $x\mapsto G\cdot x$ induces a surjective map 
\[
\Psi\colon \{\text{irreducible components of }C\}\rightarrow \mathrm{Orb}^e.
\]


Let $\mathrm{Orb}^e = \{\O_1, \dots, \O_n\}$, and choose an orbit representative $E_i \in \O_i$ for each $i$. Let $\kappa_i = X_{E_i}$ be a transverse slice to the $G$-action on $V$ at $E_i$ constructed using an $\fraksl_2$-triple as in Section \ref{subsec: graded Slodowy slices}. Write $\varphi_i: \kappa_i \to B$ for the restriction of $\pi$ to $\kappa_i$.
The map $\varphi_i$ is a faithfully flat morphism between affine spaces of the same dimension.
Moreover there is a $\G_m$-action $\rho$ on the source and the target and the morphism is $\G_m$-equivariant.

\begin{definition}\label{definition: reduced orbit}
    We say $E_i$ is \emph{reduced} if $\varphi_i\colon \kappa_i\rightarrow B$ is an isomorphism.
\end{definition}
By Lemma \ref{lemma: Gm-morphism same positive weights is isomorphism}, $E_i$ is reduced if and only if the $\G_m$-weights on $\kappa_i$ and $B$ agree, if and only if the central fiber $\varphi_i^{-1}(0)$ is a reduced scheme.
If $E_i$ is regular, then $E_i$ is reduced by Proposition \ref{proposition: transverse slice regular nilpotent is isomorphism}.

\begin{lemma}\label{lemma: irred component reduced iff corr nilpotent is reduced}
    Let $C_i$ be an irreducible component of $C$.
    Then $C_i$ is a reduced scheme if and only if $\Psi(C_i)$ is reduced.
\end{lemma}
\begin{proof}
    This follows from the fact that $G\times X_e\rightarrow V$ and $G\times \kappa_i\rightarrow V$ are smooth.
    Indeed, the orbit $\O_i = \psi(C_i)$ admits smooth surjective maps $G\times (C_i\setminus \{e\})\rightarrow \O_i$ and $G\times (\kappa_{i,0})\rightarrow \O_i$.
    Therefore $C_i$ is reduced if and only if $\O_i$ is reduced if and only if $\kappa_{i,0}$ is.
    The latter condition is equivalent (using Lemma \ref{lemma: Gm-morphism same positive weights is isomorphism}) to $E_i$ being reduced.
\end{proof}

We say $e$ is \emph{good} if all the $E_i$ are reduced. 

\begin{corollary}\label{corollary: good e has central fiber reduced and unique singular point}
    If $e$ is good, then $C$ is reduced and $e$ is the unique singular point.
\end{corollary}
\begin{proof}
    By Lemma \ref{lemma: irred component reduced iff corr nilpotent is reduced}, $C$ is generically reduced.
    Since $C$ is a fiber of a flat morphism between smooth varieties, $C$ is a local complete intersection. Hence $C$ is reduced.
    Since the transverse slice $\varphi \colon X_e\rightarrow B$ at $e$ has relative dimension $1$, the map $\pi\colon V\rightarrow V\GIT G$ is not smooth at $e$, so $\varphi$ is not smooth at $e$, so $e$ is a singular point of $C$.
    It follows that the singular locus of $C$ is a closed finite $\G_m$-invariant subset of $C$ containing $e$, so must equal $\{e\}$.
\end{proof}

\begin{lemma}\label{lemma: fibers are connected}
    Suppose that $e$ is good. 
    Then the fibers of $\varphi$ are geometrically reduced and connected.
\end{lemma}
\begin{proof}
    We may assume $k$ is algebraically closed.
    We first show that $C$ is connected.
    The $\G_m$-action $\rho$ on $X_e$ restricts to a $\G_m$-action on $C$ which contracts all points to $e$.
    Therefore any closed nonempty $\G_m$-invariant subset of $C$ must contain $e$.
    So every connected component of $C$ contains $e$, so there is exactly one such component.
    This verifies the connectedness of $C$; the reducedness of $C$ follows from Corollary \ref{corollary: good e has central fiber reduced and unique singular point}. 

    We now propagate these properties to the other fibers of $\varphi$. 
    By \cite[Corollaire 12.1.7(vi)]{EGAIV-3}, the locus of $x\in X_e$  for which the fiber $C_b:= \varphi^{-1}(b)$ (where $b = \varphi(x)$) is geometrically reduced at $x$ is an open subset of $X_e$. Since it is also $\G_m$-equivariant and contains the central point $e$ it must equal the whole of $X_e$.

    To prove connectedness, let $b\in B(k)$ and consider the map $\iota\colon \A^1\rightarrow B$ sending $t$ to $t\cdot b$.
    Let $\chi\colon Y\rightarrow \A^1$ be the base change of $\varphi$ along $\iota$.
    Then the $\G_m$-action on $X_e$ and $B$ restricts to contracting $\G_m$-actions on $Y$ and $\A^1$ and $\chi$ is $\G_m$-equivariant.
    Since every closed nonempty $\G_m$-invariant subset of $Y$ contains $e$, $Y$ is connected.
    By \cite[Tag \href{https://stacks.math.columbia.edu/tag/055J}{055J}]{stacksproject}, the generic fiber of $\chi$ is connected.
    This implies $\chi^{-1}(1) = C_b$ is connected, as desired.
\end{proof}

Recall from \S\ref{subsec: stable gradings over algebraically closed field} that the discriminant polynomial $\Delta\in k[V]^G = k[B]$ is a polynomial with the property that $v\in V$ is stable if and only if $\Delta(v)\neq 0$, if and only if $v$ is regular semisimple.

\begin{lemma}\label{lemma: nonzero discriminant implies smooth curve}
    If $b\in B(k)$ has nonzero discriminant, then $C_b = \varphi^{-1}(b)$ is smooth.
\end{lemma}
\begin{proof}
    We may assume $k$ is algebraically closed.
    Choose a regular nilpotent $E\in V(k)$, complete it to an $\fraksl_2$-triple, and let $\kappa = X_E$ be the associated transverse slice.
    The map $\kappa \to B$ is an isomorphism by Proposition \ref{proposition: transverse slice regular nilpotent is isomorphism}.
    The morphism $G\times \kappa \rightarrow V$ is smooth, hence so is the restriction $G\simeq G\times \kappa_b\rightarrow V_b$ to the fiber at $b$.
    This restricted map is surjective, since all elements in $V_b$ are $G$-conjugate by Proposition \ref{proposition: basic invariant theory of vinberg representation full generality}.
    Therefore $V_b$ is smoothly covered by the smooth $G$, hence $V_b$ is smooth itself.
    Since the action map $G\times C_b \rightarrow V_b$ is smooth, the smoothness of $C_b$ follows from the smoothness of $V_b$.
\end{proof}

\begin{remark}
    It seems likely that, similar to \cite[Proposition 4.5]{Thorne-thesis}, the singularities of the fibers of $X_e\rightarrow B$ can be determined in terms of singularities of transverse slices for Levi subalgebras of $\frakh$.
\end{remark}

\subsection{Coxeter gradings}\label{subsection: Coxeter gradings}

As a first example, we now describe nilpotent orbits and graded Slodowy slices for Coxeter gradings of simple Lie algebras, as described in Example \ref{ex-cox}. Let $H$ be simple and split over $k$.

Let $\al_1, \dots, \al_\ell$ be simple roots for $\frakh$, and let $\al_0$ be the lowest root. Then $\{e_{\al_i} \mid i \in \{0, \dots, \ell\}\}$ is a basis for $\frakh_1$, and $G = T$ is a torus. The ring of invariants for the action of $G$ on $V$ is generated by the polynomial
\begin{equation*}
\sum_{i = 0}^\ell a_ie_{\al_i} \mapsto \prod_{i = 0}^\ell a_i^{c_i},
\end{equation*}
where the exponents $c_i$ are defined by $-\al_0 = \sum_{i = 1}^\ell c_i\al_i$ and $c_0 = 1$. 

A vector $v = \sum_{i = 0}^\ell a_ie_{\al_i}$ is nilpotent if and only if $a_i = 0$ for some $i$. 
If $v$ is nilpotent, then the dimension of its centralizer $\frakz_{\frakg}(v)$ is $\# \{i \mid a_i = 0\} - 1$. 
Thus a nilpotent vector $v$ has trivial centralizer in $\frakg$ if and only if exactly one $a_i = 0$ . If $k$ is algebraically closed, this tells us that there are exactly $\ell + 1$ nilpotent $G$-orbits of maximal dimension and $\frac{1}{2}\ell(\ell + 1)$ nilpotent orbits with centralizer in $\frakg$ of dimension 1. 

Let $E = \sum_{i = 0}^\ell b_ie_{\al_i}$ be a nilpotent with trivial centralizer in $\frakg$ and $b_j = 0$. We have $X_E = E + \spn\{e_{\al_j}\}$, so the nilpotent $E$ is reduced in the sense of Definition \ref{definition: reduced orbit} if and only if $c_j = 1$, which occurs if and only if $E$ is regular. (Here we're using the fact that if $c_j = 1$, then $\{\al_0, \al_1, \dots, \al_\ell\}\setminus \{\al_j\}$ forms a set of simple roots, which one can see using the affine Dynkin diagram for $H$ cf. \cite[Section 1.8]{Iwahori-Matsumoto}). 
A nilpotent $e = \sum a_ie_{\al_i}$ is in the closure of $G\cdot E$ if and only if $a_j = 0$, and $e$ is good as defined above if and only if $a_i = 0$ implies $c_i = 1$.

Fix $j, k \in \{0, \dots, \ell\}$ with $j \neq k$, and let $e = \sum_{i \neq j, k} e_{\al_i}$. 
We can complete $e$ to a normal $\fraksl_2$-triple with $f = \sum_{i \neq j, k} a_ie_{-\al_i}$ and all $a_i \neq 0$. Then $\frakz_{\frakh_1}(f) = \spn\{e_{\al_j}, e_{\al_k}\}$. If we consider the restriction of the map $V \to V\GIT G$ to $X_e$, the fiber over $c$ is given explicitly by $\{e + be_{\al_j} + de_{\al_k} \mid b^{c_j}d^{c_k} = c\}$, i.e. it is a curve of the form $x^{c_j}y^{c_k} = c$.

\begin{remark}
We may also consider \textit{twisted Coxeter gradings}, which generalize the Coxeter gradings considered here. Given a pinned automorphsim $\vartheta$, one may define the twisted Coxeter number $h_{\vartheta}$ (cf. \cite[Section 5]{Reeder-torsion}), and then the twisted Coxeter gradings are principal $h_{\vartheta}$-gradings of the form $\check\rho\vartheta$. A grading of this form is rank one, the subgroup $G$ is a subtorus of $T$, and the ring of invariants is described explicitly in \cite[Section 4.2]{VilonenXue}. 
Both nilpotent orbits and corresponding graded Slodowy slices can be described in a completely analagous way to the (non-twisted) Coxeter grading.
\end{remark}

\section{Subregular curves}\label{section:subregular curves}

Let $H$ be a reductive group over $k$ and $\theta$ a stable $\mu_m$-action on $H$.
In this section we completely describe the construction of Section \ref{section-families-of-curves} in the case when $e \in \frakh_1$ is a subregular nilpotent such that $\dim \frakz_{\frakg}(e) = 1$. To do so, we classify all stable gradings in which such a nilpotent appears when $H$ is simple. 
Given such an $e$, we explicitly describe the families of curves that arise as fibers of the map $\varphi: X_e \to B$, proving Theorem \ref{thm-intro}.
Specifically, Theorem \ref{thm-intro} follows from Theorems \ref{theorem: subregular classification: exclusion} and \ref{theorem:exhibitionsubregularcurves-algclosed}.
Our approach builds upon and resembles Slodowy's method to explicitly determine the family of subregular surfaces \cite{Slodowy-simplesingularitiesalggroups}; here we incorporate the $\mu_m$-action into the picture.
This is related to but slightly different from the approach of \cite[Section 3]{Thorne-thesis} in the case $m=2$.

\subsection{Subregular-adapted gradings}\label{subsection: subregular adapted gradings}

Recall that a nilpotent element $e\in \frakh(k)$ is called \textit{subregular} if $\dim \frakz_{\frakh}(e) = \mathrm{rk}(H) +2$, where $\mathrm{rk}(H)$ is the reductive rank of $H$ (cf. \cite[Section 5.4]{Slodowy-simplesingularitiesalggroups} or \cite[Section 4.2]{CollingwoodMcGovern-nilpotentorbits}).
The subregular orbits are the maximal elements of the set of non-regular orbits in the nilpotent cone of $\frakh$ under the closure partial ordering.
If $H$ is simple and $k$ algebraically closed, there exists a unique $H(k)$-orbit of subregular elements.

\begin{definition}
   A nonzero nilpotent element $e\in V(k)$ is called \textit{$\theta$-subregular} if $e$ is subregular and $\dim \frakz_{\frakg}(e)=1$.
   A stable $\mu_m$-action $\theta$ is called \emph{subregular adapted} if there exists a $\theta$-subregular nilpotent element in $V(k)$.
\end{definition}

(We add the condition that $e$ is nonzero to exclude the degenerate case that $H$ is of type $A_1$.)
In general, not every subregular element $e\in V(k)$ satisfies $\dim \frakz_{\frakg}(e)=1$, and not every nilpotent element $e$ satisfying $\dim \frakz_{\frakg}(e)=1$ is subregular (see the appendix for examples of this phenomena in the stable $8$-grading of $F_4$).
We single out the $\theta$-subregular nilpotents as an interesting class of nilpotents whose transverse slices $X_e\rightarrow B$ have good properties.

\begin{lemma}
    Suppose $e\in V(k)$ is $\theta$-subregular and that $E\in V$ is a nilpotent element with the property that $e\not\in G\cdot E$ but $e\in \overbar{G\cdot E}$.
    Then $E$ is regular.
    In particular, $e$ is good (in the sense of Section \ref{section-families-of-curves}).
\end{lemma}
\begin{proof}
    The hypotheses imply that $e\in \overbar{H\cdot E}\subset \frakh$.
    Since $e$ is subregular, this implies $E$ is regular or subregular. 
    For the sake of contradiction, assume we are in the latter case and write $\calO = H\cdot E = H\cdot e$ for the subregular orbit in $\frakh$.
    A result of Richardson \cite[Lemma 3.1]{Richardson-orbitsofalgebraicgroups} (see also \cite[Section 2.3]{Vinberg-theweylgroupofgraded}) shows that there is a decomposition $\calO \cap \mathcal{N} = \calO_1\sqcup \cdots \sqcup \calO_r$ into connected components such that each $\calO_i$ is a $G$-orbit.
    By assumption, we may order these orbits so that $e\in \calO_1$ and $E\in \calO_2$.
    Since $\calO_2$ is closed in $\calO\cap V$, $\overbar{\calO_2}\cap \calO = \calO_2$. 
    By assumption, $e\in \calO_1$ lies in $\overbar{\O_2} \cap \O = \O_2$.
    So $e\in \calO_1\cap \calO_2 = \varnothing$.
    This is a contradiction, so $E$ is indeed regular.
    Since every regular nilpotent is reduced, $e$ is good.
\end{proof}

Now we consider the GIT quotient map $\varphi\colon X_e\rightarrow B$ associated to a $\theta$-subregular nilpotent element $e$.
If $\alpha \colon V_1\rightarrow V_2$ is a linear map between vector spaces, write $\mathrm{corank}(\alpha)$ for the dimension of the cokernel of $\alpha$.
The following is the analogue of \cite[Proposition 3.10]{Thorne-thesis}.
\begin{proposition}\label{proposition: corank theta-subregular slice is 1}
If $e$ is $\theta$-subregular, then the corank of $(d\varphi)_e\colon T_e(X_e)\rightarrow T_0B$ is $1$.
\end{proposition}
\begin{proof}
    Recall from \S\ref{subsec: graded Slodowy slices} the GIT quotient maps $p\colon \frakh \rightarrow \frakh \GIT H$ and $\pi\colon V\rightarrow V\GIT G$ and the $\mu_m$-actions $\sigma$ on $\frakh$ and $\frakh\GIT H$.
    Then $e\in \frakh_1$ is fixed by $\sigma$, and the map $(dp)_e\colon T_e\frakh \rightarrow T_0(\frakh\GIT H)$, with the induced $\mu_m$-actions still denoted by $\sigma$, is equivariant.
    Denoting by $(\cdots)^{[i]}$ the $\sigma = \zeta^i$ eigenspace of this action, $(dp)_e$ breaks up into a direct sum of $(dp)_e^{[i]}\colon \frakh^{[i]} \rightarrow (T_0(\frakh\GIT H))^{[i]}$ for $i=0,1, \dots, m-1$.
    Unraveling the definition of $\sigma$, we find that $\frakh^{[i]}= \frakh_{1-i}$, $(T_0(\frakh\GIT H))^{[0]} = T_0B$ (by Proposition \ref{proposition: invariant polys panyushev}), and $(dp)_e^{[0]} = (d\pi)_e$.
    Since $X_e$ is a transverse slice at $e$, the image of $(d\pi)_e$ equals the image of $(d\varphi)_e$.
    We deduce that
    \[
    \mathrm{corank}((dp)_e)=
    \mathrm{corank}((d\varphi)_e)+
    \sum_{i=1}^{m-1} \mathrm{corank}((dp)_e^{[i]}).
    \]
    Since $e$ is subregular, the left-hand side equals $1$, by \cite[Section 8.3, Proposition 1]{Slodowy-simplesingularitiesalggroups}.
    Therefore exactly one of the terms on the right hand side is $1$ and all the others are zero.
    Since $e$ is not a smooth point of $\varphi$ by Corollary \ref{corollary: good e has central fiber reduced and unique singular point}, $\mathrm{corank}((d\varphi)_e)\geq 1$. 
    So $\mathrm{corank}((d\varphi)_e)=1$, as required.
\end{proof}

For the remainder of this subsection, suppose $e\in V(k)$ is a $\theta$-subregular nilpotent element, that $k$ is algebraically closed and that $H$ is adjoint and simple.
Complete $e$ to a normal $\fraksl_2$-triple $(e,h,f)$.
Consider the associated transverse slices $\psi\colon S_e \rightarrow \frakh\GIT H$ and $\varphi\colon X_e\rightarrow B$.
Recall the $\G_m \times \mu_m$-action $\rho\times \sigma$ on $S_e$ and $\frakh\GIT H$ (cf. Section \ref{subsec: graded Slodowy slices}).
The morphism $\psi\colon S_e\rightarrow \frakh\GIT H$ is a family of affine surfaces that has been explicitly determined by Slodowy \cite{Slodowy-simplesingularitiesalggroups}, building on work of Grothendieck and Brieskorn.
We will recall these results and incorporate $\mu_m$-action into the picture.
To this end, we first need to set up some notation and and preliminary results.

To deal with the non-simply laced cases, we need an additional symmetry group.
If $H$ is of type $A$, $D$ or $E$, define $\Gamma' = \{1\}$.
Otherwise, define $\Gamma'$ as the subgroup $Z_H(e)\cap Z_H(f)$ of $H$ of elements that preserve $e$ and $f$.
Then $\Gamma'$ is determined in \cite[Section 7.4, Lemma 4]{Slodowy-simplesingularitiesalggroups}; in particular, $\Gamma'$ always has reductive identity component.
If $H$ is not of type $B_{\rl}$ for some $\rl\geq 2$, let $\Gamma = \Gamma'$.
If $H$ is of type $B_{\rl}$, let $\Gamma$ be a subgroup of $\Gamma' \simeq \G_m \rtimes \mu_2$ of order two not contained in $\G_m$.
The adjoint $\Gamma'$-action on $\frakh$ restricts to a $\Gamma'$-action on $S_e$, and the morphism $\psi\colon S_e\rightarrow \frakh \GIT H$ is $\Gamma'$-invariant.
Inside $\Aut_H$, the subgroup $\theta(\mu_m)$ normalizes $\Gamma'$.
So the $\mu_m$-action $\sigma$ and $\Gamma'$-action on $S_e$ combine to a $\mu_m \ltimes \Gamma'$-action on $S_e$, where $\mu_m$ acts on $\Gamma'$ via $\zeta\cdot \gamma = \theta(\zeta) \gamma\theta(\zeta)^{-1}$.
Since these actions commute with the $\G_m$-action $\rho$, they combine to a $\G_m \times (\mu_m \ltimes \Gamma')$-action on $S_e$.

Let $r = \dim(\frakh \GIT H) = \rank(H)$. 
View $S = S_e = e + \frakz_{\frakh}(f)$ and $\frakh\GIT H$ as affine spaces with origins $e$ and $0$ respectively.
By \cite[Section 8.3, Proposition 1]{Slodowy-simplesingularitiesalggroups}, $(d\psi)_e$ has corank $1$ at $e$.
Therefore by \cite[Section 8.1, Lemma 2]{Slodowy-simplesingularitiesalggroups}, there exist $\G_m \times (\mu_m \ltimes \Gamma')$-stable decompositions $S_e = U_1 \oplus W$, $\frakh \GIT H = U_2 \oplus W$ with $\dim U_1 = 3$, $\dim U_2 = 1$, $\dim W =  r-1$, an equivariant morphism $F\colon  U_1 \oplus W \rightarrow U_2$ and an equivariant automorphism $\alpha$ of $S_e$ such that $(\psi\circ \alpha)(u,w)=(F(u,w),w)$.
The restriction of $F$ to $U_1$ is a $\G_m \times(\mu_m\ltimes \Gamma')$-equivariant map $f\colon U_1\rightarrow U_2$.
After choosing coordinates of $U_1$ and $U_2$ (in other words, isomorphisms $U_1\simeq \A^3$ and $U_2\simeq \A^1$), $f$ can be viewed as a morphism $\mathbb{A}^3 \rightarrow \mathbb{A}^1$, in other words a three-variable polynomial $f(x,y,z)$, and the central fiber $\psi^{-1}(0)$ is isomorphic to the closed subscheme of $\A^3$ defined by the vanishing of this polynomial.

Let $d_1, \dots, d_r$ be the invariant degrees of $H$ (i.e. the degrees of the $F_i$ described in Section \ref{subsec: invariant polynomials}), ordered so that $d_r$ is maximal.

\begin{proposition}[Slodowy]\label{prop: normal form surface singularity}
    There exist isomorphisms $U_1\simeq \A^3$ and $U_2\simeq \A^1$ equivariant with respect to the $\G_m \times \Gamma$-action such that $f\colon U_1\rightarrow U_2$ is identified with $\A^3\rightarrow \A^1, (x,y,z)\mapsto f_{\mathrm{std}}(x,y,z)$, where $f_{\mathrm{std}}$, the weights of $x,y,z$ and the $\Gamma$-action on $\A^3$ are given in Table \ref{table: surface singularities}.
    Moreover the $\rho$-weight of $U_2$ is $2d_r$.
\end{proposition}

In the final row, $\zeta$ denotes a generator of $\mu_3$ and $\tau$ a generator of $\Z/2\Z$.
\begin{proof}
    Slodowy has determined the $\G_m$-weights of $U_1,U_2$, see \cite[Section 7.4, Proposition 2 and Section 8.2]{Slodowy-simplesingularitiesalggroups}, which in particular shows $U_2$ has $\rho$-weight $2d_r$.
    Moreover he shows \cite[Section 8.3, Proposition 2]{Slodowy-simplesingularitiesalggroups} that there exists a $\G_m$-equivariant isomorphisms $U_1\simeq \A^3$ and $U_2\simeq \A^1$ so that $f$ has the normal form of a simple singularity displayed at the end of \cite[Section 8.2]{Slodowy-simplesingularitiesalggroups}, depending on the Dynkin diagram of $H$.
    If $H$ is not simply laced, he also gives a normal form for the $\Gamma$-action in \cite[Section 8.4]{Slodowy-simplesingularitiesalggroups}.
    This leads to the entries of Table \ref{table: surface singularities}.
    (In case $D_4$, we have transformed the cubic form $-xy^2 + x^3$ to $x^3 +y^3$ using a linear change of variables.)
\end{proof}

\begingroup

\renewcommand{\arraystretch}{1.5} 
\begin{table}
\centering
\begin{tabular}{|l | l | l | l| l| }
	\hline
    Type & $f_{\mathrm{std}}(x,y,z)$ & Weights of $x,y,z$ &$\Gamma$ & $\Gamma\curvearrowright \A^3$  \\
	\hline       
	$A_{\rl}\, (\rl\geq 2)$ & $z^2-y^2+ x^{\rl+1}$ & $2,\rl+1,\rl+1$ & $\{1\}$ &   \\
    \hline 
    $B_{\rl}\, (\rl\geq 2)$ & $z^2- y^2 + x^{2\rl}$ & $2,2\rl,2\rl$ &$\{1,\gamma\}$ & $\gamma\cdot (x,y,z) = (-x,-y,z)$ \\
    \hline
    $C_{\rl}\, (\rl\geq 3)$ & $z^2 - xy^2 + x^{\rl}$ & $4,2\rl-2,2\rl$ &$\{1,\gamma\}$ & $\gamma\cdot (x,y,z) = (x,-y,-z)$  \\
    \hline
	$D_{\rl}\, (\rl\geq 5)$ & $z^2  -xy^2 + x^{\rl-1}$  & $4,2\rl-4,2\rl-2$& $\{1\}$ &    \\
    \hline
    $D_{4}$ & $z^2  +y^3 + x^{3}$  & $4,4,6$& $\{1\}$ &    \\
    \hline
	$E_6$ & $z^2 - y^3 + x^4$ & $6,8,12$ & $\{1\}$ &   \\
    \hline
	$E_7$ & $z^2 - y^3 + x^3y $ & $8,12,18$ & $\{1\}$ &  \\
    \hline
	$E_8$ & $z^2 - y^3 + x^5 $ & $12,20,30$& $\{1\}$ &  \\
    \hline
    $F_4$ & $z^2 - y^3 + x^4$ & $6,8,12$ &$\{1,\gamma\}$ & $\gamma\cdot (x,y,z) = (-x,y,-z)$\\
    \hline
    $G_2$ & $z^2 + y^3 + x^3$& $4,4,6$ &$\mu_3\rtimes (\Z/2\Z)$ & $\zeta\cdot (x,y,z) = (\zeta x,\zeta^2 y,z)$\\
     & &  && $\tau\cdot (x,y,z) = (y,x,-z)$\\
	\hline
\end{tabular}
\caption{Subregular surface singularities and their weights}
\label{table: surface singularities}
\end{table}
\endgroup

\begin{lemma}\label{lemma: sigma action on U1 and U2}
    The $\mu_m$-action $\sigma$ on $U_2$ is trivial, and the $\mu_m$-action $\sigma$ on $U_1$ has weights $0,0,1$.
\end{lemma}
\begin{proof}
    The $\sigma$-fixed-point locus of $S_e\rightarrow \frakh \GIT H$ is exactly $\varphi\colon X_e \rightarrow B$. 
    Assume for the sake of contradiction that $U_2^{\sigma} \neq U_2$.
    Then since $\dim U_2 = 1$, we have $U_2^{\sigma} = 0$. 
    So $\psi^{\sigma} = \varphi$ equals the map $U_1^{\sigma} \oplus W^{\sigma} \rightarrow W^{\sigma}$, $(u,w)\mapsto w$, which is surjective. 
    This contradicts Proposition \ref{proposition: corank theta-subregular slice is 1}, hence proves that $U_2^{\sigma} = U_2$.
    
    By Proposition \ref{proposition: corank theta-subregular slice is 1}, since $\dim U_2^{\sigma} = 1$, we have $\dim U_1^{\sigma} = 2$.
    Since $e$ is $\theta$-subregular, $\dim \frakz_{\frakg}(e) =1$, hence $\dim \frakz_{\frakg}(f) = 1$ by Lemma \ref{lem-ef-centralizer}.
    By definition of the $\sigma$-action on $S_e$, it follows that $1\in \Z/m\Z$ is a weight for the $\sigma$-action on $S_e = U_1 \oplus W$.
    Let $E\in \frakh_1$ be a regular nilpotent element, let $(E,Y,F)$ be a normal $\fraksl_2$-triple containing it and let $\kappa = E + \frakz_{\frakh}(F) \xrightarrow{\sim} \frakh \GIT H$ be the associated transverse slice, which is an isomorphism by Proposition \ref{proposition: transverse slice regular nilpotent is isomorphism}.
    This isomorphism is $\sigma$-equivariant and $\dim \frakz_{\frakg}(F) = 0$, so $1$ is not a weight for the $\sigma$-action on $\frakh \GIT H = U_2 \oplus W$.
    Let $(u, w) \in U_1 \oplus W$ be a nonzero $\sigma$-eigenvector with weight 1. If $w \neq 0$, then since $\psi$ is $\sigma$-equivariant, we have that $w$ is a $\sigma$-eigenvector for the $\sigma$-action on $\frakh\GIT H$, a contradiction. So $u \in U_1$ is a $\sigma$-eigenvector of weight $1$.
    It follows that $1 \in \bZ/m\bZ$ is a weight for $U_1$.
\end{proof}

In the next proposition, fix a primitive $m$th root of unity $\zeta_m\in k$ so that we may identify the $\mu_m$-action $\sigma$ with the order-$m$ automorphism $\sigma(\zeta_m)$.

\begin{proposition}\label{proposition: normal form Gamma and sigma}
    There exist coordinates $x,y,z$ on $U_1$ and $p_{d_r}$ on $U_2$ (and in the case of $B_{\rl}$, a choice of subgroup $\Gamma\leq \Gamma'$) that satisfy the conclusions of Proposition \ref{prop: normal form surface singularity} and such that one of the following hold for the integer $m$ and the automorphism $\sigma$:
    \begin{itemize}
        \item $m=2$ and $\sigma(x,y,z) = (x,y,-z)$;
        \item Type $A_{\rl}$: $m=\rl+1$ and $\sigma(x,y,z) = (\zeta_m x, y,z)$;
        \item Type $B_{\rl}$: $m=2$ and $\sigma(x,y,z)= (x,-y,z)$, or $m$ divides $2r$ and $\sigma(x,y,z) = (\zeta_mx,y,z)$;
        \item Type $C_{\rl}$: $m=2$ and $\sigma(x,y,z) = (x,-y,z)$;
        \item Type $D_{4}$: $m=3$ and $\sigma(x,y,z) = (\zeta_3x ,y,z)$ ;
        \item Type $E_6$, $m=3$ and $\sigma(x,y,z) = (x,\zeta_3 y,z)$, or $m=4$ and $\sigma(x,y,z) = (\zeta_4 x,y,z)$;
        \item Type $E_8$: $m=3$ and $\sigma(x,y,z) = (x, \zeta_3 y,z)$ or $m=5$ and $\sigma(x,y,z) = (x,y,\zeta_5 z)$;
        \item Type $F_4$, $m=3$ and $\sigma(x,y,z) = (x,\zeta_3 y,z)$, or $m=2,4$ and $\sigma(x,y,z) = (\zeta_mx ,y,z)$;
        \item Type $G_2$: $m=2$ and $\sigma(x,y,z) = (\omega y,\omega^2 x,z)$ for some $\omega\in k$ with $\omega^3=1$, or $m=3$ and $\sigma(x,y,z) = (\zeta_3x,y,z)$ or $\sigma(x,y,z) = (x,\zeta_3y,z)$.
    \end{itemize}
\end{proposition}
\begin{proof}
    Choose coordinates $x,y,z$ on $U_1$ and $p_{d_r}$ on $U_2$ satisfying the conclusions of Proposition \ref{prop: normal form surface singularity}, so we identify the morphism $U_1\rightarrow U_2$ with the morphism $\A^3\rightarrow \A^1, (x,y,z)\mapsto f_{\mathrm{std}}(x,y,z)$. 
    The automorphism $\sigma$ on $U_1$ induces an automorphism of $\A^3$, which we denote by $s$ in this proof.
    By Lemma \ref{lemma: sigma action on U1 and U2}, $s \colon \A^3\rightarrow \A^3$ is a $\rho$-equivariant isomorphism whose derivative at the origin has eigenvalues $1,1,\zeta_m$ and such that $f_{\mathrm{std}}$ is $s$-invariant.
    To prove the proposition, it suffices to classify all possible such actions.
    To this end, we argue on a case-by-case basis.

     Type $A_\rl$:  In this case $x, y, z$ have $\rho$-weights $2, \rl + 1, \rl + 1$ respectively. 
     We claim that $s$ is linear in the variables $x,y,z$. Indeed, if $r$ is even, then this follows from $\rho$-equivariance. If $r$ is odd, then $s$ has the form 
     \begin{equation*}
     (x,y,z)\mapsto (\varepsilon x, g_0(y,z) + \al x^{\frac{r + 1}{2}}, g_1(y,z) + \be x^{\frac{r + 1}{2}})
    \end{equation*}
    for some scalars $\varepsilon, \al, \be$ and some linear combinations $g_0(y,z),g_1(y,z)$ of $y$ and $z$.
    Since the derivative of $s$ at $0$ is an isomorphism $g_0(y,z), g_1(y,z)\in \spn\{y,z\}$ are linearly independent.
    The fact that $s$ preserves $f_{\mathrm{std}} = y^2-z^2 + x^{\rl+1}$ shows that $\al = \be=0$ and so $s$ is linear in $x,y,z$.
    The linear map $s$ has eigenvalues $1,1,\zeta_m$ and preserves the form $y^2-z^2 + x^{\rl+1}$.
    Therefore $s$ preserves the subspaces $\spn\{x\}$ and $\spn\{y,z\}$ and the form $y^2 - z^2$.
    An elementary calculation shows that up to linear substitution the only such linear maps are $s_1(x,y,z)=(x,y,-z)$ and $s_2(x,y,z)=(\zeta_mx,y,z)$, where last case only occurs when $m$ divides $\rl+1$.
    Suppose for the sake of contradiction that $m\neq \rl+1$.
    Then the classification of stable gradings \cite[Section 7.2]{GrossLevyReederYu-GradingsPosRank} implies that $m = 2d$ for some odd $d$ and $\theta^d$ is the stable $\Z/2\Z$-grading $\theta_0$.
    Therefore $e$ is a subregular nilpotent element of $\frakh^{\theta_0=-1}$ such that the corresponding $\mu_2$-action $\sigma_0=\sigma^{d}$ on $S_e$ acts via $(x,y,z)\mapsto (-x,y,z)$, and the central fiber of $\psi^{\sigma} = \varphi\colon X_e\rightarrow B$ is isomorphic to $y^2 - z^2 = 0$.
    This contradicts \cite[Proposition 3.11]{Thorne-thesis}, so $m=\rl+1$, as desired.

Type $B_\rl$ ($\rl \geq 2$):
In this case $x, y, z$ have $\rho$-weights $2, 2r, 2r$ respectively, and the restriction of $s$ to $U_1$ must preserve the form $z^2 - y^2 + x^{2r}$. 
A very similar proof to the $A_\rl$ case shows that after possibly a linear change of variables $s$ is a linear map that equals either $s_1(x,y,z) = (x,-y,z)$, $s_2(x,y,z) = (x,y,-z)$ or $s_3(x,y,z) = (\zeta_mx,y,z)$ (where in the last case $m$ must divide $2\rl$).
This linear change of variables might change the action of $\Gamma$, but in each case we claim that we can replace $\Gamma$ by a conjugate subgroup in $\Gamma'$ for which the action is as given in Table \ref{table: surface singularities}.
Indeed, there exists an isomorphism $\Gamma' \simeq\G_m \rtimes \mu_2$, where the semidirect product structure is determined by $(\lambda, 1) \cdot (1,t) = (1,t) (\lambda^{-1},1)$ if $t\in \mu_2$ is the nontrivial element.
Let $\varepsilon\colon \Gamma'\rightarrow \mu_2$ the nontrivial character sending $(\lambda, t^i)$ to $t^i$.
By the discussion in \cite[Section 8.4]{Slodowy-simplesingularitiesalggroups}, the $\Gamma'$-action on $\A^3$ is linear, preserves the subspaces $V_1=\spn\{x\}$, $V_2=\spn \{y,z\}$, acts on $V_1$ via $\sigma$ and the action on $V_2$ is via an isomorphism $\Gamma'\rightarrow \mathrm{O}(V_2)$, the orthogonal group of the form $y^2-z^2$. A calculation shows that there exists an order two element of $\Gamma'$ not lying in $\G_m$ that acts via $(x,y,z)\mapsto (-x,-y,z)$.
Replacing $\Gamma$ by the subgroup generated by this element gives the desired normal form. 

Type $C_\rl$ $(\rl\geq 3)$:
In this case $x, y, z$ have $\rho$-weights $4, 2r - 2, 2r$ respectively, and $s$ must preserve $z^2 - xy^2 + x^r$. First assume $r \geq 4$. If $r$ is even, then $s$ acts by $(x, y,z)\mapsto (\al_1 x,\al_2 y,\al_3 z + \al_4 x^{\frac{r}{2}})$ for some scalars $\alpha_i$.
Since $s$ preserves the form $z^2 - xy^2 + x^r$, we must have $\al_4 = 0, \al_1 = 1$, and $\al_2, \al_3 \in \{\pm 1\}$. 
By a similar argument in the case that $\rl$ is odd, we conclude that $s$ must be of the form 
$s_1(x,y,z) = (x,-y,z)$ or $s_2(x,y,z) = (x,y,-z)$ and $m=2$.
If $\rl=3$, then $x,y,z$ have weights $4,4,6$ respectively. 
Thus $s$ is a linear map with eigenvalues $1, 1, \zeta_m$ preserving $\spn\{x, y\}$, $\spn\{z\}$, and the form $z^2 - xy^2+ x^3$. 
By the classification of stable gradings \cite[\S7.2, Table 13]{GrossLevyReederYu-GradingsPosRank}, either $m=2$ or $m=6$.
    An explicit computation shows that no such linear map exists when $m=6$, and when $m=2$, $s$ must be of the form $(x, y, z) \mapsto (x, y, -z)$, $(x,y, z)\mapsto (x,-y, z)$, $(x,y, z)\mapsto (y-x,y, z)$ or $(x,y, z)\mapsto (x+y,-y, z)$. 
    Since the $\sigma$-action must normalize the $\Gamma$-action, we must be in the first or second case.

Type $D_\rl (\rl \geq 4)$:  
If $m=2$, then by Proposition \cite[Proposition 3.11]{Thorne-thesis} we must have $s(x,y,z) = (x,y,-z)$.
If $\rl\geq 5$, then similar logic as in type $C_\rl$ shows that $m=2$.
So we may assume $\rl=4$ and $m\geq 3$. In this case $x, y, z$ have weights $4, 4, 6$ respectively, so $s$ is linear in the variables $x,y,z$. 
Moreover it has eigenvalues $1, 1, \zeta_m$ and preserves the form $z^2 + x^3 + y^3$.
Such a map must fix $z$, preserve $\spn\{x,y\}$ and $s\vert_{\spn\{x,y\}}$ must preserve $x^3+ y^3$.
A calculation shows that $m=3$ and $s$ equals either $s_1(x,y,z) = (\zeta_3x,y,z)$ or $s_2(x,y,z) = (x, \zeta_3y,z)$.
The latter case can be reduced to the first case after swapping $x$ and $y$.

Type $E_6$: the coordinates $x,y,z$ have weights $6,8,12$, hence $s(x,y,z)= (\lambda_1 x,\lambda_2y,\lambda_3 z + \mu x^2)$ for some $\lambda_i, \mu \in k$.
Since $\sigma$ preserves $x^4 + y^3 + z^2$, we must have $\mu=0$ and $s$ equals either $s_1(x,y,z) = (\zeta_mx,y,z)$ with $m=2$ or $4$, $s_2(x,y,z) = (x,\zeta_3y,z)$ with $m=3$, or $s_3(x,y,z) = (x,y,-z)$ with $m=2$.
The first case cannot occur with $m=2$ because it would contradict \cite[Proposition 3.11]{Thorne-thesis}.

Type $E_7$: A calculation shows $s$ equals either $s_1(x,y,z) = (\zeta_3x,y,z)$ or $s_2(x,y,z) = (x,y,-z)$.
Since there is no stable $\Z/3\Z$-grading on $E_7$ \cite[Section 7.1, Table 4]{GrossLevyReederYu-GradingsPosRank}, the first case does not occur.

Type $E_8$: A calculation shows $s$ is either $(x,y,z)\mapsto (x,y,-z)$, $(x,y,z)\mapsto (x,\zeta_3y,z)$, or $(x,y,z)\mapsto (\zeta_5x,y,z)$. 

Type $F_4$: Similarly to the $E_6$ case, we calculate that $s$ is either $s_1(x,y,z) = (\zeta_mx,y,z)$ with $m=2$ or $4$, $s_2(x,y,z) = (x,\zeta_3y,z)$ with $m=3$, or $s_3(x,y,z) = (x,y,-z)$ with $m=2$.

Type $G_2$: Similarly to the $D_4$ case, $s$ is linear in $x,y,z$ and preserves the form $z^2+ y^3 + x^3$ with eigenvalues $1,1,\zeta_m$.
A computation shows that $s$ equals either $s_1(x,y,z) = (x,y,-z)$, $s_2(x,y,z) = (\lambda y,\lambda^{-1} x,z)$ for some $\lambda\in k$ with $\lambda^3 = 1$, $s_3(x,y,z) = (\zeta_3x,y,z)$ or $s_4(x,y,z) = (x,\zeta_3y,z)$.
\end{proof}

Fix coordinates $x,y,z$ on $U_1$ and $p_{d_r}$ on $U_2$ satisfying the conclusions of Proposition \ref{proposition: normal form Gamma and sigma}.
Inspecting the $\Gamma$ and $\sigma$-action shows that they combine to a $\mu_m \ltimes \Gamma$-action on $S_e$. 
Let $J$ be the ideal of $k[x,y,z]$ generated by $\partial f_{\mathrm{std}}/\partial x, \partial f_{\mathrm{std}}/\partial y$ and $\partial f_{\mathrm{std}}/\partial z$.
Then the quotient $k[x,y,z]/J$ is finite-dimensional and receives an action of $\G_m\times (\mu_m \ltimes \Gamma)$ by combining the $\rho, \sigma$ and $\Gamma$-action on $U_1$.
Since the identity component of $\G_m\times (\mu_m \ltimes \Gamma)$ is reductive, there exists a $\G_m \times(\mu_m\ltimes \Gamma)$-equivariant section $\chi\colon k[x,y,z]/J \rightarrow k[x,y,z]$.
The group $\G_m \times \mu_m$ acts on the $\Gamma$-fixed points $(k[x,y,z]/J)^{\Gamma}$ and the restriction $\chi^{\Gamma}\colon (k[x,y,z]/J)^{\Gamma}\rightarrow k[x,y,z]$ is $\G_m\times \mu_m$-equivariant.
Choose a basis $g_1, \dots, g_{n}\in k[x,y,z]$ for the image of $\chi^\Gamma$ that are eigenvectors for the $\G_m\times \mu_m$-action.
Let $Z = \{f_{\mathrm{std}} + t_1 g_1 + \cdots + t_n g_n = 0\}\subset \A^3 \times \A^n$, where $x,y,z$ are the coordinates on $\A^3$ and $t_1, \dots, t_n$ are the coordinates on $\A^n$. 
Let $\psi_{\mathrm{std}}\colon Z\rightarrow D = \A^n$ denote the projection to the second factor.
If $f_{\mathrm{std}}$ has $\G_m\times \mu_m$-weight $(d,d')$ and $g_i$ has weight $(r_i,r_i')$, let $\G_m\times \mu_m$ act on $t_i$ with weight $(d-r_i, d'-r'_i)$. 
Let $\Gamma$ act on $Z$ by acting trivially on $\A^n$ and on $\A^3$ via its action on $U_1$.
Then $\psi_{\mathrm{std}} \colon Z\rightarrow \A^n$ is $\G_m\times \mu_m$-equivariant and $\Gamma$-invariant.

\begin{proposition}\label{proposition: isomorphism between semiversal deformations}
    There exists a pullback diagram whose horizontal arrows are $\G_m\times \mu_m$-equivariant isomorphisms:
\[\begin{tikzcd}
	{S_e} & Z \\
	{\frakh \GIT H} & D
	\arrow["\sim", from=1-1, to=1-2]
	\arrow["\psi"', from=1-1, to=2-1]
	\arrow["{\psi_{\mathrm{std}}}", from=1-2, to=2-2]
	\arrow["\sim", from=2-1, to=2-2]
\end{tikzcd}\]
\end{proposition}
\begin{proof}
Extend $g_1, \dots, g_m$ to a basis $g_1, \dots, g_n , \dots, g_N$ of the image of $\chi\colon k[x,y,z]/J\rightarrow  k[x,y,z]$, let $Z' = \{f_{\mathrm{std}} + t_1 g_1 + \cdots + t_N g_N = 0\}\subset \A^3 \times \A^N$ and let $\psi'_{\mathrm{std}}\colon Z'\rightarrow D' = \A^N$ denote the projection to the second factor.
Then by \cite[Section 2.5]{Slodowy-simplesingularitiesalggroups}, the formal completion of $\psi'_{\mathrm{std}}\colon Z'\rightarrow D' = \A^N$ is a $\G_m \times (\mu_m \ltimes \Gamma)$-semiuniversal deformation of the variety $f^{-1}(0)$ with its $\G_m \times (\mu_m \ltimes \Gamma)$-action.
Moreover the restriction of $\psi_{\mathrm{std}}$ to the $\Gamma$-fixed points $Z'^{\Gamma}$ and $D'^{\Gamma}$ is exactly $\psi_{\mathrm{std}}$.
Since $\psi\colon S_e\rightarrow \frakh \GIT H$ is $\G_m\times (\mu_m\ltimes \Gamma)$-equivariant and $\Gamma$ acts trivially on the base $\frakh \GIT H$, the semiuniversality of $Z'\rightarrow D'$ shows that there is a pullback diagram between the formal completions at the origin
\[\begin{tikzcd}
	{\widehat{S_e}} & {\widehat{Z}} \\
	{\widehat{\frakh \GIT H}} & {\widehat{D}}
	\arrow[from=1-1, to=1-2]
	\arrow["\psi"', from=1-1, to=2-1]
	\arrow["{\psi_{\mathrm{std}}}", from=1-2, to=2-2]
	\arrow[from=2-1, to=2-2]
\end{tikzcd}\]
where the top horizontal map is $\G_m\times (\mu_m\ltimes \Gamma)$-equivariant and the bottom one is $\G_m\times \mu_m$-equivariant.
Forgetting the $\mu_m$-action, we know by \cite[Section 2.5, Theorem]{Slodowy-simplesingularitiesalggroups} that $Z\rightarrow D$ is a $\G_m$-semiuniversal deformation of $(f_{\mathrm{std}}^{-1}(0), \Gamma)$. 
So the proof \cite[Section 8.7, Theorem]{Slodowy-simplesingularitiesalggroups} shows that these horizontal maps have canonical algebraizations which are isomorphisms, proving the proposition.
\end{proof}

\subsection{Classification of subregular curves: exclusion}

Using the $\mathbb{G}_m \times \mu_m$ action on $S_e$ and $\frakh\GIT H$ and the preliminary results of \S\ref{subsection: subregular adapted gradings}, we now begin to prove Theorem \ref{thm-intro}. 
We first prove that if $\theta$ is a subregular-adapted $\mu_m$-action and $e\in V(k)$ is $\theta$-subregular, then $(m, H)$ and $X_e\rightarrow B$ appear in Table \ref{table:examples intro}.

Let $\theta\colon \mu_m\rightarrow \Aut_H$ be a stable and subregular-adapted $\mu_m$-action and $e\in V(k)$ a $\theta$-subregular nilpotent element.
Complete $e$ to a normal $\frak{sl}_2$-triple $(e,h,f)$ and let $X_e\rightarrow B$ be the corresponding graded Slodowy slice.

\begin{theorem}\label{theorem: subregular classification: exclusion}
    Suppose $k$ is algebraically closed and let $\theta\colon \mu_m\rightarrow \Aut_H$ be subregular adapted.
    \begin{enumerate}
        \item The pair $(m, {}^nH)$ appears in the first two columns of Table \ref{table:examples intro}.
        (The left superscript $n$, when present in the Dynkin type, indicates that the $\mu_m$-action has nontrivial outer type of order $n$.)
        \item Let $e\in V(k)$ be a $\theta$-subregular nilpotent element, complete it to a normal $\liesl_2$-triple, and let $X_e\rightarrow B$ be the associated graded Slodowy slice.
        Then there exist coordinates $p_{n_1}, \dots, p_{n_d}$ on $B$ and coordinates $x,y,p_{n_1}, \dots, p_{n_{d-1}}$ on $X_e$ such that each $x,y,p_{n_i}$ is $\rho$-homogeneous, $p_{n_i}$ has weight $2n_i$ and the map $X_e \rightarrow B$ is identified with the map $\A^{d+1} \rightarrow \A^d$ given by  
        \[(x,y,p_{n_1}, \dots,p_{n_{d-1}})\mapsto (-F_0(x,y,p_{n_1}, \dots,p_{n_{d-1}}), p_{n_1}, \dots, p_{n_{d-1}}),\]
        where $F_0 + p_{n_d}=0$ is displayed in the fifth column of Table \ref{table:examples intro}, depending on $(H,m)$.
    \end{enumerate}
\end{theorem}

For example, when $m=2$ and $H$ is of type $^2A_\rl$, the curves $X_e\rightarrow B$ in Table \ref{table:examples intro} are described as $y^2 = x^{\rl+1} + p_2 x^{\rl-1} +\cdots + p_{\rl+1}$, which means that $p_{n_d} = p_{\rl+1}$ and $F_0 = -y^2 + x^{\rl+1} + p_2 x^{\rl-1} + \cdots p_{\rl} x$.

\begin{proof} 
Choose decompositions $S_e = U_1\oplus W$, $\frakh\GIT H = U_2\oplus W$ and coordinates $x,y,z$ on $U_1$ and $p_{d_r}$ on $U_2$ satisfying the conclusions of Proposition \ref{proposition: normal form Gamma and sigma}.
That proposition immediately shows that the pair $(m,H)$ appears in the first two columns of Table \ref{table:examples intro}.
The classification of stable gradings \cite[Section 7]{GrossLevyReederYu-GradingsPosRank} shows that in each of those cases there is a unique $n$ for which there is a stable $\mu_m$-action on $H$ with outer type of order $n$, so this automatically determines $n$ too, proving the first part of the proposition.

We now prove the second part.
By Proposition \ref{proposition: isomorphism between semiversal deformations}, $\psi$ is isomorphic to $\psi_{\mathrm{std}}$, where $\psi_{\mathrm{std}}$ is the map $Z\rightarrow D$ constructed above Proposition \ref{proposition: isomorphism between semiversal deformations}.
This construction depends on $f_{\mathrm{std}}(x,y,z)$ and the $\mu_m\ltimes \Gamma$-action on $\A^3$.
We can carry out this construction for each case of Proposition \ref{proposition: normal form Gamma and sigma}; in fact the elements $g_1, \dots, g_n$ can always be chosen to be monomials.
Noting that $\psi_{\mathrm{std}}^{\sigma}$ is isomorphic to $\psi^{\sigma} = \varphi\colon X_e\rightarrow B$, we obtain all the cases of Table \ref{table:examples intro}.

As an example, consider the case when $H$ is of type $B_{\rl}$, $m=2$, $\gamma\cdot (x,y,z) = (-x,-y,z)$ and $\sigma(x,y,z) = (x,y,-z)$.
Since $f_{\mathrm{std}} = z^2-y^2 + x^{2r}$, $J = (x^{2r-1},y,z)$ and $k[x,y,z]/J$ has $k$-basis the images of the elements $1,x,\dots, x^{2r-2}$.
The map $\chi\colon k[x,y,z]/J\rightarrow k[x,y,z]$ sending the coset of $x^i$ to $x^i$ for each $0\leq i \leq 2r-2$ is equivariant with respect to the $\G_m\times (\mu_2\ltimes \Gamma)$-action.
Taking $\Gamma$-fixed points, we are left with the elements $1,x^2, \dots, x^{2r-2}$.
So $Z = \{z^2-y^2+ x^{2r} + p_2 x^{2r-2} + \cdots + p_{2r} = 0\}\subset \A^{3+r}$, $D = \A^r$ and $\psi_{\mathrm{std}}\colon Z\rightarrow D$ sends a tuple $(x,y,z,p_2, \dots, p_{2r})$ to $(p_2, \dots, p_{2r})$.
The $\sigma$-fixed point locus $\psi_{\mathrm{std}}$ is isomorphic to $\varphi\colon X_e\rightarrow B$ and is given by $Z^{\sigma}\rightarrow D^{\sigma}\colon (x,y,p_2, \dots, p_{2r})\mapsto (p_2, \dots, p_{2r})$.
Projecting onto the first $r+1$ coordinates defines an isomorphism $Z^{\sigma} \rightarrow \A^{r+1}$, $(x,y,p_2, \dots, p_{2r})\mapsto (x,y,p_2, \dots, p_{2r-2})$.
Under this isomorphism $\psi^{\sigma}_{\mathrm{std}}\colon \A^{r+1} \rightarrow \A^r$ defines the family of curves displayed in the second row of Table \ref{table:examples intro}.

The other cases are similar. 
We give one more example of a slightly less straightforward calculation in the case when $H$ is type $G_2$ and $m = 2$. Suppose $H$ is type $G_2$ and $\sigma$ acts on $U_1$ by $(x, y, z) \mapsto (\omega y, \omega^2 x, z)$ for some primitive third root of unity $\omega$. 
Since $f_{\mathrm{std}} = z^2 + y^3 + x^2$, we have that $J = (z, y^2, x^2)$, $k[x, y, z]/J$ has $k$-basis given by the images of $1, x, y, xy$, and we may take the section $\chi$ to be the obvious map with respect to this basis. Looking at the $\Gamma$-action, we see that $(k[x, y, z]/J)^\Gamma$ is spanned by $1$ and $xy$, and $Z = \{z^2 + y^3 + x^3 + p_2xy + p_6 = 0\} \subset \A^5$. We have that $u := x + \omega y$ and $v := x - \omega y$ are eigenvectors for the $\sigma$-action. Writing $Z$ in terms of these variables and setting $v = 0$, we have that $Z^\sigma$ is given by $z^2 + \frac{1}{4}u^3 + \frac{\omega^2 p_2}{4}u^2 + p_6 = 0$, which after a linear change of variables yields the family of the form $y^2 = x^3 + p_2x^2 + p_6$ listed in Table \ref{table:examples intro}.
\end{proof}

\subsection{Classification of subregular curves: exhibition}\label{subreg-exhibition}

In the previous subsection, we have constrained the subregular-adapted gradings on simple groups and the curves they give rise to. 
We now show that all the possibilities listed in Table \ref{table:examples intro} do indeed arise.
Specifically, we prove the following theorem.

\begin{theorem}\label{theorem:exhibitionsubregularcurves-algclosed}
    Suppose $k$ is algebraically closed.
    For every row in Table \ref{table:examples intro}, there exists a stable $\mu_m$-action on a simple $H$ of the corresponding type and a $\theta$-subregular nilpotent $e\in V(k)$ such that $X_e\rightarrow B$ is of the form corresponding to that entry.
\end{theorem}

We prove this theorem in a few steps. 
\begin{proposition}\label{proposition: tabulated gradings are subregular adapted}
    Every pair $({}^nH, m)$ (where $n$ denotes the order of the outer type) appearing in Table \ref{table:examples intro} is subregular adapted.
\end{proposition}
\begin{proof}
    If $m=2$, this was established by Thorne \cite[Lemma 2.21 and Theorem 2.22]{Thorne-thesis}, using the connection between involutions on $H$ and real forms of $H$.
    For the remaining entries, we may assume $m\geq 3$ and argue on a case-by-case basis.
    
    When $H$ is of type $A_\rl$, then $\theta$ is conjugate to the Coxeter grading of Example \ref{ex-cox}. 
    In the notation of that example, we have $\frakg = \frakt$ and $V = \sum_{i=0}^{\rl} \frakh_{\alpha_i}$.
    Choose a nonzero element $X_{\alpha_i} \in \frakh_{\alpha_i}$ for each simple root $\alpha_i$. Then $e = X_{\alpha_1} + \cdots + X_{\alpha_{\rl-1}}$ is a subregular nilpotent of $\frakh \simeq \fraksl_{\rl+1}$, by considering Jordan normal forms. It is easily checked to satisfy $\dim \frakz_{\frakt}(e) = 1$, so it is $\theta$-subregular.

    Suppose $H$ is of type $B_\rl$. 
    By the classification of stable gradings \cite[Section 7.2, Table 12]{GrossLevyReederYu-GradingsPosRank}, $m$ is an even divisor of $2\rl$ and we may assume $\theta$ equals $\theta_{c}^n$, where $\theta_{c}$ is the Coxeter grading associated with a pinning of $H$ and $n = 2\rl/m$.
    There exists a subregular nilpotent element $e$ which is supported on all but one of the simple root vectors; combine \cite[Propositions 5.2.5 and 5.4.1]{CollingwoodMcGovern-nilpotentorbits}.
    Such an element lies in $\frakh^{\theta_{c} = \zeta_{2\rl}}$ and is $\theta_{c}$-subregular by direct computation.
    We have inclusions $\frakz_{\frakh}(e)^{\theta_{c}} \subset \frakz_{\frakh}(e)^{\theta} \subset \frakz_{\frakh}(e)^{\theta_{c}^\rl}$.
    We have already shown $\dim \frakz_{\frakh}(e)^{\theta_{c}} = 1$.
    Since $e$ is subregular and $\theta_{c}^\rl$ is a stable $2$-grading, we also have $\dim\frakz_{\frakh}(e)^{\theta_{c}^\rl} = 1$ by \cite[Lemma 2.21]{Thorne-thesis}.
    Therefore $\frakz_{\frakh}(e)^{\theta} = \frakz_{\frakh_0}(e)$ has dimension $1$ and so $e$ is $\theta$-subregular.

    There are a finite number of cases left to consider. When $H$ is of type $E_8$ and $m=3$, a $\theta$-subregular nilpotent is listed as the second entry in \cite[Table 6]{VinbergElasvili-trivectors}.
    For the remaining seven cases, we can check using the \texttt{SLA} package \cite{SLA} in \texttt{GAP} \cite{GAP4}  that they all contain $\theta$-subregular nilpotent elements; see the appendix for more details. 
\end{proof}

We now pay particular attention to those pairs $({}^nH,m)$ for which there are multiple families of curves $X_e\rightarrow B$ listed in Table \ref{table:examples intro}.

\begin{lemma}\label{lemma: subregular curves m2 Bn}
    Theorem \ref{theorem:exhibitionsubregularcurves-algclosed} holds true when $H$ is of type $B_\rl$ $(\rl\geq 3)$ and $m=2$.
\end{lemma}
\begin{proof}
    Let $e\in V(k)$ be a $\theta$-subregular nilpotent element, complete it to a normal $\frak{sl}_2$-triple and let $\psi\colon S_e\rightarrow \frakh\GIT H$ be the associated Slodowy slice.
    Choose decompositions $S_e = U_1 \oplus W$ and $\frakh \GIT H = U_2 \oplus W$ and coordinates $x,y,z$ of $U_1$ and $p_{d_r}$ of $U_2$ satisfying the conclusions of Proposition \ref{proposition: normal form Gamma and sigma}.
    Then $\sigma(x,y,z)$ equals $(x,y,-z)$, $(x,-y,z)$ or $(-x,y,z)$.
    By Theorem \ref{theorem: subregular classification: exclusion} and its proof, the first two cases give rise to the family of curves in the second row of Table \ref{table:examples intro}, while the third case gives rise to the family in the last row.
    Therefore the $\rho$-action on $X_e$ either has weights $2\rl,2\rl,4,8,\dots, 4\rl-4$ (Case I) or $2,2\rl,4,8,\dots, 4\rl-4$ (Case II).
    We need to show that for both cases there exists an $e$ for which $X_e$ has those weights.

    To exhibit Case I, take $e$ to be supported on all but one simple root vector, as discussed in the proof of Proposition \ref{proposition: tabulated gradings are subregular adapted}.
    Let $\frakh = \oplus_{j \in \bZ/2\rl\bZ} \frakh(\theta_c, j)$ be the Coxeter grading of $\frakh$. Then $e \in \frakh(\theta_c, 1)$. 
    If $\sigma_c$ is the action of $\mu_{2\rl}$ on $S_e$ corresponding to the Coxeter grading, then $\sigma_c^\rl = \sigma$. 
    Since we know that $\sigma_c(x,y,z) = (\zeta_{2r}x,y,z)$ by Proposition \ref{proposition: normal form Gamma and sigma}, we must have $\sigma(x,y,z) = (-x,y,z)$, which we calculate corresponds to Case I. 

    We now exhibit an $e$ in Case II, using the results of \cite{Ohta-classificationadmissiblenilpotent}.
    We use the explicit description of the stable $\Z/2\Z$-grading $\theta$ on $\frakh = \frak{so}_{2\rl+1}$ and notation given in Example \ref{ex-Bn-2grading}. 
    We can then describe the nilpotent orbits in terms of the ab-diagrams of \cite[p. 305, Proposition 2]{Ohta-classificationadmissiblenilpotent}.
    The subregular nilpotent orbit in $\mathfrak{so}_{2\rl + 1}$ corresponds to the partition $[2\rl-1,1,1]$ (see \cite[Sections 5.1 and 5.4]{CollingwoodMcGovern-nilpotentorbits} for a proof and an explanation). 
    The ab-diagram $[abab\dots aba,a,b]$ gives rise to the subregular nilpotent element 
\[
e: v_1\mapsto w_1\mapsto v_2 \mapsto w_2 \mapsto \cdots w_{\rl-1}\mapsto v_{\rl}\mapsto 0, w_\rl\mapsto 0, v_{\rl+1}\mapsto 0.
\]
which lies in $\frakh_1$.
Since $m=2$, $e$ is automatically $\theta$-subregular by \cite[Lemma 2.21]{Thorne-thesis}.
Using \cite[Case (BDI) of Section (4.2) and Section (3.1), Lemma 4]{Ohta-classificationadmissiblenilpotent}, $e$ can be completed to a normal $\fraksl_2$-triple $(e,h,f)$, where $h \in \frakh_0$ is a diagonal matrix with entries $(h_1, \dots, h_{\rl+1}, h'_1, \dots, h'_\rl)$, where
\[
h_i = -2\rl+4i-2, \, (1\leq i\leq \rl), \, h_{\rl+1}=0, \, h_i' =-2\rl+4i , \, (1\leq i \leq \rl-1), \,h_\rl' = 0.
\]
Comparing the two cases above, we see that to show we're in Case II, it suffices to show that $2$ is a weight for the $\mathbb{G}_m$-action on $X_e$.
Unwinding the definition of the $\rho$-action, it is enough to show that $0$ is an eigenvalue for the action of $\ad(h)$ on $\frakz_{\frakh_1}(f)$, or equivalently, that $\frakz_{\frakh_1}(h) \cap \frakz_{\frakh_1}(e)$ is nonzero. If we let $A \in \frakh_1$ be the matrix corresponding to the map $A(v_{\rl + 1}) = w_\rl, A(w_\rl) = v_{\rl + 1}$, and $A(v_i) = A(w_j) = 0$ for all $i \neq \rl + 1, j \neq \rl$, then $A$ centralizes both $e$ and $h$, so we are indeed in Case II. 
\end{proof}

\begin{lemma}\label{lemma: subregular curves m2 Cn}
    Theorem \ref{theorem:exhibitionsubregularcurves-algclosed} holds true when $H$ is of type $C_\rl$ and $m=2$.
\end{lemma}
\begin{proof}
    Let $e\in V(k)$ be a $\theta$-subregular nilpotent element, complete it to a normal $\frak{sl}_2$-triple and let $\psi\colon S_e\rightarrow \frakh\GIT H$ be the associated Slodowy slice.
    Choose decompositions $S_e = U_1 \oplus W$ and $\frakh \GIT H = U_2 \oplus W$ and coordinates $x,y,z$ of $U_1$ and $p_{d_r}$ of $U_2$ satisfying the conclusions of Proposition \ref{proposition: normal form Gamma and sigma}.
    Then we either have $\sigma(x,y,z) = (x,y,-z)$ or $(x,-y,z)$.
    By Theorem \ref{theorem: subregular classification: exclusion}, these two cases give rise to the families of curves displayed in Table \ref{table:examples intro}.
    A calculation shows the $\rho$-action on $X_e$ either has weights $4,2\rl-2,4,8,\dots, 4\rl-4$ (Case I) or $4,2\rl,4,8,\dots,4\rl-4$ (Case II).
    We need to show that for both cases there exists an $e$ for which $X_e$ has those weights.
    We use the notation of Example \ref{ex-Cn-2grading}.
Therefore we can use the description of nilpotent orbits in terms of ab-diagrams of \cite[p. 305, Proposition 2]{Ohta-classificationadmissiblenilpotent} to write down subregular nilpotents in $\frakh_1$.
More precisely, the subregular nilpotent orbit in $\mathfrak{sp}_{2\rl}$ corresponds to the partition $[2\rl-2,2]$ (see \cite[Section 5.1 and 5.4]{CollingwoodMcGovern-nilpotentorbits} for a proof and an explanation). 
The ab-diagram $[abab\dots ab,ab]$ gives rise to the subregular nilpotent element 
\[
e: v_1\mapsto w_{\rl-1}\mapsto v_2 \mapsto w_{\rl-2} \mapsto \cdots v_{\rl-1}\mapsto w_{\rl-1}\mapsto 0, v_\rl\mapsto w_\rl
\]
which lies in $\frakh_1$.
Since $m=2$, $e$ is automatically $\theta$-subregular by \cite[Lemma 2.21]{Thorne-thesis}.
Using \cite[Case (CI) of Section (4.2) and Section (3.1), Lemma 4]{Ohta-classificationadmissiblenilpotent}, $e$ can be completed to a normal $\fraksl_2$-triple $(e,h,f)$, where $h \in \frakh_0$ is a diagonal matrix with entries $(h_1, \dots, h_\rl, -h_1, \dots, -h_\rl)$, where
\[
h_i = -2\rl+4i-1 ,\, (1\leq i \leq \rl-1), \quad h_\rl = -1
\]
The action of $\ad(h)$ on $\frakh_1$ has eigenvalues (with multiplicities) equal to $\{\pm (h_i + h_j) \colon 1\leq i\leq j \leq \rl\}$, which equals
\[\{\pm(-4\rl + 4(i+j)-2) \colon 1\leq i\leq j\leq \rl-1  \}
\cup
\{ \pm (-2\rl+4i-2) \colon 1\leq i\leq \rl-1\} \cup \{\pm 2\}.\]
Similarly, the ab-diagram $[abab\dots ab,ba]$ gives rise to the subregular nilpotent element
\[
e': v_1\mapsto w_{\rl-1}\mapsto v_2 \mapsto w_{\rl-2} \mapsto \cdots v_{\rl-1}\mapsto w_{\rl-1}\mapsto 0, w_\rl\mapsto v_\rl
\]
which can be completed to a normal $\fraksl_2$-triple $(e',h',f')$ where $h '  = \text{diag}(h'_1, \dots, h'_\rl, -h_1', \dots, -h_\rl')$ and
\[
h'_i = -2\rl+4i-1 ,\, (1\leq i \leq \rl-1), \quad h'_\rl = 1
\]
We compute that the action of $\ad(h')$ on $\frakh_1$ has eigenvalues (with multiplicities)
\[\{\pm(-4\rl + 4(i+j)-2) \colon 1\leq i\leq j\leq \rl-1  \}
\cup
\{ \pm (-2\rl+4i) \colon 1\leq i\leq \rl-1\} \cup \{\pm 2\}.\]
A combinatorial calculation shows that $(e,h,f)$ gives rise to the weights of Case II and $(e',h',f')$ to those of Case I.
\end{proof}

\begin{lemma}\label{lemma: subregular curves m2 F4 G2}
    Theorem \ref{theorem:exhibitionsubregularcurves-algclosed} holds true when $m=2$ and $H$ is of type $F_4$ or $G_2$.
\end{lemma}
\begin{proof}
    Suppose $H$ is of type $F_4$.
    If $e\in V(k)$ is a $\theta$-subregular nilpotent element, then in the notation of the proof of Theorem \ref{theorem: subregular classification: exclusion}, $\sigma|_{U_1}$ either equals $(-x,y,z)$ or $(x,y,-z)$, and these two cases give rise to the two families of curves listed in Table \ref{table:examples intro}.
    To show that there exists an $e\in V(k)$ for which both cases occur, it suffices to prove there exists two subregular nilpotent elements $e, e'\in V(k)$ such that the multiset of $\rho$-weights on $X_e$ and $X_{e'}$ are not equal. 
    This calculation can be done in \texttt{GAP}; see the appendix for details.
    The case of $G_2$ is similar.
\end{proof}

\begin{proof}[Proof of Theorem \ref{theorem:exhibitionsubregularcurves-algclosed}]
    If $m=2$ and $H$ is of type $B_{\rl}$ $(\rl\geq 3)$, $C_{\rl}$, $F_4$ or $G_2$, we are done by Lemmas \ref{lemma: subregular curves m2 Bn}, \ref{lemma: subregular curves m2 Cn} and \ref{lemma: subregular curves m2 F4 G2}.
    Assume we are not in one of these cases.
    Then there is a unique row of Table \ref{table:examples intro} corresponding to the pair $(H,m)$, in other words there is only one possibility for the family of curves.
    By Proposition \ref{proposition: tabulated gradings are subregular adapted}, there exists a $\theta$-subregular $e \in V(k)$.
    Complete it to a normal $\liesl_2$-triple and let $X_e\rightarrow B$ be the corresponding family of curves.
    By the second part of Theorem \ref{theorem: subregular classification: exclusion}, $X_e\rightarrow B$ must be of the form corresponding to the unique row of Table \ref{table:examples intro} corresponding to $(H,m)$, proving the theorem.
\end{proof}

\begin{proof}[Proof of Theorem \ref{thm-intro}]
    Theorems \ref{theorem: subregular classification: exclusion} and \ref{theorem:exhibitionsubregularcurves-algclosed} imply that Table \ref{table:examples intro} displays exactly all triples $(^nH,m,X_e\rightarrow B)$ arising from $\theta$-subregular nilpotents $e\in V(k)$.
    To complete the proof, we need to explain how the third, fourth and fifth column of the table can be obtained from the pair $(^nH,m)$.
    The tables in \cite[Sections 7.1 and 7.2]{GrossLevyReederYu-GradingsPosRank} display the centralizer of a stable element $v\in V$ and the Kac diagram of $\theta$.
    The determination of the isogeny class of $\tilde{G}$ and $V$ follows from applying the results of \cite{Reeder-torsion} to these Kac diagrams.
\end{proof}

\section{$5$-grading on $E_8$ and $5$-Selmer groups of elliptic curves}\label{section: 5 grading E8}

In this final section, we treat the stable $5$-grading on $E_8$ as an extended example. 
The corresponding Vinberg representation $(G,V)$ is the one used by Bhargava and Shankar in their determination of the average size of the $5$-Selmer group of elliptic curves \cite{BS-5Selmer}.
To achieve this, they rely on a series of papers by Fisher \cite{Fisher-invariantsgenusone, Fisher-invarianttheorynormalquinticI, Fisher-invarianttheoryquinticII, Fisher-minimisationreduction5coverings, FisherSadek-genusonedegree5squarefree} that studies the explicit invariant theory of $(G,V)$.
We show that one can alternatively use the Lie-theoretic methods of this paper to obtain the same orbit parametrization, namely we prove Theorem \ref{theorem:5Selmerintro}.
The strategy is similar to the one for the stable $\Z/3\Z$-grading on $E_8$ employed in \cite[Section 4.4]{Thorne-Romano-E8}.

\subsection{Notation}

We will use the concepts introduced in Section \ref{subsec: stable gradings over algebraically closed field}.
Let $H$ be a split simple group of type $E_8$ over $\Q$.
(It is simultaneously adjoint and simply connected, and unique up to non-unique isomorphism.)
Choose a pinning $(T,P, \{E_{\alpha}\})$, which gives rise to a set of roots $\Phi = \Phi(H,T)$ and a set of simple and positive roots $\Delta\subset \Phi^+\subset \Phi$.
Let $\theta\colon \mu_5\hookrightarrow H$ be the principal $\mu_5$-action with trivial outer type associated to this pinning, see \eqref{equation: definition pinned principal mu_m action}.
Theorem \ref{theorem: characterization stable gradings} and \cite[Section 7.1, Table 5]{GrossLevyReederYu-GradingsPosRank} show that $\theta$ is a stable $\mu_5$-grading.
This gives rise to a $\Z/5\Z$-grading on the Lie algebra $\lieh = \oplus_{i\in \Z/5\Z} \lieh_i$, and as usual we set $G = (H^{\theta})^{\circ}$ and $V = \lieh_1$.
Since $H$ is simply connected, a theorem of Steinberg \cite[Theorem 8.1]{Steinberg-Endomorphismsofalgebraicgroups} shows that $H^{\theta}$ is connected, so $G = H^{\theta}$ in this case.

\begin{lemma}\label{lemma: explicitGV5grading}
    There are isomorphisms of algebraic groups $G\simeq (\SL_5\times\SL_5)/\mu$ and representations $V \simeq (5) \boxtimes \wedge^2(5)$, where $\mu = \{(\zeta, \zeta^2)\colon \zeta\in \mu_5\}$ and $(5)$ denotes the standard representation of $\SL_5$.
\end{lemma}
\begin{proof}
    This follows from the results of \cite{Reeder-torsion} using the Kac diagram of $\theta$ given in \cite[\S7.1, Table 5]{GrossLevyReederYu-GradingsPosRank}.
\end{proof}

Let $B = V\GIT G$ and $B_0 = \lieh \GIT H$.
By Proposition \ref{proposition: invariant polys panyushev}, the inclusion $V\subset \lieh$ induces a closed immersion $B\subset B_0$.
Let $\pi\colon V\rightarrow V\GIT G$ and $p\colon \lieh \rightarrow \lieh \GIT H$ be the quotient maps.

Let $E = \sum_{\alpha \in \Delta} E_{\alpha}\in V(\Q)$ be the regular nilpotent associated to the chosen pinning. By Lemmas \ref{lemma: graded Jacobson-Morozov} and \ref{lemma: regular nilpotent defines top component of N}, $E$ can be completed to a unique normal $\liesl_2$-triple $(E,X,F)$, giving rise to Kostant sections $\kappa_0 = E + \frakz_{\frakh}(F)$ and $\kappa = \kappa_0 \cap V$ of $p$ and $\pi$ respectively (Proposition \ref{proposition: transverse slice regular nilpotent is isomorphism}).

Propositions \ref{proposition: tabulated gradings are subregular adapted} and \ref{proposition: nilpotents on split groups descend} show that there exists a $\theta$-subregular nilpotent element $e\in V(\Q)$.
Complete it to a normal $\liesl_2$-triple $(e,h,f)$.
As in Section \ref{subsec: graded Slodowy slices} (but with slightly different notation), let
\begin{align*}
    S_0 &:= e + \frakz_{\frakh}(f), \\
    S &:= S_0 \cap p^{-1}(B), \\
    C &:= S_0 \cap V = S \cap V.
\end{align*}
We consider the restrictions $\psi_0 \colon S_0\rightarrow B_0$, $\psi\colon S\rightarrow B$ and $\varphi\colon C\rightarrow B$ of $p$ and $\pi$ to the respective affine subspaces.
Let $\lieh^s\subset \lieh$ be the open subset of regular semisimple elements and $V^s = \lieh^s\cap V\subset V$ its restriction to $V$, which equals the open subset of stable elements for the $G$-action by Proposition \ref{proposition: regular semisimple iff stable}.
Let $B_0^s = p(\lieh^s)\subset B$ and $B^s =\pi(V^s)\subset B$.
Consider the discriminant polynomial $\Delta_0\in \Q[B_0]$ and its restriction $\Delta \in \Q[B]$. 

\subsection{Elliptic elements of order $5$ in $W(E_8)$}

Let $(L,(-,-))$ be a root lattice of type $E_8$, i.e., a free abelian group $L$ with a symmetric positive definite pairing $(-,-)\colon L\times L\rightarrow \Z$ such that $(\lambda, \lambda)$ is even for all $\lambda \in \Lambda$ and $\{\alpha\in L\colon (\alpha, \alpha)=2\}$ is a root system of type $E_8$ that generates $L$.
Let $W$ be the Weyl group of $L$, which equals $\Aut((L,(-,-))) = W$ since the Dynkin diagram of $E_8$ has no nontrivial symmetries.
Recall that an element $w\in W$ is elliptic if it has no nonzero fixed points on $L$, equivalently if $L/(1-w)L$ is finite.

\begin{lemma}\label{lemma:propertiesellipticelementorder5E8}
There exists a unique conjugacy class of elliptic elements of order $5$ in $W$. If $w$ is in this conjugacy class, then: 
\begin{enumerate}
    \item $L/(1-w)L\simeq \mathbb{F}_5^2$.
    \item The centralizer $C(w)$ of $w$ in $W$ acts on $L/(1-w)L$ and the induced homomorphism $C(w) \rightarrow \GL(L/(1-w)L)\simeq \GL_2(\mathbb{F}_5)$ has image $\SL_2(\mathbb{F}_5)$ and kernel $\langle w\rangle$. As abstract groups, $C(w) \simeq \SL_2(\mathbb{F}_5) \times \langle w\rangle$.
    \item Every nonzero element of $L/(1-w)L$ is the image of a root in $L$.
\end{enumerate}
\end{lemma}
\begin{proof}
    This is contained in \cite[Table 1 and \S4.2.1]{reeder-ellipticcentralizers}, except the claims made in Part 3, which we prove now.
    Since the roots generate $L$ and $(1-w)L$ is a proper subgroup of $L$, there exists a nonzero element of $L/(1-w)L$ that is the image of a root.
    By Part 2, $C(w)$ acts transitively on the nonzero elements of $L/(1-w)L$ and it preserves the property of being the image of a root. 
    Therefore every nonzero element of $L/(1-w)L$ is the image of a root.
\end{proof}

\begin{lemma}\label{lemma:5gradingstable elements give root lattice}
    Let $k/\Q$ be a field, $v\in V^s(k)$, $A = Z_H(v)$, and $\Lambda = X^*(A_{\bar{k}})$. 
    Then:
    \begin{enumerate}
        \item $\Lambda$ is a root lattice of type $E_8$.
        \item For every fifth root of unity $\zeta\in \bar{k}$, the order $5$ automorphism $\theta(\zeta)$ preserves $A$ and the induced automorphism $w$ of $\Lambda$ is elliptic of order $5$. 
        \item The subgroup $(1-w)\Lambda\subset \Lambda$ is independent of the choice of $\zeta$. Write this subgroup as $(1-\theta)\Lambda$ and let $\Lambda_{\theta} = \Lambda/(1-\theta)\Lambda$.
        \item The projection $\Lambda = X^*(A) \rightarrow X^*(Z_G(v)_{\bar{k}})$ is surjective with kernel $(1-\theta)\Lambda$.
        This kernel is $\Gamma_k$-stable and induces an isomorphism of Galois modules $X^*(Z_G(v)_{\bar{k}})\simeq \Lambda_{\theta}$.
    \end{enumerate}
\end{lemma}
\begin{proof}
    Since $v$ is regular semisimple, $A$ is a maximal torus, so its character group $\Lambda$ is an $E_8$ root lattice, proving Part 1.
    Part 2 follows from Proposition \ref{proposition: regular semisimple iff stable}.
    Part 3 follows from the elementary fact that $(1-w)\Lambda = (1-w^k)\Lambda$ for every integer $k$ coprime to $5$.
    Part 4 follows from the identification $Z_G(v) = A^{\theta}$ and duality of character groups.
\end{proof}

\subsection{Identifying the transverse slices}\label{subsec:5gradingtransverseslice}

In Section \ref{subsec: graded Slodowy slices}, we have defined $\G_m\times \mu_5$-actions $\rho\times \sigma$ on $S_0$ and $B_0$ with the property that $\psi_0\colon S_0\rightarrow B_0$ is $\G_m\times\mu_5$-equivariant, and that the $\sigma$-fixed points of $\psi_0$ equals $\varphi\colon C\rightarrow B$.
We make the morphisms and actions completely explicit in the next proposition.

\begin{proposition}\label{proposition: determination surface slice 5-grading}
    There exist elements $p_2, p_8, p_{12}, p_{14}, p_{18}, p_{20}, p_{24}, p_{30} \in \Q[\frakh]^H$ and $x,y,z\in \Q[S_0]$ with the following properties:
    \begin{enumerate}
        \item Each polynomial $p_i$ is homogeneous of $\rho\times \sigma$-weight $(2i,4i)$. The polynomials $p_i$ are algebraically independent and generate $\Q[\frakh]^H$. The elements $x, y, z$ are homogeneous with $\rho\times \sigma$-weights $(20, 0)$, $(30,0)$ and $(12,1)$, respectively.
        \item The restrictions of the elements $p_2, p_8, p_{12}, p_{14}, p_{18}, p_{20}, p_{24}$ to $S_0$, together with $x, y, z$, form an algebraically independent set and generate $\Q[S_0]$. 
        In these coordinates, the morphism $\psi_0 : S_0 \rightarrow  B_0$ is given by $(p_2, \dots, p_{24},x,y,z)\mapsto (p_2, \dots, p_{24},-F(x,y,z,p_2, \dots, p_{24}))$, where $F+p_{30}$ is given by 
    \begin{align}\label{equation: surfaces 5-grading}
    -y^2+  x^3 + x(p_2z^3 + p_8z^2 + p_{14}z + p_{20}) + (z^5+ p_{12}z^3 + p_{18}z^2 + p_{24}z + p_{30}).
    \end{align}
    \end{enumerate}
\end{proposition}
In other words, $S_0$ can be realized as the closed subscheme of $\A^3_{B_0}$ defined by the zero locus of \eqref{equation: surfaces 5-grading}.
\begin{proof}
    This follows from adapting the methods of Section \ref{subsection: subregular adapted gradings} to a non-algebraically closed field.
    In more detail, $(d\psi_0)_e$ has rank $7$ by \cite[Section 8.3, Proposition 1]{Slodowy-simplesingularitiesalggroups}, so by \cite[Section 8.1, Lemma 2]{Slodowy-simplesingularitiesalggroups}, there exist $\G_m \times \mu_5$-stable decompositions $S_0 = U_1 \oplus W$, $B_0 = U_2 \oplus W$ with $\dim U_1 = 3$, $\dim U_2 = 1$, $\dim W = 7$, an equivariant morphism $F\colon  U_1 \oplus W \rightarrow U_2$ and an equivariant automorphism $\alpha$ of $S_e$ such that $(\psi\circ \alpha)(u,w)=(F(u,w),w)$.
    By Proposition \ref{prop: normal form surface singularity}, $U_1$ has $\rho$-weights $20,30,12$, $U_2$ has $\rho$-weight $60$, and we can choose homogeneous coordinates $x,y,z$ of $(U_1)_{\bar{\Q}}$ and $p_{30}$ of $(U_2)_{\bar{\Q}}$ such that $F(u,0) = f(x,y,z)$ is given by $-y^2 + x^3+ z^5$.
    Since such coordinates are unique up to scaling, there exist coordinates $x,y,z$ of $U_1$ and $p_{30}$ of $U_2$ such that $F(u,0) = f(x,y,z)$ equals $-ay^2 + bx^3+ cz^5$ for some nonzero $a,b,c\in\Q$.
    Since moreoever $2,3$ and $5$ are coprime, we can write $a = u^2/v^{15}$ for some $u,v \in \Q$ and similarly for $b$ and $c$, which shows that we can rescale $x,y,z,p_{30}$ so that $a=b=c=1$.
    Lemma \ref{lemma: sigma action on U1 and U2} shows that in these coordinates, the $\mu_5$-action $\sigma$ must be $\zeta\cdot (x,y,z) = (x,y,\zeta z)$ if $\zeta\in \mu_5$.
    Applying the recipe for the construction of $\psi_{\mathrm{std}}\colon Z\rightarrow D$ to the polynomial $f = -y^2 + x^3 + z^5$ and the $\mu_5$-action $\sigma$ (where we choose the $g_i$ to be monomials) determines a $\G_m\times \mu_5$-equivariant family of surfaces $Z\rightarrow D$. 
    Proposition \ref{proposition: isomorphism between semiversal deformations} shows that there is a $\G_m\times \mu_5$-equivariant isomorphism between $\psi_0\colon S_0\rightarrow B_0$ and $Z\rightarrow D$.
    An explicit computation shows that $Z\rightarrow D$ is exactly the family of surfaces appearing in \eqref{equation: surfaces 5-grading}, as required.
\end{proof}

We fix elements $p_2,\dots,p_{30} \in \Q[\lieh]^H$ and $x,y,z\in \Q[S_0]$ satisfying the conclusions of Proposition \ref{proposition: determination surface slice 5-grading}.
Using these elements, we can think of $S_0\rightarrow B$ as the explicit family of surfaces defined by the vanishing of \eqref{equation: surfaces 5-grading}.
Let $c_i$ denote the restriction of $p_i$ to $\Q[V]^G$.
Since $S_0^{\mu_5} = C$ and $B_0^{\mu_5} = B$, Proposition \ref{proposition: determination surface slice 5-grading} immediately implies that $c_{20},c_{30}$ are algebraically independent, $\Q[V]^G = \Q[c_{20},c_{30}]$ and $C$ is the subvariety of $S$ given by setting $z = 0$; in other words, $C\rightarrow B = \A^2_{(c_{20},c_{30})}$ is the family of affine curves with equation 
\begin{align}\label{equation: curves 5-grading}
    y^2  = x^3 + c_{20}x +c_{30}.
\end{align}
Lemma \ref{lemma: nonzero discriminant implies smooth curve} shows if $b\in B(k)$ then $C_b$ is smooth if $\Delta(b)\neq 0$.
In other words, if $C_b$ is singular then $\Delta(b) = 0$.
Since $4c_{20}^3+27c_{30}^2\in \Q[B]$ is irreducible, it follows by examining degrees that $\Delta = \lambda (4c_{20}^3+27c_{30}^2)^4$ for some nonzero $\lambda \in \Q$.

We now analyze the Picard groups of the fibers of $S_0\rightarrow B_0$.
For each $b = (p_2,\dots,p_{30})\in B_0(k)$, $S_{0,b}$ is an affine surface in $\A^3$ defined by Equation \eqref{equation: surfaces 5-grading}.
The projection map onto the $z$-coordinate defines a morphism $S_{0,b} \rightarrow \A^1$ whose fibers are affine cubic curves.
Therefore we may and do compactify $S_{0,b}$ to a projective surface $\bar{S}_{0,b}$ admitting a Weierstrass fibration $\bar{S}_{0,b}\rightarrow \P^1_k$ in the sense of \cite{Miranda-moduliofWeierstrassfibrations}.
The complement of $S_{0,b}$ in $\bar{S}_{0,b}$ is the union $F \cup P_{\infty}$, where $F$ is the fiber over $z=\infty$ and $P_{\infty}$ is the zero section of the fibration.
This construction works in families, giving a family of projective surfaces $\bar{S}_0\rightarrow B_0$ whose fiber over $b\in B_0(k)$ equals $\bar{S}_{0,b}$.

Now suppose $b\in B^s_0(k)$.
Then $\bar{S}_{0,b}$ is a smooth rational elliptic surface \cite[\S8]{shoidschutt-ellipticsurfaces}.
Therefore $\Pic(\bar{S}_{0,b,\bar{k}})$ is free $\Z$-module of rank $10$ and comes equipped with a unimodular intersection pairing.
The restriction map $r\colon \Pic(\bar{S}_{0,b,\bar{k}}) \rightarrow \Pic(S_{0,b,\bar{k}})$ is surjective with kernel $\langle F, P_{\infty}\rangle$.
Since the restriction of the intersection pairing to $\langle F, P_{\infty}\rangle$ is also unimodular, $r$ restricts to an isomorphism $\langle F, P_{\infty}\rangle^{\perp} \xrightarrow{\sim} \Pic(S_{0,b,\bar{k}})$.
This identifies $\Pic(S_{0,b,\bar{k}})$ with a rank-$8$ submodule of $\Pic(\bar{S}_{0,b,\bar{k}})$; we will use this identification without further mention.

Let $\frakh_b = p^{-1}(b)$.
Since the Kostant section $\kappa_0\rightarrow B_0$ is an isomorphism, there exists a unique element $\kappa_{0,b} \in \frakh_b(k)$ that lies in $\kappa_0(k)$.
Since every two elements in $\frakh_b(\bar{k})$ are $H(\bar{k})$-conjugate, the map $H\rightarrow \lieh_b, h\mapsto \Ad(h)(\kappa_{0,b})$ is a $Z_H(\kappa_{0,b})$-torsor.
Pulling this torsor back along the inclusion $S_{0,b} \hookrightarrow \frakh_b$ determines a $Z_H(\kappa_{0,b})$-torsor $T_b\rightarrow S_{0,b}$.
Pushing out this torsor along a homomorphism $Z_{H}(\kappa_{0,b})_{\bar{k}} \rightarrow \G_{m,\bar{k}}$ determines a $\G_m$-torsor on $S_{0,b,\bar{k}}$, in other words an element of $\Pic(S_{0,b,\bar{k}})$.
This procedure determines a $\Gamma_k$-equivariant homomorphism 
\begin{align*}
    \gamma\colon \Lambda_b = X^*(Z_H(\kappa_{0,b})_{\bar{k}}) \rightarrow \Pic(S_{0,b,\bar{k}}).
\end{align*}

\begin{proposition}\label{proposition:isorootlatticepicardgroup}
    The map $\gamma$ is an isomorphism that intertwines the pairing on $\Lambda_b$ with minus the intersection pairing on $\Pic(S_{0,b,\bar{k}})$.
\end{proposition}
\begin{proof}
    This follows from \cite[Lemma 4.13]{Thorne-Romano-E8}.
    (It is assumed in loc. cit. that the nilpotents $E$ and $e$ lie in the $1$-part of a stable $3$-grading on $\lieh$, but this assumption is not used in the proof.)
\end{proof}

\subsection{Identifying the $5$-torsion}

Let $k/\Q$ be a field and $b\in B^s(k)$.
Since the Kostant section $\kappa\rightarrow B$ is an isomorphism, we may write $\kappa_b$ for the unique element in $\kappa(k)$ above $b$.
Write $V_b\subset V$ for the fiber of $b$ under $\pi\colon V\rightarrow B$.
Since every two elements in $V_b(\bar{k})$ are $G(\bar{k})$-conjugate by Proposition \ref{proposition: basic invariant theory of vinberg representation full generality}, the map $G\rightarrow V_b$ defined by $g\mapsto \Ad(g)(\kappa_b)$ is a torsor under $Z_G(\kappa_b)$.
Pulling back this torsor along the inclusion $C_b\hookrightarrow V_b$ determines a $Z_G(\kappa_b)$-torsor $\mathcal{T}_b\rightarrow C_b$. 
Every homomorphism $Z_G(\kappa_b)_{\bar{k}} \rightarrow \G_{m, \bar{k}}$ determines, via pushout, a $\G_m$-torsor on $C_{b,\bar{k}}$, in other words an element of $\Pic(C_{b,\bar{k}})$. 
The above recipe thus defines a $\Gamma_k$-equivariant map $\alpha\colon X^*(Z_G(\kappa_b)_{\bar{k}}) \rightarrow \Pic(C_{b,\bar{k}})$.
Since the domain of $\alpha$ is $5$-torsion by Lemma \ref{lemma:5gradingstable elements give root lattice}, its image must land in $\Pic(C_{b,\bar{k}})[5]$.

Let $E\rightarrow B$ be the family of projective curves with Weierstrass equation \eqref{equation: curves 5-grading}.
The complement of the section at infinity in $E$ is the family of affine curves $C\rightarrow B$.
If $b \in B^s(k)$, then $E_b$ is an elliptic curve and the inclusion $C_b\subset E_b$ induces an isomorphism $\Pic(C_{b,\bar{k}})\rightarrow \Pic^0(E_{b,\bar{k}})=E_b(\bar{k})$. 
Precomposing this isomorphism with the map $\alpha$, we obtain a map 
\begin{align*}
    \beta\colon X^*(Z_G(\kappa_b)_{\bar{k}}) \rightarrow E_b[5](\bar{k}),
\end{align*}
where $E_b[5]$ denotes the $5$-torsion subscheme of $E_b$.

\begin{theorem}\label{theorem: centralizer isomorphic to 5-torsion}
    The map $\beta$ is an isomorphism of finite Galois modules.
\end{theorem}

The proof of Theorem \ref{theorem: centralizer isomorphic to 5-torsion} is given at the end of this subsection and follows from embedding $C\rightarrow B$ in the family of surfaces $S\rightarrow B$ and using the description of $\Pic(S_{b,\bar{k}})$ of Section \ref{subsec:5gradingtransverseslice}.

Note that $\theta$ preserves the maximal torus $Z_H(\kappa_b)$ and induces a $\mu_5$-action on the character group $\Lambda_b$. 
On the other hand, the assignment $\zeta\cdot (x,y,z) = (x,y,\zeta z)$ defines a $\mu_5$-action on $S_{b,\bar{k}}$ and hence a $\mu_5$-action on $\Pic(S_{b,\bar{k}})$. 
\begin{lemma}\label{lemma: mu5 actions coincide}
    If $b\in B^s(k)$, then the $\mu_5$-actions just constructed agree under the isomorphism $\gamma$ of Proposition \ref{proposition:isorootlatticepicardgroup}.
\end{lemma}
\begin{proof}
    The formula $\zeta \cdot x = \zeta^{-1} \theta(\zeta)(x)$ defines the $\mu_5$-action $\sigma$ on $\frakh_b$ (and hence on $S_b$ via restriction) and $\theta$ defines a $\mu_5$-action on $H$.
    The torsor $H\rightarrow \frakh_b$, and hence the torsor $\mathcal{T}'_b\rightarrow S_b$, is equivariant with respect to these actions. 
    It follows that the homomorphism $\gamma$ is $\mu_5$-equivariant too. 
    Proposition \ref{proposition: determination surface slice 5-grading} displays the $\sigma$-weights on $S_e$, which shows that $\sigma$ acts as $(x,y,z)\mapsto (x,y,\zeta z)$ on $S_{b, \bar{k}}$, as claimed.
\end{proof}

Let $b\in B^s(k)$ and consider the composite $\Lambda_b\xrightarrow{\gamma} \Pic(S_{b,\bar{k}})\rightarrow \Pic(C_{b,\bar{k}})$, where the second map is given by restriction along the closed subscheme $C_b \subset S_b$.
Lemma \ref{lemma: mu5 actions coincide} and the equality $C_b = S_b^{\mu_5}$ imply that this composite map commutes with the $\mu_5$-action on the domain.
Let $\zeta\in \mu_5(\bar{k})$ be a choice of primitive $5$th root of unity, and let $w$ be the induced order-$5$ automorphism on $\Lambda_b$.
By Lemma \ref{lemma:5gradingstable elements give root lattice},  $w$ is an elliptic element of order $5$ of the $E_8$ root lattice $\Lambda_b$. 
We thus obtain a map $\Lambda_b/(1-w)\Lambda_b\rightarrow \Pic(C_{b,\bar{k}})$ which must land in the $5$-torsion by Lemma \ref{lemma:propertiesellipticelementorder5E8}.
In conclusion, we have constructed a Galois equivariant map 
\begin{align*}
    \delta\colon \Lambda_b/(1-w)\Lambda_b\rightarrow \Pic(C_{b,\bar{k}})[5].
\end{align*}

\begin{proposition}\label{proposition: coinvariants isomorphic to 5-torsion}
    The map $\delta$ is an isomorphism.
\end{proposition}
\begin{proof}
    We may and do assume $k$ is algebraically closed.
    Let $O_{\infty,b}$ be the origin of $E_b$, so $C_b = E_b \setminus \{O_{\infty,b}\}$ and $\Pic(C_b) \simeq E_b(k)$.
    Since the domain and target of $\delta$ both have order $25$, it suffices to prove that $\delta$ has trivial kernel.
    Since every nonzero class in $\Lambda_b/(1-w)\Lambda_b$ is represented by a root $\alpha\in \Lambda_b$ (Lemma \ref{lemma:propertiesellipticelementorder5E8}), it suffices to prove that $\delta(\alpha)\neq 0$ for every root $\alpha$ of $\Lambda_b$.
    Such a root corresponds, under the isomorphism $\gamma$ of Proposition \ref{proposition:isorootlatticepicardgroup}, to an element $D_{\alpha}\in \Pic(S_{b})\subset \Pic(\bar{S}_b)$ with self-intersection $-2$.
    Since $f\colon \bar{S}_b\rightarrow \P^1$ has no reducible fibers, \cite[Lemma 3.1]{Shioda-grobnerbasisMordellWeil} implies that $D_{\alpha}$ is represented by a divisor on $\bar{S}_b$ of the form $(P) - (P_{\infty})-(F)$, where $P$ is a section of $f$ that does not intersect the section at infinity $P_{\infty}$.
    Consequently, $\delta(\alpha)$ is represented by the point $(P\cap \{z=0\})$ on $C_b$. 
    Since $P$ does not intersect $P_{\infty}$, this point is distinct from the point at infinity, so is nonzero, as desired.
\end{proof}

\begin{proof}[Proof of Theorem \ref{theorem: centralizer isomorphic to 5-torsion}]
    We may and do assume $k$ is algebraically closed.
    The maps constructed so far fit in a commutative diagram:
\[\begin{tikzcd}
	{\Lambda_b} & {\Pic(S_b)} \\
	{X^*(Z_G(\kappa_b))} & {\Pic(C_b)[5]}
	\arrow["\gamma", from=1-1, to=1-2]
	\arrow[from=1-1, to=2-1]
	\arrow[from=1-2, to=2-2]
	\arrow["\beta", from=2-1, to=2-2]
\end{tikzcd}\]
    The top horizontal map is an isomorphism by Proposition \ref{proposition:isorootlatticepicardgroup}.
    The left vertical map is surjective with kernel $(1-w)\Lambda_b$ by Lemma \ref{lemma:5gradingstable elements give root lattice}.
    The composition $\Lambda_b\rightarrow \Pic(C_b)[5]$ is also surjective with kernel $(1-w)\Lambda_b$, by Proposition \ref{proposition: coinvariants isomorphic to 5-torsion}.
    This implies $\beta$ is an isomorphism.
\end{proof}

\begin{theorem}\label{theorem: global iso centralizer 5-torsion}
    Let $E\rightarrow B$ be the compactification of $C\rightarrow B$ and let $E^s\rightarrow B^s$ be the restriction to $B^s$, a family of elliptic curves.
    Then there is an isomorphism $Z_G(\kappa|_{B^s})\simeq E^s[5]$ of finite etale group schemes over $B^s$ intertwining the respective pairings.
\end{theorem}
\begin{proof}
    By \cite[Tag \href{https://stacks.math.columbia.edu/tag/0BQM}{0BQM}]{stacksproject}, it suffices to prove that there is an isomorphism $Z_G(\kappa_{\eta})_{\bar{k}} \simeq E_{\eta}[5](\bar{k})$ of Galois modules, where $k$ is the function field of $B$ and $\eta\in B(k)$ the generic point.
    Theorem \ref{theorem: centralizer isomorphic to 5-torsion} shows that there is an isomorphism $X^*(Z_{G}(\kappa_{\eta})_{\bar{k}})\simeq E_{\eta}[5](\bar{k})$, in other words an isomorphism $Z_G(\kappa_{\eta})_{\bar{k}}\simeq \Hom(E_{\eta}[5]_{\bar{k}},\mu_{5,\bar{k}})$.
    The existence of the Weil pairing on $E_{\eta}[5]$ shows that there is an isomorphism $\Hom(E_{\eta}[5]_{\bar{k}},\mu_{5,\bar{k}}) \simeq E_{\eta}[5](\bar{k})$.
    We conclude that $Z_G(\kappa_{\eta})_{\bar{k}} \simeq E_{\eta}[5](\bar{k})$.
\end{proof}

\subsection{The orbit parametrization}

Let $k/\Q$ be a field and $b\in B^s(k)$.
Then every element in $V_b(\bar{k})$ is stable and $G(\bar{k})$-conjugate to $\kappa_b$ by Proposition \ref{proposition: basic invariant theory of vinberg representation full generality}(2).
However, $V_b(k)$ might break up into multiple $G(k)$-orbits.
If $v\in V_b(k)$, then the transporter scheme $T(v,\kappa_b) = \{g\in G(\bar{k}) \colon g\cdot v = \kappa_b\}$ is a torsor under the stabilizer $Z_G(\kappa_b)$. 
The following lemma is a special case of \cite[Proposition 1]{BhargavaGross-AIT}:

\begin{lemma}\label{lemma: AIT bhargava gross}
    The map $v\mapsto T(v, \kappa_b)$ induces a bijection 
    \begin{align*}
    G(k) \backslash V_b(k) \xleftrightarrow{1:1} \ker(\HH^1(k,Z_{G}(\kappa_b))\rightarrow \HH^1(k,G)).
    \end{align*}
    (Here $\HH^1(k,-)$ denotes Galois cohomology.)
\end{lemma}

We call an isomorphism of group schemes $Z_G(\kappa|_{B^s})\simeq E^s[5]$ over $B^s$ a ``universal isomorphism.'' 
Such an isomorphism exists by Theorem \ref{theorem: global iso centralizer 5-torsion}.
It induces an isomorphism $Z_G(\kappa_b)\simeq E_b[5]$ of Galois modules and hence, by combining it with the previous lemma, a map $V_b(k)\rightarrow \HH^1(k, E_b[5])$.
The affine curve $C_b\subset E_b$ is defined as a closed subscheme of $V_b$, so we have a tautological inclusion $\iota\colon C_b(k) \rightarrow V_b(k)$.
Composing these two maps defines a map $C_b(k) \rightarrow \HH^1(k,E_b[5])$. 

On the other hand, since $E_b$ is an elliptic curve, there is a $5$-descent map $E_b(k) \xrightarrow{\delta} \HH^1(k, E_b[5])$, associating to every $P\in E_b(k)$ to isomorphism class of the $E_b[5]$-torsor $\{Q\in E_b(\bar{k})\colon 5Q = P\}$.

\begin{theorem}\label{theorem: compatibility descent map orbit construction}
    There exists a universal isomorphism such that for every field $k/\Q$ and every $b\in B^s(k)$, the following diagram commutes and is functorial in $k$:
\[\begin{tikzcd}
	{C_b(k)} & {V_b(k)} \\
	{E_b(k)/5E_b(k)} & {\HH^1(k,E_b[5])}
	\arrow["\iota", hook, from=1-1, to=1-2]
	\arrow[from=1-1, to=2-1]
	\arrow[from=1-2, to=2-2]
	\arrow["\delta", from=2-1, to=2-2]
\end{tikzcd}\]
\end{theorem}
\begin{proof}
    The proof is identical to \cite[Theorem 4.15]{Thorne-thesis} and is omitted.
\end{proof}

We henceforth fix a universal isomorphism satisfying the conclusion of Theorem \ref{theorem: compatibility descent map orbit construction}. 
We immediately obtain the following orbit parametrization:
\begin{corollary}\label{corollary: 5selmer first orbit parametrization}
    For every field $k/\Q$ and $b\in B^s(k)$, there is a unique injection 
    \begin{align}\label{equation: orbit parametrization 5-grading}
        \eta_b\colon E_b(k)/5E_b(k) \hookrightarrow G(k)\backslash V_b(k),
    \end{align}
    functorial in $k$, with the following properties:
    \begin{enumerate}
        \item $\eta_b$ maps $0 \in E_b(k)/5E_b(k)$ to the orbit $G(k)\cdot \kappa_b$ of the Kostant section.
        \item If $P\in C_b(k)$, then $\eta_b(P \mod 5)$ lies in the orbit of $P$, seen as an element of $C_b(k)\subset V_b(k)$.
    \end{enumerate}
\end{corollary}
\begin{proof}
    Since the two properties guarantee uniqueness, it suffices to prove the existence of $\eta_b$.
    Consider the composition 
    \[f\colon E_b(k)/5 E_b(k) \xrightarrow{\delta} \HH^1(k, E_b[5])\xrightarrow{\sim} \HH^1(k, Z_G(\kappa_b))\rightarrow \HH^1(k,G).\] 
    We claim that $f(P)$ is trivial for every $P\in E_b(k)/5E_b(k)$.
    If $P$ is itself trivial, this is clear, since $\delta(P)$ is trivial.
    If $P\neq 0$, then $P$ lies in the image of the map $C_b(k)\rightarrow E_b(k)$, and so the triviality follows from Theorem \ref{theorem: compatibility descent map orbit construction} and Lemma \ref{lemma: AIT bhargava gross}.

    It follows that for every $P\in E_b(k)/5E_b(k)$, there exists a unique element $\eta_b(P)$ of $G(k) \setminus V_b(k)$ that corresponds to $f(P)$ under Lemma \ref{lemma: AIT bhargava gross}.
    This map satisfies the conclusions of the corollary by construction.
\end{proof}

\begin{corollary}
    For every $b\in B^s(\Q)$, the injection \eqref{equation: orbit parametrization 5-grading} extends to an injection $\Sel_5 E_b \rightarrow G(\Q) \backslash V_b(\Q)$.
\end{corollary}
\begin{proof}
    Similarly to the proof of Corollary \ref{corollary: 5selmer first orbit parametrization}, we need to show that the composite map $\Sel_5E_b\rightarrow \HH^1(\Q, E_b[5])\simeq \HH^1(\Q, Z_G(\kappa_b))\rightarrow \HH^1(\Q,G)$ is trivial.
    Let $c\in \HH^1(\Q, G)$ be an element in the image of this map.
    By definition of the $5$-Selmer group and the previous corollary, we know that the restriction $c_v\in \HH^1(\Q_v, G)$ is trivial, for every place $v$ of $\Q$.
    Since $G$ is split semisimple by Lemma \ref{lemma: explicitGV5grading}, a Hasse principle for $\HH^1(-,G)$ implies that $c=1$, see \cite[Proposition 6.8]{Laga-ADEpaper}.
\end{proof}

\subsection{Integral representatives}

We now show that when $k=\Q$, the image of the map \eqref{equation: orbit parametrization 5-grading} has integral representatives away from small primes, thereby completing the proof of Theorem \ref{theorem:5Selmerintro}.

We first need to fix integral structures on the objects involved.
Using Lemma \ref{lemma: explicitGV5grading}, fix isomorphisms $G\simeq (\SL_5\times \SL_5)/\mu$ and $V\simeq (5) \boxtimes \wedge^2(5)$.
Since $(\SL_5\times \SL_5)/\mu$ and $(5) \boxtimes \wedge^2(5)$ are defined over $\Z$, we may view $G$ as a group scheme over $\Z$ and $V$ as a $\Z$-module so that these isomorphisms are defined over $\Z$.
We redefine $B = \Spec(\Z[c_{20},c_{30}])$.
We redefine $C = \Spec(\Z[x,y,c_{20}])$ and $\psi\colon C\rightarrow B, (x,y,c_{20})\mapsto (c_{20}, y^2 - x^3-c_{20}x)$.
We let $E$ be the closed subscheme of $\P^2_{\Z} \times_{\Z} B$ defined by the same Weierstrass equation \eqref{equation: curves 5-grading}.

\begin{proposition}\label{proposition: local integral representatives}
    There exists an integer $N\geq 1$ with the following property:
    for every prime $p$ and $b = (c_{20}, c_{30})\in \Z_p^2$ with $\Delta(b)\neq 0$, every element in the image of the map $\eta_b\colon E_b(\Q_p) \rightarrow G(\Q_p)\backslash V_b(\Q_p)$
    has a representative in $\frac{1}{N}V(\Z_p)$.
\end{proposition}

\begin{proof}
    The inclusion $\iota\colon C_{\Q} \rightarrow V_{\Q}$ is an algebraic map of affine spaces over $\Q$.
    By spreading out (in other words, examining denominators), there exists an integer $N_1\geq 1$ such that $N_1\cdot \iota(x) \in V(\Z_p)$ for all $p$ and elements $(x,y,c_{20})\in C(\Q_p)$ that satisfy $5^2x,5^3y,c_{20} \in \Z_p$.
    Similarly, there exists an integer $N_2\geq 1$ such that $N_2\cdot \kappa_b \in V(\Z_p)$ for all $p$ and $b\in B(\Z_p)$.
    We claim that $N = N_1N_2$ satisfies the properties of the proposition.

    Indeed, fix a prime $p$ and a class $\alpha\in E_b(\Q_p)/5E_b(\Q_p)$. It follows from formal group considerations that:
    \begin{itemize}
        \item If $p\neq 5$, then either $\alpha=0$ or $\alpha$ is represented by a point $(x,y)\in C_b(\Z_p)$.
        \item If $p=5$, then either $\alpha=0$ or $\alpha$ is represented by a point $(x,y) \in C_b(\Q_p)$ with $(5^2x,5^3y) \in C_b(\Z_p)$.
    \end{itemize}
    If $\alpha=0$, then $\kappa_b$ is a representative of $\eta_b(\alpha)$ that lies in $\frac{1}{N_2} V(\Z_p)\subset \frac{1}{N} V(\Z_p)$.
    If $\alpha\neq 0$, it is represented by a point $P = (x,y) \in C_b(\Q_p)$ with $(5^2x,5^3y)\in C_b(\Z_p)$.
    Then $\iota(P)$ is a representative of $\eta_b(P)$ (by Corollary \ref{corollary: 5selmer first orbit parametrization}) which lies in $\frac{1}{N_1}V(\Z_p)\subset \frac{1}{N}V(\Z_p)$.
\end{proof}

\begin{remark}
    We do not obtain exact integral representatives but only up to a uniformly bounded integer $N$; this is usually enough for the purposes of determining the average size of $\Sel_5 E_b$.
    Also note that it would be possible to analyze the squarefree discriminant case and obtain \cite[Proposition 11]{BS-5Selmer} using Lie theory (in a manner entirely analogous to \cite[Section 7.4]{Laga-ADEpaper}), but we have chosen to not pursue this.
\end{remark}

\begin{proof}[Proof of Theorem \ref{theorem:5Selmerintro}]
Let $N\geq 1$ be an integer satisfying the conclusion of Proposition \ref{proposition: local integral representatives}.
Let $v\in V_b(\Q)$ be a representative for an orbit in the image of $\eta_b$.
Then there is a finite set of primes $S$ such that $v\in V(\Z_p)$ for all $p\not\in S$.
By Proposition \ref{proposition: local integral representatives}, for each $p\in S$ there is an element $g_p\in G(\Q_p)$ such that $g_p\cdot v \in \frac{1}{N} V(\Z_p)$. 

Lemma \ref{lemma: explicitGV5grading} and the choice of integral structure on $G$ shows that $G$ is a split semisimple group over $\Z$, in other words a Chevalley group.
Such a group has class number $1$. 
Therefore there exists an element $g\in G(\Q)$ such that $g g_p^{-1} \in G(\Z_p)$ for all $p\in S$ and $g \in G(\Z_p)$ for all $p\not\in S$.
By construction, $g\cdot v\in \frac{1}{N}V(\Z_p)$ for all primes $p$, so $g\cdot v \in \frac{1}{N} V(\Z)$, as desired.
\end{proof}

\appendix

\section{Computing nilpotent orbits in \texttt{GAP}}

We have used de Graaf's \texttt{SLA} package \cite{SLA} in the open-source computer algebra system \texttt{GAP} \cite{GAP4} to compute nilpotent orbits in Vinberg representations.
Our (small amount of) code can accessed via the following GitHub repository:
\begin{center}
    \url{https://github.com/jeflaga/slodowy-slices-vinberg-reps}
\end{center}
Loading the file "orbitsandweights.g" in the repository allows one to run the examples below. 
The repository also contains files that explicitly verify the claims made in Proposition \ref{proposition: tabulated gradings are subregular adapted} and Lemma \ref{lemma: subregular curves m2 F4 G2}.

The \texttt{SLA} package can generate a list of all the conjugacy classes of gradings on a simple complex Lie algebra of a given order, together with their Kac diagrams.
The classification of stable gradings of \cite[Section 7]{GrossLevyReederYu-GradingsPosRank} can then be used to determine which grading in the list is stable.

We consider $\mu_2$-actions on $F_4$ as an example.
Executing the commands
\begin{lstlisting}[language=bash]
LoadPackage("sla");
f:=FiniteOrderInnerAutomorphisms("F",4,2);;
 for i in [1..Length(f)] do;
 Print(i," ",KacDiagram(f[i]),"\n");
 od;
\end{lstlisting}
 shows that the stable grading is $f[2]$, the second element of the list. 
 This corresponds to a $\mu_2$-action on the adjoint group $H$ of type $F_4$ over $\mathbb{C}$, giving rise to the Vinberg representation $(G,V)$ with $\Lie(G) = \frakh_0$ and $V = \frakh_1$.
 Executing the commands
 \begin{lstlisting}[language=bash]
 theta:=f[2];;
 smallorbitdata(theta);;\end{lstlisting}
 generates the following output:
 \begin{lstlisting}[language=bash]
 1. Dynkin labelling: [ 2, 2, 2, 2 ]
 Relative dimension: 0
 Weight [ 4, 0 ] with multiplicity 1
 Weight [ 12, 0 ] with multiplicity 1
 Weight [ 16, 0 ] with multiplicity 1
 Weight [ 24, 0 ] with multiplicity 1

 2. Dynkin labelling: [ 2, 2, 0, 2 ]
 Relative dimension: 1
 Weight [ 6, 1 ] with multiplicity 1
 Weight [ 4, 0 ] with multiplicity 1
 Weight [ 8, 0 ] with multiplicity 1
 Weight [ 12, 0 ] with multiplicity 2
 Weight [ 16, 0 ] with multiplicity 1

 3. Dynkin labelling: [ 2, 2, 0, 2 ]
 Relative dimension: 1
 Weight [ 12, 1 ] with multiplicity 1
 Weight [ 4, 0 ] with multiplicity 1
 Weight [ 6, 0 ] with multiplicity 1
 Weight [ 8, 0 ] with multiplicity 1
 Weight [ 12, 0 ] with multiplicity 1
 Weight [ 16, 0 ] with multiplicity 1

 Closure relations: [ [ 2, 1 ], [ 3, 1 ] ]\end{lstlisting}
 This means that there are three $G$-orbits of nilpotent elements $e\in V$ with $\dim \frakz_{\frakh_0}(e)\leq 1$.
 The Dynkin labeling of a nilpotent orbit is a labeling of the Dynkin diagram of $H$ and classifies the corresponding $H$-orbit; see \cite[Section 7.3]{Slodowy-simplesingularitiesalggroups}.
 These Dynkin labelings show that the first listed nilpotent orbit $\calO_1$ is regular and has maximal dimension.
 The two other nilpotent orbits $\mathcal{O}_2, \mathcal{O}_3$ are subregular and have codimension $1$ in the nilpotent cone. 
  The closure relation is the obvious one: $\mathcal{O}_2\subset \bar{\calO_1}$ and $\mathcal{O}_3\subset \bar{\calO}_1$.
 The weights are denoted as pairs $(a,b)$ with $a\in \Z$ and $b\in \Z/2\Z$.

In the GitHub repository we have listed many stable gradings on exceptional Lie algebras. 
For example, the stable $5$-grading on $E_8$ and information about its codimension-$1$ orbits can be accessed as follows:
\begin{lstlisting}[language=bash]
theta:=FiniteOrderInnerAutomorphisms("E",8,5)[6];;
data:=smallorbitdata(theta);;
\end{lstlisting}



Examining the output shows that there exists a subregular nilpotent orbit of codimension $1$ in $V$ and the weights on the Slodowy slice $S$ and on the base $\lieh \GIT H$ are the ones recorded in Proposition \ref{proposition: determination surface slice 5-grading}.

As a final example, we show how our code can be used to explicitly compute families of curves in non-subregular-adapted gradings. We demonstrate this process using the stable $8$-grading of $F_4$. We input

\begin{lstlisting}[language=bash]
theta:=FiniteOrderInnerAutomorphisms("F",4,8)[20];;
data:=smallorbitdata(theta);;
\end{lstlisting}

The output shows that there are four nilpotent orbits of codimension zero in $V$: one regular orbit, two subregular orbits, and one non-reduced orbit. There are seven orbits of codimension one. Examining the closure relations, we see that four of these are good. Given a representative $e$ of one of these good orbits, the $\bG_m$-weights are enough to determine the map $\varphi: X_e \to B$. As an example, consider the following orbit, which is labeled 5 in the output:

\begin{lstlisting}[language=bash]
5. Dynkin labeling: [ 0, 2, 0, 2 ]
Relative dimension: 1
Weight [ 2, 1 ] with multiplicity 1
Weight [ 8, 0 ] with multiplicity 2
Weight [ 2, -1 ] with multiplicity 1
Weight [ 8, -2 ] with multiplicity 1
Weight [ 12, -2 ] with multiplicity 1
Weight [ 8, -4 ] with multiplicity 1
Weight [ 2, -5 ] with multiplicity 1
Weight [ 4, -6 ] with multiplicity 1
Weight [ 8, -6 ] with multiplicity 1
\end{lstlisting}

We see that $\bG_m$ acts on $X_e$ with weight $8$ with multiplicity 2. The affine space $B$ is one dimensional, and $\bG_m$ acts on it with weight 16. We may choose coordinates $x, y$ on $X_e$ such that $\bG_m$ acts on both $x$ and $y$ with weight $8$. To be $\bG_m$-equivariant, we see that the map $\varphi: X_e \to V\GIT G$ must be of the form $(x, y) \mapsto ax^2 + bxy + cy^2$ for some constants $a, b, c$. Since $\varphi$ is surjective, we cannot have $a = b = c = 0$, and after a linear change of variables we may assume $b = 0$. Since $\varphi^{-1}(0)$ is reduced, we must have $ab \neq 0$. After scaling $x$ and $y$, we see that $\varphi$ is of the form $(x, y) \mapsto x^2 + y^2$, and the fibers of $\varphi$ are of the form $x^2 + y^2 = p_8$.  

A similar analysis of the orbit labeled 10 in the output

\begin{lstlisting}[language=bash]
10. Dynkin labeling: [ 0, 0, 0, 2 ]
Relative dimension: 1
Weight [ 6, 1 ] with multiplicity 1
Weight [ 4, 0 ] with multiplicity 1
Weight [ 8, 0 ] with multiplicity 1
Weight [ 6, -1 ] with multiplicity 1
Weight [ 4, -2 ] with multiplicity 2
Weight [ 6, -3 ] with multiplicity 1
Weight [ 4, -4 ] with multiplicity 1
Weight [ 8, -4 ] with multiplicity 1
Weight [ 6, -5 ] with multiplicity 1
Weight [ 4, -6 ] with multiplicity 2
\end{lstlisting}

shows that if $e$ is in this orbit, then the fibers of $\varphi: X_e \to B$ are of the form $x^4 + y^2 = p_8$. In low-rank gradings, this kind of analysis is often enough to determine curves in examples, but as the rank increases the computations become more tedious.

\bibliographystyle{alpha}

\end{document}